\documentclass[twoside,12pt]{amsart}
\usepackage{amscd, bbm, amsaddr}
\usepackage{amssymb,mathabx}
\usepackage{mathrsfs,wasysym, tikz-cd, extpfeil}
\usetikzlibrary{matrix, arrows}
\usepackage{amsmath}
\usepackage[margin=2cm]{geometry}

\usepackage{color}
\usepackage{amscd}
\usepackage[pdftex,colorlinks,hypertexnames=false,linkcolor=black,urlcolor=black,citecolor=black]{hyperref}

\makeatletter
\@addtoreset{equation}{section}

\usepackage{setspace}
\numberwithin{equation}{section}

\newtheorem{Theorem}{Theorem}[section]
\newtheorem*{Theorem*}{Theorem}
\newtheorem*{Corollary*}{Corollary}

\newtheorem{Lemma}[Theorem]{Lemma}
\newtheorem{Proposition}[Theorem]{Proposition}
\newtheorem{Corollary}[Theorem]{Corollary}

\theoremstyle{definition}

\theoremstyle{remark}
\newtheorem{Remark}[Theorem]{Remark}
\newtheorem*{Remark*}{Remark}
\newtheorem*{proofoffirstmain}{Proof of Theorem~\ref{T:firstmain}}
\newtheorem*{proofoforbitmethod}{Proof of Theorem~\ref{T:Losevrefined}}
\newtheorem*{proofofsemiclassicalcomponentmap}{Proof of Theorem~\ref{T:semiclassicalcomponents}}
\newtheorem{Example}[Theorem]{Example}

\newbox\squ  
\setbox\squ=\hbox{\vrule width.6pt
   \vbox{\hrule height.6pt width.4em\kern1ex
         \hrule height.6pt}%
   \vrule width.6pt}

\newcommand{\C}{\mathbb{C}}
\newcommand{\Q}{\mathbb{Q}}
\newcommand{\Nbb}{\mathbb{N}}
\newcommand{\Z}{\mathbb{Z}}
\renewcommand{\k}{\mathbbm{k}}
\renewcommand{\Z}{\mathbb{Z}}

\newcommand{\g}{\mathfrak{g}}
\newcommand{\gl}{\mathfrak{gl}}
\renewcommand{\sl}{\mathfrak{sl}}
\newcommand{\so}{\mathfrak{so}}
\renewcommand{\sp}{\mathfrak{sp}}
\newcommand{\p}{\mathfrak{p}}
\newcommand{\cc}{\mathfrak{c}}
\newcommand{\n}{\mathfrak{n}}
\renewcommand{\r}{\mathfrak{r}}
\newcommand{\h}{\mathfrak{h}}

\renewcommand{\u}{\mathfrak{u}}
\newcommand{\af}{\mathfrak{a}}

\newcommand{\m}{\mathfrak{m}}
\renewcommand{\l}{\mathfrak{l}}
\newcommand{\z}{\mathfrak{z}}
\renewcommand{\v}{\mathfrak{v}}
\newcommand{\mi}{\mathfrak{i}}

\newcommand{\Pbb}{\mathfrak{P}}

\renewcommand{\P}{\mathcal{P}}
\renewcommand{\O}{\mathcal{O}}
\newcommand{\N}{\mathcal{N}}

\newcommand{\E}{\mathcal{E}}
\newcommand{\Fc}{\mathcal{F}}

\newcommand{\A}{\mathcal{A}}

\newcommand{\D}{\mathcal{D}}
\newcommand{\Sc}{\mathcal{S}}

\newcommand{\X}{\mathcal{X}}
\newcommand{\VA}{\mathcal{V\! A}}
\newcommand{\J}{\mathcal{J}}

\renewcommand{\L}{\mathcal{L}}

\newcommand{\ad}{\operatorname{ad}}
\newcommand{\Ad}{\operatorname{Ad}}

\newcommand{\HC}{\operatorname{HC}}
\newcommand{\Lie}{\operatorname{Lie}}
\newcommand{\Ker}{\operatorname{Ker}}
\renewcommand{\ker}{\operatorname{Ker}}
\newcommand{\Char}{\operatorname{char}}

\newcommand{\reg}{{\operatorname{reg}}}

\newcommand{\gr}{\operatorname{gr}}
\newcommand{\Spec}{\operatorname{Spec}}
\newcommand{\End}{\operatorname{End}}

\newcommand{\Mat}{\operatorname{Mat}}
\newcommand{\red}{{\operatorname{red}}}

\newcommand{\Aut}{\operatorname{Aut}}

\newcommand{\Ind}{\operatorname{Ind}}
\newcommand{\an}{{\operatorname{an}}}

\newcommand{\Comp}{\operatorname{Comp}}

\newcommand{\mSpec}{\operatorname{Max}}
\newcommand{\lmod}{\operatorname{-mod}}
\newcommand{\Ann}{\operatorname{Ann}}
\newcommand{\Rad}{\operatorname{Rad}}
\newcommand{\Prim}{\operatorname{Prim}}

\newcommand{\ab}{{\operatorname{ab}}}
\newcommand{\GL}{\operatorname{GL}}

\newcommand{\Tr}{\operatorname{Tr}}
\newcommand{\TC}{\operatorname{TC}}
\newcommand{\T}{\operatorname{T}}
\newcommand{\F}{\operatorname{F}}
\newcommand{\fd}{\operatorname{f.d.}}

\newcommand{\ve}{\varepsilon}

\newcommand{\PeN}{\mathcal{P}_\ve(N)}

\newcommand{\wt}{\widetilde}
\renewcommand{\h}{\widehat}
\renewcommand{\c}{{\!\wedge}}
\newcommand{\he}{\widehat{e}}
\newcommand{\hcc}{\widehat{\cc}}
\newcommand{\hB}{\widehat{B}}
\newcommand{\ce}{\widecheck{e}}
\newcommand{\cy}{\widecheck{y}}
\newcommand{\td}{\widetilde{d}}

\newcommand{\teta}{\widetilde{\eta}}

\newcommand{\otheta}{\overline{\theta}}
\newcommand{\oeta}{\bar{\eta}}
\renewcommand{\i}{{\bf i}}
\renewcommand{\j}{{\bf j}}

\newcommand{\tg}{\tilde{\g}}

\newcommand{\ttheta}{\dot{\theta}}
\newcommand{\od}{\dot{d}}

\newcommand{\lez}{\operatorname{\!<\!0}}
\newcommand{\geez}{\operatorname{\ge\!0}}
\newcommand{\lemw}{\operatorname{\!<\!--1}}

\newcommand{\isoto}{\overset{\sim}{\longrightarrow}}
\newcommand{\onto}{\twoheadrightarrow}
\newcommand{\into}{\hookrightarrow}
\newcommand{\onetoone}{\overset{1-1}{\longrightarrow}}
\DeclareRobustCommand\longtwoheadrightarrow
     {\relbar\joinrel\twoheadrightarrow}
\mathchardef\mh="2D
\newcommand{\longhookrightarrow}{\lhook\joinrel\longrightarrow}

\title[]{\boldmath One dimensional representations of finite $W$-algebras,\\ Dirac reduction and the orbit method}

\author{Lewis Topley}

\thanks{\nonumber{\it Mathematics Subject Classification} (2000 {\it revision}).
Primary 17B35, 17B63. Secondary 17B20, 17B10.}

\begin{document}

\maketitle
\begin{center}
{\it Dedicated to my teacher Sasha Premet, with admiration and gratitude.}
\end{center}


\begin{abstract}
In this paper we study the variety of one dimensional representations of a finite $W$-algebra attached to a classical Lie algebra, giving a precise description of the dimensions of the irreducible components. We apply this to prove a conjecture of Losev describing the image of his orbit method map. In order to do so we first establish new Yangian-type presentations of semiclassical limits of the $W$-algebras attached to distinguished nilpotent elements in classical Lie algebras, using Dirac reduction.
\end{abstract}

\newcounter{parno}
\renewcommand{\theparno}{\thesection.\arabic{parno}}
\newcommand{\parn}{\refstepcounter{parno}\noindent\textbf{\theparno .} \space}

\section{Introduction}
\parn Let $G$ be a complex connected reductive algebraic group with Lie algebra $\g$, nilpotent cone $\N(\g)$. Identify $\g$ with $\g^*$ by a choice of non-degenerate $G$-invariant trace form on $\g$. The primitive spectrum $\Prim U(\g)$ is the set of all primitive ideals of the enveloping algebra, equipped with the Jacobson topology. These ideals are classically studied via their invariants, and the most important of these are the {\it associated variety} and {\it Goldie rank}. The associated variety $\VA(I)$ is defined to be the vanishing locus in $\g$ of the associated graded ideal $\gr I \subseteq S(\g) = \C[\g]$ with respect to the PBW filtration. A celebrated theorem of Kostant states that $\N(\g)$ is the vanishing locus of the positive degree invariant polynomials $S(\g)^G_+$ whilst Joseph's irreducibility theorem  states that $\VA(I)$ is irreducible (see \cite{Ja04} for a detailed survey). Together with Dixmier's lemma, these results show that the associated variety is the closure of a nilpotent orbit. The Goldie rank is defined to be the uniform dimension of the primitive quotient $U(\g)/I$.

It is natural to consider the decomposition $\Prim U(\g) = \bigcup_\O \Prim_\O U(\g)$, were the union is taken over all nilpotent $G$-orbits $\O$ and $\Prim_\O U(\g) = \{ I \in \Prim U(\g) \mid \VA(I) = \overline{\O}\}.$
Now fix an orbit $\O \subseteq \N(\g)$ and $e\in \O$, and let $U(\g,e)$ denote the finite $W$-algebra, first associated to $(\g,e)$ by Premet \cite{Pr02}. The reductive part of the centraliser $G^e(0)$ acts naturally on $U(\g,e)$ by algebra automorphisms, and this induces an action of  the component group $\Gamma = G^e(0)/G^e(0)^\circ$ on the category of finite dimensional modules. Losev famously gave a new construction of $U(\g,e)$ via deformation quantization \cite{Lo10a} and used this to show that $\Prim_\O U(\g)$ is in bijection with $U(\g,e)\lmod_{\fd}/\Gamma$ \cite{Lo11}. The one dimensional representations of $U(\g,e)$ play an especially important role here for two reasons: on one hand the images under Skryabin's equivalence are all completely prime, and therefore play a key role in Joseph's theory of Goldie rank polynomials \cite{Lo15}, and on the other hand they classify quantizations of $G$-equivariant coverings of $\O$ \cite{Lo10b}.
\vspace{8pt}

\parn The above narrative leads us to consider the affine scheme $\E(\g,e) := \Spec U(\g,e)^\ab$ associated to the maximal abelian quotient. By Hilbert's nullstellensatz the closed points classify the one dimensional representations of $U(\g,e)$. The work of Losev and Premet \cite{Lo10a, Pr14} shows that $\E(\g,e)$ is nonempty and in \cite{Pr10, PT14} the first steps were made towards a full description of the variety of closed points.

Recall that the sheets of $\g$ are the maximal irreducible subsets consisting of orbits of constant dimension. They are classified via the theory of decomposition classes which, in turn, are classified by the Lusztig--Spaltenstein induction data. One of the main themes of \cite{PT14}, which we build upon in this paper, is the interplay between the sheets of $\g$ and the structure of $\E(\g,e)$.

In the case where $\g$ is classical, by which we mean a simple Lie algebra of type {\sf A, B, C} or {\sf D}, we described a combinatorial procedure for enumerating the sheets of $\g$ containing a given orbit $\O$, and we named it the {\it Kempken--Spaltenstein (KS) algorithm}. This algorithm played a key role in \cite[Theorem~1]{PT14}, which states that for $\g$ classical the variety $\E(\g,e)$ is an affine space if and only if $e$ lies in a unique sheet of $\g$. The first goal of this paper is to elucidate the structure of $\E(\g,e)$ when $\g$ is classical and $e$ is {\it singular}, which means that it lies in multiple sheets.

Let $\Sc_1,...,\Sc_l$ be the set of all sheets containing $\O \in \N(\g)/G$. If $e + \g^f$ denotes the Slodowy slice to $\O$ at $e$ then we define the {\it Katsylo variety}
\begin{eqnarray}
\label{e:Katsylointrodef}
e + X := (e + \g^f) \cap \bigcup_{i=1}^l \Sc_i.
\end{eqnarray}
In \cite{Ka82} Katsylo used this variety to construct a geometric quotient of the variety $\bigcup_{i=1}^l \Sc_i$. Perhaps the first indication that $e+X$ should influence the representation theory of $U(\g,e)$ appeared in \cite{Pr10}. Premet used reduction modulo $p$ to show that there is a surjective map on the sets of irreducible components
\begin{eqnarray}
\label{e:Premetsmap}
\Comp \E(\g,e) \longtwoheadrightarrow \Comp (e + X)
\end{eqnarray}
which restricts to a dimension preserving bijection on some subset of $\Comp \E(\g,e)$. The following is our first main result. 
\begin{Theorem}
\label{T:main}
When $\g$ is a simple Lie algebra of classical type, the map \eqref{e:Premetsmap} is a dimension preserving bijection. 
\end{Theorem}
\noindent Thanks to \cite[Corollary~3.2]{Pr10} the variety $\E(\g,e)$ is irreducible when $\g = \sl_n$, and so Theorem~\ref{T:main} follows in this case, using the properties of \eqref{e:Premetsmap} listed above. Hence we focus on types {\sf B}, {\sf C} and {\sf D} in this paper. For these classical types, the dimensions of the irreducible components of $e+X$ can be calculated from the KS algorithm, which depends only on the partition associated to $e$; see Proposition~\ref{P:Katsyloproperties}(2) and Proposition~\ref{P:KSclassifiessheets}. Thus Theorem~\ref{T:main} provides an effective method for computing dimensions of all components of $\E(\g,e)$. We note that these dimensions were calculated in low ranks in \cite{BG18}.
\vspace{8pt}

\parn In \cite{Lo22} Losev demonstrated that for every conic symplectic singularity the functor of filtered quantizations of Poisson deformations admits an initial object (see \cite{ACET20} for more detail). We call such an initial object a {\it universal quantization}. Using this result, he then showed that every coadjoint orbit uniquely gives rise to a quantization of the affinization of a certain cover of a nilpotent orbit, and that each such quantization give rise to a completely prime primitive ideal. Thus we have a map $\J : \g^* / G \to \Prim U(\g)$, which is known to be an embedding whenever $\g$ is classical \cite[Theorem~5.3]{Lo22}. The search for such a map is motivated by the orbit method of Kostant and Kirillov, and we will refer to the map as {\it Losev's orbit method map}. An introduction to the orbit method can be found in \cite{Vo94}, whilst Losev's construction is surveyed in Section~\ref{ss:birationalandinduction} of the current paper.

It is important to understand and characterise the primitive ideals appearing in the image of the orbit method map for $\g$. Losev has conjectured that they are precisely the annihilators of simple Whittaker modules coming from one dimensional representations of $W$-algebras. In the final Section of this paper we deduce his conjecture from Theorem~\ref{T:main}.
\begin{Theorem}
\label{T:Losevsmap}
For $\g$ classical, the image of $\J$ consists of primitive ideals obtained from one dimensional representations of $W$-algebras under Skryabin's equivalence.
\end{Theorem}

\parn It is worth commenting on Theorems~\ref{T:main} and \ref{T:Losevsmap} in the context of exceptional Lie algebras: there are 6 rigid orbits for which the finite $W$-algebra admits two 1-dimensional representations (see \cite[Table~1]{PT21} and \cite{Pr14} for more detail). In these cases $\E(\g,e)$ has two points and $e+X$ consists of a single point. This shows that Theorem~\ref{T:main} cannot hold outside classical types. Similarly, for each rigid orbit Losev's orbit method map attaches a unique primitive ideal of $U(\g)$ with associated variety equal to the closure of $G\cdot e$ (see Lemma~\ref{L:associatedvariety}) and so these same examples show that Theorem~\ref{T:Losevsmap} also fails outside classical types.
\vspace{8pt}

\parn There are several new tools involved in the proof of Theorem~\ref{T:main}, and we now briefly describe the most important ones. One of the basic ideas comes from deformation theory. The finite $W$-algebra is a filtered quantization of the transverse Poisson structure on the Slodowy slice $\C[e+\g^f]$ and so there are two natural degenerations associated to $\E(\g,e)$. On the one hand, we may degenerate $U(\g,e)$ to the classical finite $W$-algebra $S(\g,e) \cong \C[e+\g^f]$ and then abelianise, which leads to the spectrum of the maximal Poisson abelian quotient $\Fc(\g,e) := \Spec S(\g,e)^\ab$. On the other hand we may abelianise and then degenerate, which leads us to the {\it asymptotic cone of $\E(\g,e)$}, denoted $\C\E(\g,e) := \Spec (\gr U(\g,e)^\ab)$.

One of the general results of this paper states that there is a closed immersion of schemes inducing a bijection on closed points, which we prove using reduction to prime characteristic
\begin{eqnarray}
\label{e:coneimmersedinKatsylo}
\C\E(\g,e) \longhookrightarrow \Fc(\g,e).
\end{eqnarray}
Furthermore by considering the rank strata and symplectic leaves of the Poisson structure of $e + \g^f$ we see that the reduced subscheme associated to $\Fc(\g,e)$ is $e + X$ (Proposition~\ref{P:Katsyloproperties}). Therefore the main theorem will follow if we can show that $\Comp \E(\g,e)$ is no larger than $\Comp \C\E(\g,e)$. It is an elementary fact from commutative algebra that $\Comp \E(\g,e) \le\Comp \C\E(\g,e)$ provided $\C\E(\g,e)$ is reduced, and so our approach is to show that $\gr U(\g,e)^\ab$ has no nilpotent elements. By \eqref{e:coneimmersedinKatsylo} it suffices to show that $S(\g,e)^\ab$ is reduced.

In this paragraph we take $\g$ classical. By passing to the completion of $S(\g,e)^\ab$ at the maximal graded ideal and using the fact that the Slodowy slice is transverse to every point of $e + X$ we are able to reduce the problem of showing that $S(\g,e)^\ab$ is reduced to the case where $e$ is distinguished. To be more precise, $S(\g,e)^\ab$ is reduced if and only if the completion at the maximal graded ideal is so, and we show that this completion is reduced if and only if $(S(\tilde\g, \tilde e)_x^\c)^\ab$ is reduced, where $x$ is a point on the Slodowy slice attached to a distinguished element in a larger classical Lie algebra. For the proof we make use of the fact that transverse Poisson manifolds are locally diffeomorphic, which follows from Weinstein's splitting theorem \cite[Theorem~2.1]{We83}.\vspace{8pt}

\parn Now Theorem~\ref{T:main} will follow if we can show that $S(\g,e)^\ab$ is reduced for distinguished elements $e$. We introduce a new method to attack this problem. If $X$ is a complex Poisson scheme of finite type and $H$ is a reductive group acting rationally by Poisson automorphisms then the invariant subscheme $X^H$ can be equipped with a Poisson structure via Dirac reduction. The reduced Poisson algebra will be denoted $R(\C[X], H)$. Now if $\g = \Lie(G)$ for $G$ reductive and $H \subseteq \Aut(\g)$ is a reductive group fixing an $\sl_2$-triple $\{e,h,f\} \subseteq \g$ then we prove the following isomorphism of Poisson algebras
\begin{eqnarray}
\label{e:Diracreductionisomorphism}
R(S(\g,e), H) \isoto S(\g^H, e).
\end{eqnarray}
Some special cases of \eqref{e:Diracreductionisomorphism} were discovered by Ragoucy \cite{Ra01}.

We apply this isomorphism in the case where $\g = \gl_n$ and $H = \Z/2\Z$ is generated by some involution $\tau$. Then $\g^\tau = \so_n$ or $\g = \sp_n$. It follows from the work of Brundan and Kleshchev that $S(\g, e)$ is a quotient of a (semiclassical) shifted Yangian $y_n(\sigma)$ depending on $e$. If $e$ is distinguished then we can identify an involution $\tau$ on $y_n(\sigma)$ defined so that the Poisson homomorphism $y_n(\sigma) \onto S(\g,e)$ is $\tau$-equivariant. As a consequence we can apply the Dirac reduction procedure to the shifted Yangian $y_n(\sigma)$ and thus obtain a presentation of the Poisson structure on $S(\g,e)$. We mention that obtaining presentations of finite $W$-algebras outside of type {\sf A} is one of the key open problems in the field.

Let $\m_0$ be the maximal graded ideal. Finally we use the presentation of $S(\g,e)^\ab$ to calculate generators and certain relations of $\gr_{\m_0} S(\g,e)^\ab$. We see that the reduced algebra of $\gr_{\m_0} S(\g,e)^\ab$ is naturally identified with the coordinate ring on the tangent cone $\C[\TC_e(e+X)]$. Micha{\"e}l Bulois has recently demonstrated that the sheets of $\g$ containing $e$ are transversal at $e$ (work in preparation \cite{Bu}), which allows us to calculate the dimensions of the irreducible components of $\TC_e(e+X)$, once again they are determined by the KS algorithm. Using a combinatorial argument we then show that the relations mentioned above give a full presentation for $\gr_{\m_0}S(\g,e)^\ab$. It follows quickly that both $\gr_{\m_0}S(\g,e)^\ab$ and $S(\g,e)^\ab$ are reduced, which allows us to conclude the proof of Theorem~\ref{T:main}.
\vspace{8pt}

\parn Since the presentation of the distinguished semiclassical finite $W$-algebras in types {\sf B, C, D} is an important result in its own right we formulate it straight away. For the proof combine Theorem~\ref{T:PDyangian} and Propositions~\ref{P:Walgpresentation1} and \ref{P:Walgpresentation-1}.
\begin{Theorem}
\label{T:Slicepresentation}
Let $\g = \so_N$ or $\sp_N$ and let $e$ be a distinguished nilpotent element with partition $\lambda = (\lambda_1,...,\lambda_n)$. Then $S(\g,e)$ is generated as a Poisson algebra by elements
\begin{eqnarray}
\label{e:DYgensintro}
\{\eta_i^{(2r)} \mid 1\le i \le n, \ 0 < r\} \cup \{ \theta_i^{(r)} \mid 1\le i < n, \ \frac{\lambda_{i+1} - \lambda_i}{2} < r\}
\end{eqnarray}
together with the following relations
\begin{eqnarray}
& & \{\eta_i^{(2r)}, \eta_j^{(2s)}\} = 0 \\ 
& &\big\{\eta_i^{(2r)}, \theta_j^{(s)}\big\} = (\delta_{i,j} - \delta_{i,j+1}) \sum_{t=0}^{r-1} \eta_i^{(2t)} \theta_j^{(2r+s -1- 2t)}
\end{eqnarray}
\begin{eqnarray}
\label{e:dyrel34}
\{\theta_i^{(r)}, \theta_i^{(s)}\} = \frac{1}{2} \sum_{t=r}^{s-1} \theta_i^{(t)} \theta_i^{(r+s-1-t)} + (-1)^{s_{i}}\varpi_{r,s} \sum_{t=0}^{(r+s-1)/2} \eta_{i+1}^{(r+s-1-2t)} \teta_{i}^{(2t)} & \text{ for } & r < s.\\
\big\{\theta_i^{(r+1)}, \theta_{i+1}^{(s)}\big\} - \big\{\theta_i^{(r)}, \theta_{i+1}^{(s+1)}\big\} = \frac{1}{2}\theta_i^{(r)} \theta_{i+1}^{(s)} & &
\end{eqnarray}
\begin{eqnarray}
\big\{\theta_i^{(r)}, \theta_j^{(s)}\big\} = 0 & & \text{ for } |i - j| > 1\\
\Big\{\theta_i^{(r)}, \big\{\theta_i^{(s)}, \theta_j^{(t)}\big\}\Big\} + \Big\{\theta_i^{(s)}, \big\{\theta_i^{(r)}, \theta_j^{(t)}\big\}\Big\} = 0 & & \text{ for } |i - j| =1 \text{ and } r + s \text{ odd}
\end{eqnarray}
\begin{eqnarray}
& &\Big\{\theta_i^{(r)}, \big\{\theta_i^{(s)}, \theta_j^{(t)}\big\}\Big\} + \Big\{\theta_i^{(s)}, \big\{\theta_i^{(r)}, \theta_j^{(t)}\big\}\Big\} =  \\
& & \nonumber  \ \ \ \ \ \ \ \ \ \  2(-1)^{s+s_{i}-1} \delta_{i,j+1} \sum_{m_1=0}^{m-1} \sum_{m_2=0}^{m_1} \eta_{i+1}^{(2(m-m_1-1))}  \teta_i^{(2m_2)} \theta_j^{(2(m_1-m_2) + t)} \\
& & \nonumber \ \ \ \ \ \ \ \ \ \ \ \ \ \ \ \ \ \ \ \ + 2(-1)^{s+s_{i}-1} \delta_{i+1,j} \sum_{m_1=0}^{m-1} \sum_{m_2=0}^{m-m_1-1} \teta_i^{(2m_1)}\eta_{i+1}^{(2m_2)} \theta_j^{(2(m-m_1-m_2 - 1) + t)} \\
& & \nonumber \ \ \ \ \ \ \ \ \ \ \ \ \ \ \ \ \ \ \ \ \ \ \ \ \ \ \ \ \ \ \text{ for } |i - j| =1 \text{ and } r + s = 2m \text{ even}
\end{eqnarray}
\begin{eqnarray}
\eta_1^{(2r)} = 0 & \text{ for } & 2r > \lambda_1\\
\ttheta_i^{(\lambda_i + s_i + 1)} = 0 & \text{ for } & \ i=1,...,n-1 \text{ when } \g = \sp_N.
\end{eqnarray}
where we adopt the convention $\eta_i^{(0)} = \teta_i^{(0)} = 1$, the symbol $\varpi_{r,s} \in \{-2, 0, 2\}$ is defined in \eqref{e:definevarpi}, the elements $\ttheta_i^{(\lambda_i + s_{i} + 1)}$ are defined recursively in \eqref{e:additionalgenerators} and $s_{i} := |\lambda_i - \lambda_{i+1}|/2$. The elements $\{\teta_i^{2(r)} \mid  r\ge0\}$ are defined via the recursion
\begin{eqnarray}
\label{e:tetatwisteddefinition}
\teta_i^{(2r)} := -\sum_{t=1}^r \eta_i^{(2t)} \teta_i^{(2r-2t)}.
\end{eqnarray}
\end{Theorem}
\begin{Remark}
In fact the presentation holds for a slightly larger class of nilpotent elements than the distinguished nilpotent elements: in symplectic types we only require the sizes of Jordan blocks of $e$ to be even, and in orthogonal types we only require them to be odd.
\end{Remark}

\parn To conclude the introduction we describe the structure of the paper, which is divided into two parts. The first part is very algebraic, dealing with Dirac reduction and the presentation of semiclassical $W$-algebras in the distinguished case. The second part is more geometric, studying the degenerations of $\E(\g,e)$ using Lusztig--Spaltenstein induction and the tangent cone of the Katsylo variety.

{\bf Part I:} We begin Section~\ref{s:Diracreduction} by giving an elementary introduction to the version of Dirac reduction used in this paper. In Subsection~\ref{ss:quantumDiractheorem} we prove the isomorphism \eqref{e:Diracreductionisomorphism}, which should have independent interest. In Section~\ref{s:DracredcutionforshiftedYangians} we describe the semiclassical shifted Yangian $y_n(\sigma)$ by generators and relations. The main results on the Dirac reduction of $y_n(\sigma)$ are presented in Subsection~\ref{ss:PDreductionofyangians}, including the canonical grading, the loop filtration, the PBW theorem and the presentation by generators and relations. All of the results about $R(y_n(\sigma), \tau)$ are ultimately deduced from similar results on $y_n(\sigma)$. In Section~\ref{s:finiteWalgebrasforclassical} we recall the definition of Brundan--Kleshchev's isomorphism $y_n(\sigma) \onto S(\g,e)$ and show by an explicit calculation that this is $\tau$-equivariant for a suitable choice of involution on $y_n(\sigma)$. By \eqref{e:Diracreductionisomorphism} this leads to a surjection $R(y_n(\sigma), \tau) \onto S(\g^\tau, e)$ and in Subsection~\ref{ss:semiclassicalyangiansandWalgebras} we describe a full set of Poisson generators for the kernel. 

{\bf Part II:} In Section~\ref{s:poissonschemes} we gather together some important general facts about degenerations and completions of schemes, as well as reviewing the theory of rank stratification and symplectic leaves of a Poisson scheme. In Section~\ref{s:sheetsandinduction} we explain how $S(\g,e)^\ab$ is related to Katsylo variety $e + X$, and use the aforementioned results of Bulois to enumerate the irreducible components of $\TC_e(e+X)$ and calculate their dimensions. In Section~\ref{s:quantumabelianquotients} we introduce quantum finite $W$-algebras and prove the existence of the closed immersion \eqref{e:coneimmersedinKatsylo}. The proof of the latter uses a reduction modulo $p$ argument similar to Premet's construction of the component map \eqref{e:Premetsmap} in \cite[Theorem~1.2]{Pr10}, along with the identification of reduced schemes $\Fc(\g,e)_\red = \C[e+X]$ from Section~\ref{s:sheetsandinduction}. Finally in Section~\ref{s:abquotsection} we describe the Kempken--Spaltenstein algorithm, as well as its relationship with sheets, and then use this to construct an algebraic variety $X_\lambda$ associated to a distinguished nilpotent orbit $\O$ with partition $\lambda$, which we call the {\it combinatorial Katsylo variety}. In Theorem~\ref{T:distinguishedreduced} we use the presentation of $S(\g,e)^\ab$ obtained in part I of the paper to demonstrate that $\C[X_\lambda] \onto S(\g,e)^\ab \onto \C[\TC_e(e+X)]$, and we show that these are isomorphisms by comparing the dimensions of the irreducible components. In particular this implies that $S(\g,e)^\ab$ is reduced for $e$ distinguished. Theorem~\ref{T:generalreduced} reduces the general case to the distinguished case. Finally we conclude the proof of Theorem~\ref{T:main} in Subsection~\ref{ss:abelianisationviadeformation}, making use of deformation techniques gathered in Section~\ref{s:poissonschemes}.

In Subsection~\ref{ss:quantsanddefs} we recall the classification of Poisson deformations and their quantizations for conic symplectic singularities, due to Losev and Namikawa. In Subsection~\ref{ss:birationalandinduction} we recall Losev's theory of birational induction, describe the orbit method map and formulate a slight refinement of Theorem~\ref{T:Losevsmap}. In Subsection~\ref{ss:quantumHam} we recall some of the key properties of Losev's dagger functor from \cite{Lo11} and explain how they treat universal quantizations of affinisations of orbit covers. Finally in Subsection~\ref{ss:commquotandorbitmethod} we relate the orbit method map to $U(\g,e)^\ab$ and prove Theorem~\ref{T:Losevsmap}.

\tableofcontents

\subsection*{Notation and conventions}

The following notation will be used throughout the paper.
All algebras and vector spaces are defined over $\C$, except in Section~\ref{s:quantumabelianquotients} where we use reduction modulo a large prime. We use capital letters $G, H,...$ for algebraic groups and gothic script $\g, \mathfrak{h},...$ for their Lie algebras.

If $A$ is an algebra and $X\subseteq A$ then $(X)$ will denote the two-sided ideal of $A$ generated by $X$.

If $\g$ is a Lie algebra then $U(\g)$ denotes the enveloping algebra and $S(\g)$ the symmetric algebra equipped with its {\it Lie--Poisson structure}: this is the unique Poisson bracket on $S(\g)$ extending the Lie bracket on $\g \subseteq S(\g)$. The associated graded algebra of an almost commutative, filtered associative algebra is equipped with a Poisson structure in the usual manner. 
 If $A$ is a commutative algebra and $I \subseteq A$ an ideal, then we write $\gr_I A$ for the graded algebra with respect to the $I$-adic filtration.

Almost all schemes appearing in this paper will be affine schemes of finite type over $\C$, thus the reader may almost always think of complex affine varieties, except that the coordinate rings will often be non-reduced. In fact, the consideration of nilpotent elements will be vital to our main results. Occasionally we write $\mSpec(A)$ for the variety of closed points in the prime spectrum of a commutative algebra $A$. If $X$ is a noetherian scheme then we write $\Comp(X)$ for the set of irreducible components.

\vspace{8pt}



\subsection*{Acknowledgements} I would like to offer thanks to Simon Goodwin, Ivan Losev, Dmytro Matvieievskyi, Sasha Premet and Matt Westaway for useful comments on the first version of this paper. I would also like to thanks Lukas Tappenier for careful reading, and spotting various typos. I'm especially grateful to Ivan for suggesting some of the constructions used in the proof of Theorem~\ref{T:Losevsmap}, and to Sasha to whom this paper is dedicated - his insightful teaching first introduced me to these fascinating problems. I would also like to thank the referees for their many comments which have improved the exposition of this paper significantly. I have also benefited from many interesting conversations and email correspondence with Jon Brundan, Micha{\"e}l Bulois, Paul Levy, Anne Moreau, and Daniele Valeri, and am grateful for all their help. Some of these results were announced at the conference ``Geometric and automorphic aspects of $W$-algebras'', Lille 2019. This research is supported by the UKRI Future Leaders Fellowship project ``Geometric representation theory and $W$-algebras'', grant numbers MR/S032657/1, MR/S032657/2, MR/S032657/3.

\part{ Presentations of classical $W$-algebras}

\section{Dirac reduction for classical finite $W$-algebras}
\label{s:Diracreduction}

\subsection{Invariants via Dirac reduction}
\label{ss:quantizingviaDirac}

Throughout this Section $H$ will be a reductive algebraic group, although in our later applications $H$ will be a cyclic group of order 2. When $H$ acts locally finitely and rationally on some vector space $A$ we write $A_+ := A^H$ for the invariants and write $A_-$ for the unique $H$-invariant complement to $A_+$ in $A$.

Let $X$ be an complex affine Poisson variety and suppose $z_1,...,z_n\in \C[X]$ such that the determinant of the matrix $(\{z_i, z_j\})_{1\le i,j\le n}$ is a unit. In his seminal paper \cite{Dir50} Dirac defined a new Poisson bracket on $\C[X]$, such that the $z_i$ are Casimirs, thus equipping $\C[X]/(z_1,...,z_n)$ with a Poisson structure. In fact this is a special case of the following procedure: say that $I \subseteq \C[X]$ is a {\it Dirac ideal} if $N_{\C[X]}(I)\onto \C[X]/I$ surjects where $N_{\C[X]}(I):=\{f\in \C[X] \mid \{f, I\} \subseteq I\}$ denotes the Poisson idealiser. Then the subscheme associated to $I$ inherits a Poisson structure from $X$, and we call this induced structure the {\it Dirac reduction} \cite[\textsection 5.4.3]{LPV13}.

Now let $H$ be a reductive group acting rationally on $X$ by Poisson automorphisms. We regard the set of invariants $X^H$ as a (not necessarily reduced) affine scheme such that the structure sheaf has global sections $\C[X^H] := \C[X] / I_H$ and $I_H := (h\cdot f - f \mid h\in H, f\in \C[X])$. Although $I_H$ is usually not a Poisson ideal, it is always a Dirac ideal (this is a corollary of Lemma~\ref{L:Poissonreduction}) so that $X^H$ acquires the structure of a Poisson scheme.

The following suggests an alternative approach of the Poisson structure on $X^H$, better-suited to calculations. 
\begin{Lemma}
\label{L:Poissonreduction}
The map $\C[X]^H \to \C[X^H]$ is surjective and
\begin{eqnarray}
\C[X]^H / \C[X]^H \cap I_H \isoto \C[X^H]
\end{eqnarray}
is a Poisson isomorphism. 
\end{Lemma}
\begin{proof}
If $f\in \C[X]^H$, $g\in \C[X]$ and $h\in H$ then $\{f, h\cdot g - g\} = h\cdot\{f, g\} - \{f,g\}$ and so $\C[X]^H \subseteq N_{\C[X]}(I_H)$ is a Poisson subalgebra. Since $H$ is reductive and acts rationally we can decompose $\C[X] = \C[X]_+ \oplus \, \C[X]_-$ into its $H$-invariant part, and its complementary $H$-submodule. If $V \subseteq \C[X]_-$ is an irreducible $H$-submodule then the vector space spanned by $\{h\cdot v - v\mid h\in H, v\in V\} \subseteq V$ is $H$-stable and nonzero hence equal to $V$, and it follows that $I_H = (\C[X]_-)$. Therefore the composition $\C[X]^H \into N_{\C[X]}(I_H) \to \C[X^H] = \C[X]/I_H$ is surjective.
\end{proof}

Now let $A$ be any Poisson algebra with a locally finite, rational action of $H$ by Poisson automorphisms. Let $I_H$ denote the ideal generated by $A_-$ and define {\it the Dirac reduction of $A$ by $H$} by
\begin{eqnarray}
\label{e:definePDreduction}
R(A, H) := A^H / A^H \cap I_H
\end{eqnarray}
If $\tau\in \Aut(A)$ is a semisimple Poisson automorphism of finite order, then we often abuse notation writing $R(A, \tau)$ for $R(A, H)$ where $H$ is the group generated by $\tau$.

Since $H$ is reductive the functor of $H$-invariants is exact on the category of locally finite, rational $H$-modules. This implies that whenever $A$ is a Poisson algebra as above and $I \subseteq A$ is an $H$-stable Poisson ideal we have
\begin{eqnarray}
\label{e:notrightexact}
R(A, H) \longtwoheadrightarrow R(A/I, H).
\end{eqnarray}

\begin{Lemma}
\label{L:PBWforDiracreduction}
Suppose the following:
\begin{enumerate}
\item $V$ is a direct sum of locally finite $H$-modules with $H$-stable decomposition $V = V^H \oplus V_-$.
\item $S(V)$ is a Poisson algebra with $H$ acting by Poisson automorphisms.
\end{enumerate}
Then the natural map $S(V^H) \to R(S(V), H)$ is an isomorphism of commutative algebras.
\end{Lemma}
\begin{proof}
Since $S(V) = S(V^H) \oplus (V_-)$ and $S(V^H)$ is $H$-fixed we must have $S(V)_- \subseteq (V_-)$. Combining with $V_- \subseteq S(V)_-$ we deduce that $I_H = (V_-)$ which proves the map $S(V^H) \to R(S(V), H)$ is a commutative algebra isomorphism.
\end{proof}

\begin{Remark}
\label{R:invariants}
Let $\g$ be a Lie algebra and $H$ a reductive group of automorphisms of $\g$ acting locally finitely. In this case the composition $S(\g^H) \to S(\g)^H \to R(S(\g), H)$ is a Poisson homomorphism and so Lemma~\ref{L:PBWforDiracreduction} shows that $S(\g^H) \isoto R(S(\g), H)$ as Poisson algebras. 
\end{Remark}

\subsection{Classical finite $W$-algebras}
\label{ss:classicalfiniteWalg}
For the rest of the section we fix a connected reductive algebraic group $G$ such that the derived subgroup is simply connected, and write $\g = \Lie(G)$. Let $\kappa : \g\times\g \to \C$ be a choice of non-degenerate trace form on $\g$ which is preserved by $\Aut(\g)$. Pick a nilpotent element $e\in \g$ and write $\chi := \kappa(e, \cdot) \in \g^*$. Pick an $\sl_2$-triple $\{e,h,f\}$ and write $\g = \bigoplus_{i\in \Z} \g(i)$ for the grading by $\ad(h)$-eigenspaces.  Throughout the paper we use the notation $\g(\le \! i) = \bigoplus_{j \le i} \g(j)$ and similar for $\g(\!<\!i)$. Since $e\in \g(2)$ we see that $\chi$ restricts to a character on $\g(\lemw)$. Make the following notation $\g(\lemw)_\chi := \{ x - \chi(x) \mid x\in \g(\lemw)\} \subseteq S(\g)$. The nilpotent Lie algebra $\g(\lez)$ is algebraic and we write $\g(\lez) = \Lie G(\lez)$.

The (classical) finite $W$-algebra associated to $(\g,e)$ is a Poisson reduction of $S(\g)$
\begin{eqnarray}
S(\g,e) := (S(\g) / S(\g)\g(\lemw)_\chi)^{G(<0)}.
\end{eqnarray}
In more detail, the Poisson normaliser $N = \{f \in S(\g) \mid \{f, \g(\lemw)_\chi \} \subseteq S(\g)\g(\lemw)_\chi$ is a Poisson subalgebra of $S(\g)$ with $N \cap S(\g)\g(\lemw)_\chi$ embedded as a Poisson ideal, and $(S(\g) / S(\g)\g(\lemw)_\chi)^{\g(<0)}$ is equipped with a Poisson structure via the isomorphism $N / N\cap S(\g)\g(\lemw)_\chi \cong (S(\g) / S(\g)\g(\lemw)_\chi)^{\g(<0)}$. 

The {\it Kazhdan grading} is defined on $S(\g)$ by placing $\g(i)$ in Kazhdan degree $i + 2$. Notice that $\g(\lemw)_\chi$ generates a homogeneous ideal of $S(\g)$ and that $S(\g) / S(\g)\g(\lemw)_\chi$ inherits a connected grading in non-negative degrees, with $S(\g,e)$ embedded as a graded subalgebra, with Poisson bracket in degree $-2$.

Since the grading is good for $e$ the map $\g^e \to S(\g)/ S(\g)\g(\lemw)_\chi$ is injective.
\begin{Theorem}
\label{T:PBW}
Let $\m_0 \subseteq S(\g)/ S(\g)\g(\lemw)_\chi$ denote the unique maximal graded ideal.
\begin{enumerate}
\item There exists a Kazhdan graded map $\theta : \g^e \to S(\g,e)$ such that \begin{eqnarray}
\label{e:thetaproperty}
\theta(x) - x \in \m_0^2.
\end{eqnarray}
Furthermore $\theta$ can be chosen to be equivariant with respect to any reductive group of Poisson automorphisms acting rationally on $S(\g,e)$ by graded automorphisms.
\item If $\theta$ is any map satisfying \eqref{e:thetaproperty} then the induced map $S(\g^e) \to S(\g,e)$ is an isomorphism of commutative algebras.
\end{enumerate}
\end{Theorem}
\begin{proof}
It follows from \cite[Lemma~2.1]{GG02} that the restriction homomorphism $S(\g)/ S(\g)\g(\lemw)_\chi \cong \C[\chi + \g(\lemw)^\perp] \to \C[\chi + \g^f] \cong S(\g^e)$ gives a $G^e(0)$-equivariant isomorphism $S(\g,e) \to S(\g^e)$ of commutative algebras, where $\g(\lemw)^\perp := \{\eta\in \g^* \mid \eta(\g(\lemw)) = 0\}$. Taking the inverse isomorphism restricted to $\g^e \subseteq S(\g^e)$ gives the desired map $\theta$. If $H \subseteq G^e(0)$ is any reductive group of Poisson isomorphisms then $\theta$ can be replaced with an $H$-equivariant map using the standard trick of projecting onto isotypic components of $S(\g,e)$ for the $H$-action.

Now if $\theta$ is any map satisfying \eqref{e:thetaproperty} then $S(\g^e) \to S(\g,e)$ and it suffices to show that this map is surjective. If $x_1,...,x_r \in \g^e$ is a homogenous basis then \cite[Lemma~7.1]{Ja04} shows that $\theta(x_1),...,\theta(x_r)$ generate a graded radical ideal of finite codimension. The only such ideal is the maximal graded ideal of $S(\g,e)$ and this implies that $\theta(x_1),...,\theta(x_r)$ generate $S(\g,e)$ as a commutative algebra.
\end{proof}

\subsection{Dirac reduction of classical $W$-algebras}  
\label{ss:quantumDiractheorem}

Now let $H \subseteq \Aut(\g)$ a reductive subgroup which fixes our choice of $\sl_2$-triple pointwise. Make the notation $\g_+ := \g^H$ and let $\g_-$ denote an $H$-invariant complement to $\g_+$ in $\g$. Also write $G_+ \subseteq G$ for the connected component of the subgroup consisting of elements $g\in G$ such that $\Ad(g)$ commutes with the action of $H$.

Our next goal is to relate the Dirac reduction $R(S(\g,e), H)$ with the classical $W$-algebra $S(\g_+,e)$. Later in the paper we will only exploit this relationship in the case where $\g = \gl_N$ and $H$ is a cyclic group of order 2. For the sake of simplicity, we state and prove our results in a much greater generality.

\begin{Lemma}
\label{L:kapparestrictsandtheinvariantgroup}
\begin{enumerate}
\item The restriction of $\kappa$ to $\g_+$ is non-degenerate and $\g_+^\perp = \g_-$.
\item $G_+$ is reductive and $\g_+ = \Lie(G_+)$.
\end{enumerate}
\end{Lemma}
\begin{proof}
Since $\g_-$ is spanned by elements $\{h \cdot x - x \mid h\in H, x\in \g\}$ (see Lemma~\ref{L:Poissonreduction}) it follows from a short calculation that $\kappa(\g_+, \g_-) = 0$. Since $\kappa$ is non-degenerate and $\g = \g_+ \oplus \g_-$, part (1) follows.

Since $\kappa$ is a non-degenerate trace form, we may pick a representation $\rho : \g\to \gl(V)$ such that $\kappa(x,y) = \Tr(\rho(x) \rho(y))$. If $\n\subseteq \g_+$ is a nilpotent ideal then by Engel's theorem there exists $k>0$ such that $\rho(\n)^k = 0$. Hence $(\rho(x) \rho(n))^k = 0$ for all $x\in \g_+$ and $n\in \n$, so $\kappa(x, n) = 0$. Since $\kappa$ is non-degenerate the nilradical of $\g_+$ is trivial.

If $\Ad_G : G \to \GL(\g)$ is the adjoint representation then we consider $\rho_+ :=  \Ad_{\GL(\g)} \circ \Ad_G : G \to \GL(\End(\g))$ and let $\rho : G \to \GL(W)$ be any faithful representation admitting $\rho_+$ as a direct summand. If we identify $G$ (resp. $\g$) with its image in $\GL(W)$ (resp. $\gl(W)$) via $\rho$ (resp. $d_0 \rho$), and identify $H$ with a subset of $\End(\g) \subseteq W$, then $G_+$ is precisely the subgroup of $\GL(W)$ fixing $H$, and similar for $\g_+$. Now apply \cite[Theorem~13.2]{Hum75} to see that $\g_+ = \Lie(G_+)$. Since the nilradical of $\g_+$ is trivial, $G_+$ is reductive thanks to \cite[Theorem~13.5]{Hum75}.
\end{proof}

Since $\g_+$ is a reductive subalgebra of $\g$ containing $(e,h,f)$ we can consider $S(\g_+, e)$. Furthermore $H$ preserves $\g(\lemw)_\chi$ and the induced action on $S(\g)/S(\g) \g(\lemw)_\chi$ stabilises the $G(\lez)$-invariants, so that $H$ acts by Poisson automorphisms on $S(\g,e)$. 

We are now ready to formulate one of our first main theorems, stated in \eqref{e:Diracreductionisomorphism}.
\begin{Theorem}
\label{T:firstmain}
Let $\g$ be the Lie algebra of a reductive group. The projection map $\phi$ defined in \eqref{e:theprojectionmap} descends to a Poisson isomorphism
\begin{eqnarray}
\label{e:semiclassicalreductioniso}
R(S(\g,e), H) \isoto S(\g_+, e).
\end{eqnarray}
\end{Theorem}

The proof of Theorem~\ref{T:firstmain} will be given at the end of the current section. First we prepare for the proof with two Lemmas and a Proposition.

Thanks to Lemma~\ref{L:kapparestrictsandtheinvariantgroup}(1) we know that $\g_-$ is the orthogonal complement of $\g_+$ with respect to $\kappa$. This implies
\begin{equation}
\label{e:chigone}
\chi(\g_-) = 0.
\end{equation}
\begin{Lemma}
\label{L:idealcaps}
\begin{enumerate}
\setlength{\itemsep}{4pt}
\item $S(\g_+) \cap S(\g) \g(\lemw)_\chi = S(\g_+) \g_+(\lemw)_\chi$;
\item $S(\g) \g_- \cap S(\g) \g(\lemw)_\chi = S(\g) (\g_-(\lemw) + \g_-\g_+(\lemw)_\chi)$. 
\end{enumerate}
\end{Lemma}
\begin{proof}
Observe that $\g_+$ and $\g_-$ are graded subspaces of $\g$. In both parts (1) and (2) it is easy to see that the right hand side is contained in the left (for part (2) one should use \eqref{e:chigone}). Therefore to prove (1) and (2) it remains to show that we have the reverse inclusion.

Let $x\in \g(\lemw)$ and $f \in S(\g)$, and let $f = f_+ + f_-$ and $x = x_+ + x_-$ be the decompositions over $S(\g) = S(\g_+) \oplus S(\g) \g_-$ and $\g = \g_+ \oplus \g_-$ respectively. By \eqref{e:chigone} the projections of $f(x-\chi(x))$ to $S(\g_+)$ and $S(\g)\g_-$ are $f_+(x_+ - \chi(x_+))$ and $f_+x_- + f_-(x-\chi(x))$ respectively. Therefore if $f(x-\chi(x))$ lies in $S(\g_+)$ we must have $f_+x_- + f_-(x-\chi(x)) = 0$ and $f(x-\chi(x)) = f_+ (x_+ - \chi(x_+))$, proving (1).

As a step towards proving (2) we claim that the projection of $S(\g)$ onto $S(\g) \g_-$ across the decomposition $S(\g) = S(\g_+) \oplus S(\g) \g_-$ preserves $S(\g) \g(\lemw)_\chi$. To see this observe that $S(\g) \g(\lemw)_\chi$ is spanned by elements of the form $f (x-\chi(x))$ with $f \in S(\g)$ and $x\in \g(\lemw)$. It suffices to show that the projection of such an element lies in $S(\g)\g(\lemw)_\chi$. Expanding $f = f_+ + f_-$ across $S(\g) = S(\g_+) \oplus S(\g) \g_-$ and writing $x = x_+ + x_- \in \g_+ \oplus \g_-$ we see that the required projection is $fx_- + f_-(x_+ - \chi(x_+)) \in S(\g)\g(\lemw)_\chi$ thanks to \eqref{e:chigone}, which proves the claim.

Let $F \in S(\g) \g_- \cap S(\g) \g(\lemw)_\chi$. Then $F = \sum_i f_i(x_i - \chi(x_i))$ for some elements $f_i \in S(\g)$ and $x_i \in \g(\lemw)$. The fact that $F \in S(\g) \g_-$ implies that the projection of $F$ to $S(\g_+)$ is zero. Therefore if we replace each $f_i(x_i - \chi(x_i))$ with its projection to $S(\g) \g_-$ then we leave the equality $F = \sum_i f_i(x_i - \chi(x_i))$ unchanged. It follows that every element of $S(\g) \g_- \cap S(\g) \g(\lemw)_\chi$ can be written as a sum of elements $f(x- \chi(x)) \in S(\g)\g_-$ with $f \in S(\g)$, $x\in \g(\lemw)$ (here we use the claim from the previous paragraph). This observation reduces (2) to showing that every such element $f (x-\chi(x)) \in S(\g) \g_- \cap S(\g) \g(\lemw)_\chi$ lies in $S(\g) (\g_-(\lemw) + \g_-\g_+(\lemw)_\chi)$. We let $f(x-\chi(x))$ lie in the required space and write $x = x_+ + x_-$ where $x_\pm \in \g_\pm(\lemw)$. Then $f(x_- - \chi(x_-)) \in S(\g)\g_-(\lemw)$ by \eqref{e:chigone}, whilst $f(x_+ - \chi(x_+)) \in S(\g)\g_-$ implies that $f\in S(\g)\g_-$ so that $f(x_+ - \chi(x_+)) \in S(\g)(\g_-\g_+(\lemw)_\chi)$. This proves (2).
\end{proof}

\begin{Lemma}
\label{L:decompprop}
The ideal $S(\g)\g(\lemw)_\chi$ is the direct sum of its intersections with $S(\g_+)$ and $S(\g)\g_-$. Therefore $S(\g) = S(\g_+) \oplus S(\g) \g_-$ gives a $G_+(\lez)$-module decomposition
\begin{eqnarray}
\label{e:decomppropS}
S(\g) / S(\g) \g(\lemw)_\chi = S(\g_+) / S(\g_+) \g_+(\lemw)_\chi \oplus S(\g) \g_- / S(\g) (\g_-(\lemw) + \g_-\g_+(\lemw)_\chi).
\end{eqnarray} 
\end{Lemma}
\begin{proof}
We begin by proving the claim
\begin{eqnarray}
\label{e:decompeq}
S(\g) \g(\lemw)_\chi = S(\g_+) \cap S(\g) \g(\lemw)_\chi \oplus S(\g) \g_- \cap S(\g) \g(\lemw)_\chi.
\end{eqnarray}
The $G_+(<0)$-decomposition \eqref{e:decomppropS} is an immediate consequence.

The right hand side of \eqref{e:decompeq} is clearly contained in the left hand side and so our proof will focus on the reverse inclusion. Let $f \in S(\g)$ and $x- \chi(x) \in \g(\lemw)_\chi$. Since $\g(\lemw) = \g_+(\lemw) \oplus \g-(\lemw)$ we can consider two cases: (i) if $x\in \g_-$ then by \eqref{e:chigone} we have $f(x-\chi(x)) \in S(\g) \g_- \cap S(\g) \g(\lemw)_\chi$; (ii) if $x \in \g_+$ then we can write $f = f_+ + f_- \in S(\g_+) \oplus S(\g)\g_-$, in which case $f_+(x-\chi(x)) \in S(\g_+) \g_+(\lemw)_\chi \subseteq S(\g_+) \cap S(\g) \g(\lemw)_\chi$ and $f_- (x - \chi(x)) \in  S(\g) \g_- \cap S(\g) \g(\lemw)_\chi$. 
Now the Lemma follows from Lemma~\ref{L:idealcaps}.
\end{proof}

In what follows we will use the notation $S(\g, e) = S(\g, e)_+ \oplus S(\g, e)_-$ for the decomposition into trivial and non-trivial $H$-modules, generalising our notation for $S(\g)$. Consider the two sets
\begin{eqnarray}
\label{e:introduceN}
N & := & \{f \in S(\g) \mid g \cdot f - f \in S(\g)\g(\lemw)_\chi \text{ for all } g\in G(\lez)\};\\
N_+ & := & \{f \in S(\g_+) \mid g \cdot f - f \in S(\g_+) \g_+(\lemw)_\chi \text{ for all } g\in G_+(\lez)\}.
\end{eqnarray}
By differentiating the locally finite actions on $G(\lez)$ on $S(\g)/S(\g)\g(\lemw)_\chi$ and of $G_+(\lez)$ on $S(\g_+) / S(\g_+) \g_+(\lemw)_\chi$ we see that
\begin{eqnarray}
\label{e:glemwstability}
\begin{array}{l}
\{\g(\lemw), N \cap S(\g)_+\} \subseteq N \cap S(\g)\g(\lemw)_\chi; \vspace{4pt}\\ 
\{\g_+(\lemw), N_+\} \subseteq N_+ \cap S(\g_+)\g_+(\lemw)_\chi.\end{array}
\end{eqnarray}

We consider the projection
\begin{eqnarray}
\label{e:theprojectionmap}
\phi : S(\g) \longrightarrow S(\g_+)
\end{eqnarray}
across the decomposition $S(\g) = S(\g_+) \oplus S(\g) \g_-$. Although this is not a Poisson homomorphism the next Proposition shows that $\phi$ descends to a Poisson homomorphism $ R(S(\g,e), H) \onto S(\g_+, e)$.

\begin{Proposition}
\label{P:Nispoisson}
\begin{enumerate}
\setlength{\itemsep}{4pt}
\item[(i)] $N \cap S(\g)_+ \subseteq S(\g)$ and $N_+ \subseteq S(\g_+)$ are Poisson subalgebras;
\item[(ii)] $\phi(N\cap S(\g)_+) \subseteq N_+$ and the map $\phi : N\cap S(\g)_+ \to N_+$ is a Poisson homomorphism;
\item[(iii)] The map $\pi : S(\g) \to S(\g)/S(\g)\g(\lemw)_\chi$ restricts to a surjective Poisson  homomorphism $$N\cap S(\g)_+ \twoheadrightarrow S(\g,e)_+;$$
\item[(iv)] The map $\pi_+ : S(\g_+) \to S(\g_+)/S(\g_+)\g_+(\lemw)_\chi$ restricts to a surjective Poisson  homomorphism $$N_+ \twoheadrightarrow S(\g_+,e);$$
\item[(v)] The kernel of the map $N\cap S(\g)_+ \onto S(\g,e)_+$ is contained in the kernel of the map $\pi_+\circ \phi$. Thus $\phi$ induces a Poisson homomorphism $S(\g,e)_+ \onto S(\g_+, e)$ which we also denote by $\phi$;
\item[(vi)] $S(\g, e)_+ \cap S(\g,e)S(\g,e)_-$ is contained in the kernel of $S(\g,e)_+ \onto S(\g_+, e)$, inducing a surjective Poisson homomorphism
\begin{eqnarray}
\phi : R(S(\g,e), H) \onto S(\g_+, e).
\end{eqnarray}
\end{enumerate}
\end{Proposition}
\begin{proof}
Since $H$ acts on $S(\g)$ by Poisson automorphisms the invariant subspace $S(\g)_+$ is a Poisson subalgebra. Therefore the proofs of the two claims in (i) are identical, and we will only prove that $N_+\subseteq S(\g_+)$ is a Poisson subalgebra. It is evidently closed under multiplication so we only need to show that it is closed under the bracket. Suppose that $f_1, f_2\in N_+$, that $g\in G_+(\lez)$ and that $g \cdot f_i - f_i = h_i \in S(\g_+) \g_+(\lemw)_\chi$. We have $g\cdot \{f_1, f_2\} = \{g\cdot f_1, g\cdot f_2\}= \{f_1 + h_1, f_2 + h_2\}$ and so (i) will follow if we can show that $\{f_1, h_2\}, \{h_1, f_2\}, \{h_1, h_2\} \in S(\g_+) \g_+(\lemw)_\chi$. Since $\chi$ vanishes on $\g(\!<\!\!-2)$ we have $[\g_+(\lemw)_\chi, \g_+(\lemw)_\chi] \subseteq \g_+(\lemw)_\chi$, therefore $\{h_1, h_2\} \in S(\g_+) \g_+(\lemw)_\chi$ by the Leibniz rule. To complete the proof of (i) we observe that $N_+ \cap S(\g_+) \g_+(\lemw)_\chi$ is a Poisson ideal of $N_+$, which follows quickly from \eqref{e:glemwstability}.

We now address (ii). Let $f \in N\cap S(\g)_+$ and write $f = f_+ + f_-$ according to the decomposition $S(\g) = S(\g_+) \oplus S(\g)\g_-$. If $g\in G_+(\lez)$ then $g \cdot f_+ - f_+ \in S(\g_+) \cap S(\g) \g(\lemw)_\chi = S(\g_+)\g_+(\lemw)_\chi$ by Lemma~\ref{L:idealcaps} and Lemma~\ref{L:decompprop}. This shows that $\phi(f) = f_+ \in N$. 

Finally, to see that $\phi : N\cap S(\g)_+ \to N_+$ is a Poisson homomorphism it suffices to show that $N \cap S(\g)_+ \cap S(\g)\g_-$ is a Poisson ideal of $N\cap S(\g)_+$. Recall that $\g_- = \{h\cdot x - x \mid h\in H, x\in \g\}$ and $S(\g)_- = \{h\cdot x - x \mid h\in H, x\in S(\g)\}$. If $f\in S(\g)_+$ then $\{f,h\cdot x - x\} = h\cdot \{f,x\} - \{f,x\} \in S(\g)_-$ for any $x\in S(\g)_-$. Lemma~\ref{L:PBWforDiracreduction} shows that the ideals generated by $\g_-$ and $S(\g)_-$ coincide, hence $\{S(\g)_+, \g_-\} \subseteq S(\g) \g_-$, which shows that $S(\g)_+ \cap S(\g) \g_-$ is a Poisson ideal of $S(\g)_+$. This completes the proof of (ii).

The map $\pi$ restricts to a surjection $N \onto S(\g,e)$ by definition. Since the latter map is $H$-equivariant we get $N \cap S(\g)_+ \to S(\g,e)_+$, which proves (iii), whilst (iv) is proven similarly. 

We move on to (v). The kernel of $N\cap S(\g)_+ \to S(\g,e)_+$ is $N \cap S(\g)_+ \cap S(\g) \g(\lemw)_\chi$. This is mapped to $S(\g_+) \cap S(\g) \g(\lemw)_\chi = S(\g_+) \g_+(\lemw)_\chi$ by $\phi$, thanks to Lemma~\ref{L:idealcaps} and Lemma~\ref{L:decompprop}. Finally $S(\g_+) \g_+(\lemw)_\chi$ lies in the kernel of $\pi_+$, which proves (v). Now we take $f \in N \cap S(\g)_+$ such that $\pi(f) \in S(\g,e)S(\g,e)_-$. Using Lemma~\ref{L:decompprop} again we see that $S(\g) / S(\g)\g(\lemw)_\chi$ decomposes as the direct sum of the image of $S(\g_+)$ and the image of $S(\g)\g_-$. Therefore $S(\g,e)_-$, and the ideal which it generates, are contained in the image of $S(\g)\g_-$. It follows immediately that 
\begin{eqnarray*}
& & f \in S(\g) S(\g)_- + S(\g) \g(\lemw)_\chi   \\
& & \phi(f) \in S(\g) \g(\lemw)_\chi \cap S(\g_+) = S(\g_+) \g_+(\lemw)_\chi 
\end{eqnarray*}
Hence $\phi(f)$ is in the kernel of $\pi_+$, completing (vi). This concludes the proof.
\end{proof}

\begin{proofoffirstmain}
Thanks to Proposition~\ref{P:Nispoisson},(vi) we have $\phi : R(S(\g,e), H) \onto S(\g_+, e)$.

Pick an $H$-equivariant map $\theta : \g^e\to S(\g,e)$ satisfying the properties of Theorem~\ref{T:PBW}, and define $\theta_+ : \g_+^e \to S(\g_+, e)$ via 
$$\theta_{+}(x) := \phi (\theta(x) + S(\g,e)_+ \cap S(\g,e) S(\g,e)_-)$$
Property \eqref{e:thetaproperty} for $\theta$ implies \eqref{e:thetaproperty} for $\theta_{+}$. By Theorem~\ref{T:PBW}(2) this implies that $\phi$ is surjective. Applying Proposition~\ref{L:PBWforDiracreduction} we see that $R(S(\g,e), H)$ is a polynomial algebra generated by the image of $\theta(\g^e_+)$ under the map $S(\g,e)_+ \to R(S(\g,e),H)$. It follows that $\phi$ maps a basis of $R(S(\g,e), H)$ to a basis of $S(\g_+, e)$, hence it is an isomorphism. $\hfill\qed$
\end{proofoffirstmain}

\begin{Remark}
It is natural to search for an analogue of Dirac reduction for quantum finite $W$-algebras. Although the na{i}ve approach to quantising $S(\g,e)$ works extremely well (see Section~\ref{ss:finiteWalgebras}), the na{i}ve approach to quantising the Dirac reduction fails.
\end{Remark}

\section{Dirac reduction for shifted Yangians}
\label{s:DracredcutionforshiftedYangians}

\subsection{Poisson algebras by generators and relations}
\label{ss:universalconstructions}
Let $X$ be a set. 
The free Lie algebra $L_X$ on $X$ is the initial object in the category of (complex) Lie algebras generated by $X$ and can be constructed as the Lie subalgebra of the free algebra $\C\langle X \rangle$ generated by the vector space spanned by $X$. 
When $L$ is a Lie algebra generated by $X$ we say that $L$ has relations $Y\subseteq L_X$ if $Y$ generates the kernel of $L_X \onto L$.

The {\it free Poisson algebra generated by $X$} is the initial object in the category of (complex) Poisson algebras generated by $X$. It can be constructed as the symmetric algebra $S(L_X)$ together with its Poisson structure. If there is a Poisson surjection $S(L_X) \onto A$ then we say that {\it $A$ is Poisson generated by $X$}. It is important to distinguish this from $A$ being generated by $X$ as a commutative algebra, as both notions will occur frequently.

We say that a (complex) Poisson algebra $A$ has Poisson generators $X$ and relations $Y \subseteq S(L_X)$  if there is a surjective Poisson homomorphism $S(L_X) \onto A$ and the kernel is the Poisson ideal generated by $Y$.

Let $X$ be a set and $Y \subseteq L_X \subseteq S(L_X)$. Write $I$ (resp. $J$) for the ideal of $S(L_X)$ (resp. $L_X$) generated by $Y$ . It is easy to see that the natural map $S(L_X)\to S(L_X)/I$ induces an isomorphism
\begin{eqnarray}
\label{e:poissonvsliegeneration}
S(L_X/J) \isoto S(L_X)/I.
\end{eqnarray}

\subsection{Chevalley--Serre presentations for shifted current Lie--Poisson algebras} 
\label{ss:CSpresentationsforshifts}
Throughout this section we fix an integer $n  > 0$. Following \cite{BK06, BG07} a {\it shift matrix} is an $n \times n$ array $\sigma = (s_{i,j})_{1\le i,j\le n}$ of non-negative integers with zero on the diagonal, satisfying
\begin{eqnarray}
\label{e:shiftmatrixdefn}
s_{i,k} = s_{i,j} + s_{j,k}
\end{eqnarray}
whenever $i \le j \le k$ or $k \le j \le i$. A shift matrix is said to be {\it symmetric} if it is equal to its transpose. These shift matrices serve two key purposes: they classify certain good gradings for nilpotent elements in general linear Lie algebras \cite{BG07}, and they provide one of the ingredients in the definition of {\it shifted Yangians} which give presentations of finite $W$-algebras in type {\sf A} \cite{BK06}. The symmetric shift matrices correspond to Dynkin gradings.

All examples of shift matrices used in the applications of our results are symmetric, and so for convenience we assume that every shift matrix in this paper is symmetric. Theorems~\ref{T:CSshiftedcurrents1} and \ref{T:shiftedyangianPBW} do not require this hypothesis.

The {\it current algebra} is the Lie algebra $\cc_n := \gl_n \otimes \C[t]$. It has a basis consisting of elements $\{e_{i,j} t^r \mid 1\leq i , j \leq n, \ r \ge 0\}$ where we write $x\otimes t^r = xt^r$ for $x\in \gl_n$ and $r\ge 0$.  For any shift matrix $\sigma = (s_{i,j})_{1\le i,j\le n}$ we define the {\it shifted current algebra} $\cc_n(\sigma)$ to be the subalgebra spanned by 
\begin{eqnarray}
\label{e:shiftedcurrentgens}
\{e_{i,j}t^r \mid 1\le i,j\le n, \ s_{i,j} \le r \}.
\end{eqnarray}
\begin{Lemma}
\label{L:utgenslemma}
Let $\sigma$ be a symmetric shift matrix. The Lie subalgebra $\u_n(\sigma) \subseteq \cc_n(\sigma)$ spanned by elements \eqref{e:shiftedcurrentgens} with $i < j$ is generated as a complex Lie algebra by
\begin{eqnarray}
\label{e:utcurrentsgens}
\{e_{i; r}  \mid 1\le i < n, \ s_{i,i+1} \le r \}
\end{eqnarray}
subject to the relations
\begin{eqnarray}
\label{e:utcurrentrels1}
\big[e_{i; r}, e_{j; s}\big] &=& 0 \text{ \ \ \ for \ \ \ } |i-j| \ne 1,\\
\label{e:utcurrentrels2}
{}\big[e_{i; r+1}, e_{i+1; s}\big] - \big[e_{i; r}, e_{i+1; s+1}\big] &=&0,\\
\label{e:utcurrentrels3}
{}\Big[e_{i; r_1} \big[e_{i; r_2}, e_{j;r_3} \big]\Big] + \Big[e_{i; r_2} \big[e_{i; r_1}, e_{j;r_3} \big]\Big] &=& 0 \text{ \ \ \ for all } |i-j| = 1.
\end{eqnarray}
\end{Lemma}
\begin{proof}
Write $\underline{0}$ for the $n\times n$ zero matrix. It follows from \eqref{e:shiftmatrixdefn} that the linear map $\u_n(\sigma) \to \u_n(\underline{0})$ defined by $e_{i,j}t^r \mapsto e_{i,j}t^{r-s_{i,j}}$ is a Lie algebra isomorphism and so it suffices to prove the current lemma when $\sigma = \underline{0}$.

Let $\h\u_n$ be the Lie algebra with generators \eqref{e:utcurrentsgens} and relations \eqref{e:utcurrentrels1}--\eqref{e:utcurrentrels3}, with $\sigma = \underline{0}$. We inductively define elements $e_{i,j;r} \in \h\u_n$ by setting $e_{i,i+1;r} := e_{i;r}$ and $e_{i,j;r} := [e_{i,j-1; r}, e_{j-1, j; 0}]$ for $1\le i < j \le n$. There is a homomorphism $\h\u_n \onto \u_n(\underline 0)$ given by $e_{i;r} \mapsto e_{i,i+1}t^r$ and, in order to show that it is an isomorphism, we show that $\h\u_n$ is spanned by the elements $\{e_{i,j;r} \mid 1\le i<j\le n, \ 0 \le r\}$.

Following (1)--(7) in the proof of \cite[Lemma~5.8]{BK05} verbatim 
we have for all $i,j,k,l,r,s$
\begin{eqnarray}
\label{e:haturelations}
[e_{i,j;r} , e_{k,l;s}] = \delta_{j,k} e_{i,l;r+s} - \delta_{i,l} e_{k,j;r+s}.
\end{eqnarray}

Define an ascending filtration on $\h\u_n = \bigcup_{d > 0} \F_d \h\u_n$ satisfying $\F_{d} \h\u_n = \sum_{\substack{d_1 + d_2 = d}} [\F_{d_1} \h\u_n, \F_{d_2} \h\u_n]$  by placing $e_{i;r}$ in degree 1. We prove by induction that $\F_d\h\u_n$ is spanned by elements $e_{i,j;r}$ with $j-i \le d$. The base case $d = 1$ holds by definition. For $d_1 + d_2 = d > 1$ we know by the inductive hypothesis that $\F_{d_1} \h\u_n$ and $\F_{d_2} \h\u_n$ are spanned by elements $e_{i,j;r}$. Using \eqref{e:haturelations} we complete the induction, which finishes the proof. 
\end{proof}

\begin{Theorem}
\label{T:CSshiftedcurrents1}
Let $\sigma$ be a symmetric shift matrix and let $\cc_n(\sigma)$ denote the shifted current algebra. Then $S(\cc_n(\sigma))$ is Poisson generated by
\begin{eqnarray}
\label{e:shiftedcurrentfreegens}
\begin{array}{c}
\{d_{i;r} \mid 1\le i \le n, \ 0\le r\}\cup \{e_{i;r} \mid 1\le i < n, \ s_{i,i+1} \le r\} \\ \cup \{f_{i;r} \mid 1\le i < n, \ s_{i+1,i} \le r\} \end{array}
\end{eqnarray}
subject to relations
\begin{eqnarray}
\setlength{\itemsep}{4pt}
\label{e:shiftedcurrentfreerels1}
& & \big\{d_{i;r}, d_{j;s}\big\} = 0,\\
\label{e:shiftedcurrentfreerels3.5}
& & \big\{e_{i; r}, f_{j;s}\big\} = -\delta_{i,j}(d_{i+1; r+s-1} - d_{i; r+s-1})\\
\label{e:shiftedcurrentfreerels2}
& & \big\{d_{i;r}, e_{j;s}\big\} = (\delta_{i,j} - \delta_{i,j+1}) e_{j; r+s},\\
\label{e:shiftedcurrentfreerels3}
& & \big\{d_{i;r}, f_{j;s}\big\} = (\delta_{i,j+1} - \delta_{i,j}) f_{j; r+s},\\
\label{e:shiftedcurrentfreerels4}
&  &\big\{e_{i; r_1}, e_{i+1; r_2 + 1}\big\} - \big\{e_{i; r_1+1}, e_{i+1; r_2}\big\} = 0,\\
\label{e:shiftedcurrentfreerels5}
& & \big\{f_{i; r_1}, f_{i+1; r_2 + 1}\big\} - \big\{f_{i; r_1+1}, f_{i+1; r_2}\big\} = 0,\\
\label{e:shiftedcurrentfreerels6}
& & \big\{e_{i;r}, e_{j;s}\big\} = 0 \text{ for } |i-j| \ne 1,\\
\label{e:shiftedcurrentfreerels7}
& & \big\{f_{i;r}, f_{j;s}\big\} = 0\text{ for } |i-j| \ne 1,\\
\label{e:shiftedcurrentfreerels8}
& & \Big\{e_{i; r_1}, \big\{e_{i; r_2}, e_{j;r_3}\big\}\Big\} + \Big\{e_{i; r_2}, \big\{e_{i; r_1}, e_{j;r_3}\big\}\Big\} = 0 \text{ for } |i-j| = 1,\\
\label{e:shiftedcurrentfreerels9}
& & \Big\{f_{i; r_1}, \big\{f_{i; r_2}, f_{j; r_3}\big\}\Big\} + \Big\{f_{i; r_2}, \big\{f_{i; r_1}, f_{j; r_3}\big\}\Big\} = 0 \text{ for } |i-j| = 1.
\end{eqnarray}
\end{Theorem}

\begin{proof}
By \eqref{e:poissonvsliegeneration} it suffices to show that $\cc_n(\sigma)$ is generated as a Lie algebra by \eqref{e:shiftedcurrentfreegens} subject to relations \eqref{e:shiftedcurrentfreerels1}--\eqref{e:shiftedcurrentfreerels5}. Let $(\hcc_n(\sigma), \{\cdot, \cdot\})$ denote the Lie algebra with these generators and relations. Define a map from the set \eqref{e:shiftedcurrentfreegens} to $\cc_n(\sigma)$ by $d_{i;r} \mapsto e_{i,i}t^r, \ e_{i;r} \mapsto e_{i,i+1}t^r, \ f_{i;r} \mapsto e_{i+1, i} t^r$.
One can easily verify using \eqref{e:haturelations} that this extends to a surjective Lie algebra homomorphism $\h\cc_n(\sigma) \onto \cc_n(\sigma)$. To show that this is an isomorphism it suffices to show that the elements
\begin{eqnarray}
\label{e:hSgens}
\{e_{i,j;r} \mid 1 \le i,j\le n, \ s_{i,j}\le r\} \subseteq \hcc_n(\sigma)
\end{eqnarray}
defined inductively by setting $e_{i,i+1; r} := e_{i,r}$, $e_{i+1, i; r} := f_{i;r}$ and
\begin{eqnarray}
\label{e:shiftedcurrenthighereij}
& e_{i,j;r} := \{e_{i;s_{i,i+1}}, e_{i+1, j; r-s_{i,i+1}}\}  & \text{ for } i < j;\\
\label{e:shiftedcurrenthigherfij}
& e_{i,j;r} := \{f_{i-1; s_{i,i-1}}, f_{i-1, j; r-s_{i,i-1}}\}  & \text{ for } i > j.
\end{eqnarray}
form a spanning set. Using \eqref{e:shiftedcurrentfreerels1},\eqref{e:shiftedcurrentfreerels2}, \eqref{e:shiftedcurrentfreerels3} and a simple inductive argument one can see that $\h\cc_n(\sigma)$ is a direct sum of three subalgebras: the diagonal subalgebra, spanned by the elements $d_{i;r}$, and the upper and lower triangular subalgebras $\u_n^+(\sigma)$ and $\u_n^-(\sigma)$ generated by the elements $e_{i; r}$, respectively by the elements $f_{i; r}$.

In order to complete the proof it suffices to show that $\u_n^+(\sigma)$ and $\u_n^-(\sigma)$ are spanned by the elements defined in \eqref{e:shiftedcurrenthighereij} and \eqref{e:shiftedcurrenthigherfij} respectively. 
Since the argument is identical for $\u_n^+(\sigma)$ and $\u_n^-(\sigma)$ we only need to consider the former, where the claim follows from Lemma~\ref{L:utgenslemma} 
\end{proof}

The following theorem is one of the key stepping stones for understanding the Dirac reduction of the shifted Yangian. The algebra described here is the {\it twisted shifted current Lie--Poisson algebra}.

\begin{Theorem}
\label{T:CSshiftedcurrents2}
If $\sigma$ is any symmetric shift matrix then $S(\cc_n(\sigma))$ admits a Poisson automorphism
\begin{eqnarray}
\label{e:tauonshiftedcurrents}
\tau : e_{i,j}t^r \mapsto (-1)^{r - 1 + s_{i,j}} e_{j,i}t^r.
\end{eqnarray}
The Dirac reduction $R(S(\cc_n(\sigma), \tau) = S(\cc_n(\sigma)^\tau)$ is Poisson generated by
\begin{eqnarray}
\label{e:shiftedcurrentinvariantsfreegens}
\begin{array}{l} \{\eta_{i; 2r-1} \mid 1\le i \le n, \ 0 < r\}  \cup \{\theta_{i; r} \mid 1\le i\le n, \ s_{i,i+1} \le r \} \end{array}
\end{eqnarray}
subject to the relations
\begin{eqnarray}
\label{e:shiftedcurrentinvariantsfreerels1}
\big\{\eta_{i;2r-1}, \eta_{j;2s-1}\big\} & = & 0, \vspace{4pt} \\
\label{e:shiftedcurrentinvariantsfreerels1.5}
\big\{\eta_{i;2r-1}, \theta_{j;s}\big\} & = & (\delta_{i,j} - \delta_{i,j+1})\theta_{j; 2r-1 + s}, \vspace{4pt} \\
\label{e:shiftedcurrentinvariantsfreerels2}
 \big\{\theta_{i; r}, \theta_{j;s}\big\} & = & 0 \text{ if } |i-j| > 1, \vspace{4pt} \\
 \label{e:shiftedcurrentinvariantsfreerels3}
\big\{\theta_{i ; r}, \theta_{i+1; s+1}\big\} &=& \big\{\theta_{i; r+1}, \theta_{i+1; s}\big\},\\
 \label{e:shiftedcurrentinvariantsfreerels4}
 \big\{\theta_{i;r}, \theta_{i; s}\big\} & = & \left\{ \begin{array}{ll} 	2(-1)^{s - 1 + s_{i,i+1}} (\eta_{i;r+s} - \eta_{i+1; r+s}) & \text{ if } r + s \text{ is odd} \\ 0 & \text{ if } r + s \text{ is even}, \end{array} \right.
 \end{eqnarray}
 \begin{eqnarray} 
 \label{e:shiftedcurrentinvariantsfreerels4.5}\Big\{\theta_{i; r_1}, \big\{ \theta_{i; r_2} , \theta_{j;r_3}\big\}\Big\} + \Big\{\theta_{i; r_2}, \big\{ \theta_{i; r_1} , \theta_{j;r_3}\big\}\Big\} = 0 \text{ for } |i-j| = 1, \ r_1 + r_2 \text{ odd}.\\
\nonumber \Big\{\theta_{i; r_1}, \big\{ \theta_{i; r_2} , \theta_{j;r_3}\big\}\Big\} + \Big\{\theta_{i; r_2}, \big\{ \theta_{i; r_1} , \theta_{j;r_3}\big\}\Big\} = \ \ \ \ \ \ \ \ \ \ \ \ \ \ \ \ \ \ \ \ \ \ \ \ \ \ \ \ \ \ \ \ \ \  \\
\label{e:shiftedcurrentinvariantsfreerels5}
2(-1)^{r_2-1+s_{i,i+1}}(\delta_{i+1, j} + \delta_{i,j+1}) \theta_{j; r_1 + r_2 + r_3} \ \ \ \ \ \ \ \ \ \ \ \ \ \ \\  
\nonumber \text{ for } | i - j| = 1,\  r_1 + r_2 \text{ even}.
\end{eqnarray} 
\end{Theorem}
\begin{proof}
Relation \eqref{e:haturelations} implies that $\tau$ gives a Poisson automorphism. By Remark~\ref{R:invariants} we can identify $R(S(\cc_n(\sigma), \tau)$ with $S(\cc_n(\sigma)^\tau)$, so it suffices to check that $S(\cc_n(\sigma)^\tau)$ has the stated Poisson presentation. Use \eqref{e:poissonvsliegeneration} to reduce the claim to a statement about the presentation of $\cc_n(\sigma)^\tau$.

Consider the Lie algebra $\hcc_n(\sigma)^\tau$ which is generated by the set \eqref{e:shiftedcurrentinvariantsfreegens} subject to relations \eqref{e:shiftedcurrentinvariantsfreerels1}--\eqref{e:shiftedcurrentinvariantsfreerels5}. Make the notation $\theta_{i,i+1}t^r := e_{i,i+1}t^r + \tau(e_{i,i+1}t^r) \in \cc_n(\sigma)^\tau$. We define a map from the set \eqref{e:shiftedcurrentinvariantsfreegens} to $\cc_n(\sigma)^\tau$ by sending $\eta_{i;2r-1} \mapsto e_{i,i}t^{2r-1} \in\cc_n(\sigma)^\tau$, and sending $\theta_{i;r} \mapsto \theta_{i,i+1}t^r \in\cc_n(\sigma)^\tau$. One can check that this determines a surjective Lie algebra homomorphism
$\hcc_n(\sigma)^\tau \onto \cc_n(\sigma)^\tau$, indeed, checking that relations \eqref{e:shiftedcurrentinvariantsfreerels1}--\eqref{e:shiftedcurrentinvariantsfreerels5} hold amongst the corresponding elements of $\cc_n(\sigma)^\tau$ is a routine calculation using \eqref{e:haturelations}.

In order to complete the proof of (2) it is sufficient to show that this map is an isomorphism. For $1\le i < j \le n$ and $r \ge s_{i,j}$ we inductively define elements
\begin{eqnarray}
\label{e:someotherlabel}
\theta_{i,j;r} := \{\theta_{i,i+1;s_{i,i+1}}, \theta_{i+1, j; r-s_{i,i+1}}\} \in \hcc_n(\sigma)^\tau. 
\end{eqnarray}
where $\theta_{i,i+1;r} := \theta_{i;r}$. It remains to check that $\hcc_n(\sigma)^\tau$ is spanned by the elements
\begin{eqnarray}
\label{e:shiftedcurrentinvariantsfreegensproof}
\begin{array}{l} \{\eta_{i; 2r-1} \mid 1\le i \le n, \ r  > 0\}  \cup \{\theta_{i,j; r} \mid \ 1\le i < j \le n, \ r \ge s_{i,j}\} \end{array}
\end{eqnarray}

We define a filtration $\hcc_n(\sigma)^\tau = \bigcup_{i > 0} \F_i \hcc_n(\sigma)^\tau$ by placing the generators \eqref{e:shiftedcurrentinvariantsfreegens} in degree 1 and satisfying $\F_{d} \h\cc_n(\sigma)^\tau = \sum_{d_1 + d_2 = d} \{\F_{d_1} \hcc_n(\sigma)^\tau, \F_{d_2} \hcc_n(\sigma)^\tau \}$. By convention $\F_{0} \hcc_n(\sigma)^\tau = 0$. The associated graded Lie algebra $\gr \hcc_n(\sigma)^\tau = \bigoplus_{i > 0} \F_i \hcc_n(\sigma)^\tau / \F_{i-1} \hcc_n(\sigma)^\tau$ is generated by elements 
\begin{eqnarray}
& & \overline{\eta}_{i; 2r-1} := \eta_i t^{2r-1} + \F_{0} \hcc_n(\sigma)^\tau \text{ for } 1\le i \le n, \ 0 \le r \\
\label{e:somelabel}
& & \overline{\theta}_{i;r} := \theta_i t^r + \F_0 \hcc_n(\sigma)^\tau \text{ for } 1\le i < n, \ s_{i,i+1} \le r.
\end{eqnarray}
These generators of $\gr \hcc_n(\sigma)^\tau$ satisfy the top graded components of the relations \eqref{e:shiftedcurrentinvariantsfreerels1}--\eqref{e:shiftedcurrentinvariantsfreerels5}. 

Let $\af$ be an abelian Lie algebra with basis $\{d_i t^r \mid 1\le i \le n, \ 0 < r\}$. Let $\u_n(\sigma)$ be the Lie subalgebra of $\cc_n(\sigma)$ described in Lemma~\ref{L:utgenslemma}. Then $\af \oplus \u_n(\sigma)$ is a Lie algebra with $\af$ an abelian ideal. Comparing the top components \eqref{e:shiftedcurrentinvariantsfreerels1}--\eqref{e:shiftedcurrentinvariantsfreerels5} to \eqref{e:utcurrentrels1}--\eqref{e:utcurrentrels3} we see that there is surjective Lie algebra homomorphism $\af \oplus \u_n(\sigma) \onto \gr \hcc_n(\sigma)^\tau$ defined by $d_i t^{r} \mapsto \overline{\eta}_{i; 2r-1}$ and $e_{i,i+1}t^r \mapsto \overline{\theta}_{i,r}$. The algebra $\af \oplus \u_n(\sigma)$ has basis consisting of element $d_it^r, e_{j,k}t^s$ where $i=1,...,n, \ r >0, \ 1\le j < k \le n, \ s \ge s_{j,k}$. It follows that $\gr \hcc_n(\sigma)^\tau$ is spanned by elements $\overline{\eta}_{i; 2r-1}, \overline{\theta}_{j,k;s}$ where the indexes vary in the ranges specified in \eqref{e:shiftedcurrentinvariantsfreegensproof}, and $\overline{\theta}_{j,k;s}$ defined inductively from \eqref{e:somelabel}, analogously to \eqref{e:someotherlabel}. We deduce that $\hcc_n(\sigma)^\tau$ is spanned by the required elements, which completes the proof.
\end{proof}

\subsection{The semiclassical shifted Yangian}

In this section we fix $n > 0$, a shift matrix $\sigma$ of size $n$ and $\ve \in \{\pm 1\}$. The {\em (semiclassical) shifted Yangian} $y_n(\sigma)$ is the Poisson algebra generated by the set
\begin{equation}\label{e:Ybasicgens}
\begin{array}{c} \{d_i^{(r)} \mid 1\leq i \leq n, \ 0 < r\} \cup \{e_{i}^{(r)} \mid 1\leq i < n, \ s_{i,i+1} < r \} \\ \cup \, \{f_{i}^{(r)} \mid 1\leq i < n, \ s_{i+1,i} < r\} \end{array}
\end{equation}
subject to the following relations
\begin{eqnarray}
\setlength{\itemsep}{4pt}
 & & \big\{d_i^{(r)}, d_j^{(s)}\big\} =  0,\label{e:r2}\\
& & \big\{e_i^{(r)},f_j^{(s)}\big\} = -\delta_{i,j}
\sum_{t=0}^{r+s-1} d_{i+1}^{(r+s-1-t)}\widetilde d_{i}^{(t)},\label{e:r3}
\end{eqnarray}
\begin{eqnarray}
\big\{d_i^{(r)}, e_j^{(s)}\big\} &=& (\delta_{i,j}-\delta_{i,j+1})
\sum_{t=0}^{r-1} d_i^{(t)} e_j^{(r+s-1-t)},\label{e:r4}\\
\big\{d_i^{(r)}, f_j^{(s)}\big\} &=& (\delta_{i,j+1}-\delta_{i,j})
\sum_{t=0}^{r-1} f_j^{(r+s-1-t)}d_i^{(t)} ,\label{e:r5}
\end{eqnarray}
\begin{eqnarray}
\big\{e_i^{(r)}, e_i^{(s)}\big\} &=&
\sum_{t=r}^{s-1} e_i^{(t)} e_i^{(r+s-1-t)}
\hspace{11mm}\text{if $r < s$},\label{e:r6}\\
\big\{f_i^{(r)}, f_i^{(s)}\big\} &=&
\sum_{t=s}^{r-1}
f_i^{(r+s-1-t)} f_i^{(t)}\hspace{11mm}\text{if $r > s$},\label{e:r7}
\end{eqnarray}
\begin{eqnarray}
\big\{e_i^{(r+1)}, e_{i+1}^{(s)}\big\}&-& \big\{e_i^{(r)}, e_{i+1}^{(s+1)}\big\}=
e_i^{(r)} e_{i+1}^{(s)},\label{e:r8}\\
\big\{f_i^{(r)}, f_{i+1}^{(s+1)}\big\}&-& \big\{f_i^{(r+1)}, f_{i+1}^{(s)}\big\} =
f_{i+1}^{(s)} f_i^{(r)},\label{e:r9}
 \end{eqnarray}
 \begin{eqnarray}
\big\{e_i^{(r)}, e_j^{(s)}\big\} &=& 0 \quad \text{ if }|i-j|> 1,\label{e:r10}\\
\big\{f_i^{(r)}, f_j^{(s)}\big\} &=& 0 \quad\text{ if }|i-j|>
1,\label{e:r11}\\
\Big\{e_i^{(r)}, \big\{e_i^{(s)}, e_j^{(t)}\big\}\Big\} +
\Big\{e_i^{(s)}, \big\{e_i^{(r)}, e_j^{(t)}\big\}\Big\} &=& 0 \quad\text{ if }|i-j|=1,\label{e:r12}\\
\Big\{f_i^{(r)}, \big\{f_i^{(s)}, f_j^{(t)}\big\}\Big\} +
\Big\{f_i^{(s)}, \big\{f_i^{(r)}, f_j^{(t)}\big\}\Big\} &=& 0 \quad\text{ if
}|i-j|=1.\label{e:r13}
\end{eqnarray}
for all admissible $i,j,r,s,t$. In these relations, the notation $d_i^{(0)} = \widetilde{d}_i^{(0)} := 1$ is used, and the elements $\widetilde{d}_i^{(r)}$ for $r > 0$ are defined recursively by
\begin{eqnarray}
\label{e:dtildedefinition}
\widetilde{d}_i^{(r)} := -\sum_{t=1}^r d_i^{(t)} \widetilde d_i^{(r-t)}
\end{eqnarray}

In order to describe the structure of $y_n(\sigma)$ as a commutative algebra we make the notation
\begin{eqnarray}
e_{i,i+1}^{(r)} := e_i^{(r)} \text{ for } 1\le i < n, \ s_{i,i+1} < r ;\\ [3pt]
f_{i,i+1}^{(r)} := f_i^{(r)} \text{ for } 1\le i \le n, \ s_{i+1,i} < r ,
\end{eqnarray}
and inductively define
\begin{eqnarray}
\label{e:eijrels}
e_{i,j}^{(r)} := \{e_{i,j-1}^{(r-s_{j-1, j})}, e_{j-1}^{(s_{j-1, j} + 1)}\} \text{ for } 1\le i < j\le n, \ s_{i,j} < r ;\\ [3pt]
\label{e:fijrels}
f_{i,j}^{(r)} := \{f_{j-1}^{(s_{j,j-1} + 1)}, f_{i,j-1}^{(r-s_{j,j-1})}\}  \text{ for } 1\le i < j\le n, \ s_{j,i} < r.
\end{eqnarray}

The shifted Yangian admits a Poisson grading $y_n(\sigma) = \bigoplus_{r \ge 0} y_n(\sigma)_r$, which we call the {\it canonical grading}. It places $d_i^{(r)}, e_{i}^{(r)}, f_{i}^{(r)}$ in degree $r$ and the bracket lies in degree $-1$, meaning
$\{\cdot, \cdot\} : y_n(\sigma)_r \times y_n(\sigma)_s \to y_n(\sigma)_{r+s-1}$.
There is also an important Poisson filtration $y_n(\sigma) = \bigcup_{i \ge 0} \F_i y_n(\sigma)$, called the {\it loop filtration}, which places $d_i^{(r)}, e_{i}^{(r)}, f_{i}^{(r)}$ in degree $r-1$. With respect to this filtration, the bracket is in degree $0$ so that $\F_r y_n(\sigma) \times \F_s y_n(\sigma) \to \F_{r+s}y_n(\sigma)$. The associated graded Poisson algebra is denoted $\gr y_n(\sigma)$.

The following theorem is a semiclassical analogue of \cite[Theorem~2.1]{BK06}, which we ultimately deduce from the noncommutative setting.
\begin{Theorem}
\label{T:shiftedyangianPBW}
Let $\sigma$ be a symmetric shift matrix. There is a Poisson isomorphism $S(\cc_n(\sigma)) \isoto \gr y_n(\sigma)$ defined by
\begin{eqnarray}
\label{e:loopiso}
\begin{array}{rcl}
e_{i;r-1} \longmapsto e_i^{(r)} + \F_{r-2}y_n(\sigma) & \text{ for } & 1\le i<n, \ s_{i,i+1} < r;\\
f_{i;r-1} \longmapsto f_i^{(r)} + \F_{r-2}y_n(\sigma)& \text{ for } &1\le i<n, \ s_{i+1,i} < r;\\
d_{i;r-1} \longmapsto d_i^{(r)} + \F_{r-2}y_n(\sigma) & \text{ for } & 1\le i \le n, \ r<0.
\end{array}
\end{eqnarray}
As a consequence $y_n(\sigma)$ is isomorphic to the polynomial algebra on infinitely many variables
\begin{equation}\label{e:Ygens}
\begin{array}{c} \{d_i^{(r)} \mid 1\leq i \leq n, \ 0 < r\} \cup \{e_{i,j}^{(r)} \mid 1\leq i < j\le n, \ s_{i,j} < r \} \\ \cup \, \{f_{i,j}^{(r)} \mid 1\leq i < j \le n, \ s_{j,i} < r\}.
\end{array}
\end{equation}
\end{Theorem}
\begin{proof}
Comparing  the top graded components of the relations \eqref{e:r2}--\eqref{e:r13} with respect to the loop filtration, with relations \eqref{e:shiftedcurrentfreerels1}--\eqref{e:shiftedcurrentfreerels9}, it is straightforward to see that \eqref{e:loopiso} gives a surjective Poisson homomorphism. To prove the theorem we demonstrate that the ordered monomials in the elements \eqref{e:Ygens} are linearly independent.

Consider the set
$$X := \{E_i^{(r)}, F_i^{(s)}, D_j^{(t)} \mid 1 \le i < n, \ 1\le j \le n,\ s_{i,i+1} < r, \ s_{i+1, i}< s, \ 0< t\}.$$
In  \cite[\textsection 2]{BK06} the shifted Yangian $Y_n(\sigma)$ is defined as a quotient of the free algebra $\C\langle X \rangle$ by the ideal generated by the relations \cite[(2.4)--(2.15)]{BK06}. Let $L := L_X$ be the free Lie algebra on $X$ and define a grading $L = \bigoplus_{i \ge 0} L_i$ by placing $E_i^{(r)}, F_i^{(r)}, D_i^{(r)}$ in degree $r-1$. Then we place a filtration on the enveloping algebra $U(L)$ so that $L_i$ lies in degree $i+1$. By the PBW theorem for $U(L)$ we see that $\gr U(L) \cong S(L)$.
 
The universal property of $U(L)$ ensures that there is a surjective algebra homomorphism $U(L) \onto \C\langle X \rangle$ and by \cite[I, Ch. IV, Theorem 4.2]{Ser06} this is an isomorphism. Identifying these algebras, the filtration on $U(L)$ descends to $Y_n(\sigma) = \bigcup_{i\ge 0} \F'_i Y_n(\sigma)$, and this resulting filtration is commonly referred to as the canonical filtration \cite[\textsection 5]{BK06}. The associated graded algebra $\gr Y_n(\sigma)$ is equipped with a Poisson structure in the usual manner. Comparing relations \eqref{e:r2}--\eqref{e:r13} with the top graded components of relations \cite[(2.4)--(2.15)]{BK06} we see that the Poisson surjection $S(L) \onto \gr Y_n(\sigma)$ factors through $S(L) \to y_n(\sigma)$. As a result there is a surjective Poisson homomorphism $\pi: y_n(\sigma) \onto \gr Y_n(\sigma)$ given by $e_i^{(r)} \mapsto E_i^{(r)} + \F'_{r-1} Y_n(\sigma), f_i^{(r)} \mapsto F_i^{(r)} + \F'_{r-1} Y_n(\sigma), d_i^{(r)} \mapsto D_i^{(r)} + \F'_{r-1} Y_n(\sigma)$. Following \cite[(2.18), (2.19)]{BK06} we introduce elements $E_{i,j}^{(r)}, F_{i,j}^{(r)}$ of $Y_n(\sigma)$ lying in filtered degree $r$. By the definition of the filtration and the elements \eqref{e:eijrels}, \eqref{e:fijrels} we have
 \begin{eqnarray*}
 \label{e:piandthegenerators}
 \begin{array}{rcl}
 \pi(e_{i,j}^{(r)}) & = & E_{i,j}^{(r)} + \F'_{r-1} Y_n(\sigma);\\
 \pi(f_{i,j}^{(r)}) & = & F_{i,j}^{(r)}  + \F'_{r-1} Y_n(\sigma).\\
 \end{array}
 \end{eqnarray*}
By \cite[Theorem~2.3(iv)]{BK06} the ordered monomials in $E_{i,j}^{(r)}, F_{i,j}^{(r)}, D_i^{(r)}$
are linearly independent in $Y_n(\sigma)$ and so we deduce that the images of these monomials in $\gr Y_n(\sigma)$ are linearly independent. This completes the proof.
\end{proof}

We record two formulas for future use, which are semiclassical analogues of \cite[(4.32)]{BT18}
\begin{Lemma}
The following hold for $i=1,...,n$, $j = 1,...,n-1$, $r > 0$ and $s > s_{j,j+1}$:
\begin{eqnarray}
\label{e:tildedone}
\{\td_i^{(r)}, e_j^{(s)}\} = (\delta_{i,j+1}-\delta_{i,j}) \sum_{t=0}^{r-1} \td_i^{(t)} e_j^{(r+s-1-t)},\\
\label{e:tildedonf}
\{\td_i^{(r)}, f_j^{(s)}\big\} = (\delta_{i,j}-\delta_{i,j+1}) \sum_{t=0}^{r-1} f_j^{(r+s-1-t)}\td_i^{(t)}.
\end{eqnarray}
\end{Lemma}
\begin{proof}
We only sketch \eqref{e:tildedone}, as the proof of \eqref{e:tildedonf} is almost identical. The argument is by induction based on \eqref{e:dtildedefinition}. Note that $\td_i^{(1)} = -d_i^{(1)}$ and so \eqref{e:tildedone} is equivalent to \eqref{e:r4} in this case. We have $\{\td_i^{(r)}, e_j^{(s)}\} = \{\sum_{t=0}^{r-1} d_i^{(r-t)} \td_i^{(t)}, e_j^{(s)}\}$ and relation \eqref{e:r4} together with the inductive hypothesis imply that the coefficient of $e_j^{(s+t)}$ is $(\delta_{i,j+1}- \delta_{i,j}) (d_i^{(r-1-t)} + \sum_{m=t+1}^{r-1} (\td_i^{(r-m)} d_i^{(m-t-1)} - d_i^{(r-m)} \td_i^{(m-t-1)}))$. Using $d_i^{(0)} = \td_i^{(0)} = 1$ we see that all quadratic terms cancel, whilst the linear term equates to $(\delta_{i,j+1} - \delta_{i,j})\td_i^{(r-1-t)}$, which concludes the induction.
\end{proof}

\subsection{The Dirac reduction of the shifted Yangian}
\label{ss:PDreductionofyangians}
In this section we suppose that $\sigma$ is symmetric. Examining the relations \eqref{e:r2}--\eqref{e:r13} we see that there is unique involutive Poisson automorphism $\tau$ of $y_n(\sigma)$ determined by
\begin{eqnarray}
\label{e:tauon}
\begin{array}{rcl}
\tau(d_i^{(r)})& := & (-1)^{r} d_{i}^{(r)}; \\ [3pt]
\tau(e_{i}^{(r)}) & := & (-1)^{r + s_{i,i+1}} f_i^{(r)}; \\ [3pt]
\tau(f_i^{(r)}) & := & (-1)^{r + s_{i+1, i}} e_i^{(r)} .
\end{array}
\end{eqnarray}
Our present goal is to give a complete description of the Dirac reduction $R(y_n(\sigma), \tau)$.

For $i=1,...,n$ and $r>s_{i,i+1}$ we write
\begin{eqnarray}
\label{e:ehatandcheck}
\begin{array}{rcl}
\he_i^{(r)} & := & e_i^{(r)} + (-1)^{r + s_{i,i+1}} f_i^{(r)} \in y_n(\sigma); \vspace{4pt}\\
\ce_i^{(r)} & := & e_i^{(r)} - (-1)^{r + s_{i,i+1}} f_i^{(r)} \in y_n(\sigma),
\end{array}
\end{eqnarray}
so that $\he_i^{(r)}$ and $d_j^{(2s)}$ are $\tau$-invariants, whilst $\ce_i^{(r)}$ and $d_j^{(2s-1)}$ each span a non-trivial representation of the cyclic group of order 2 generated by $\tau$. 
We can recover $e_i^{(r)}$ and $f_i^{(r)}$ from $\he_i^{(r)}$ and $\ce_i^{(r)}$ via
\begin{eqnarray}
\label{e:recoverhatcup}
\begin{array}{l}
e_i^{(r)} = \frac{1}{2} (\he_i^{(r)} + \ce_i^{(r)});\\ f_i^{(r)} = \frac{1}{2}(-1)^{r + s_{i,i+1}} (\he_i^{(r)} - \ce_i^{(r)}).
\end{array}
\end{eqnarray}
Now  let $\cy_n(\sigma)$ be the ideal of $y_n(\sigma)$ generated by
$$\{\ce_i^{(r)}, d_j^{(2s-1)} \mid i=1,...,n-1,\ j=1,...,n, \ s_{i,i+1}  < r, \ 0 < s\}.$$
Also write $\cy_n(\sigma)^\tau := \cy_n(\sigma) \cap y_n(\sigma)^\tau$.
By Lemma~\ref{L:PBWforDiracreduction} we see that $R(y_n(\sigma), \tau) = y_n(\sigma)^\tau / \cy_n(\sigma)^\tau$ is Poisson generated by elements
\begin{eqnarray}
\label{e:thetaandeta}
\begin{array}{rcl}
\theta_i^{(r)} &:=& \he_i^{(r)} + \cy_n(\sigma)^\tau \text{ for } i=1,...,n-1, \ s_{i,i+1} < r ,\vspace{4pt}\\
\eta_j^{(2r)} &:=& d_j^{(2r)} + \cy_n(\sigma)^\tau \text{ for } i=1,...,n, \ 0 < r ,
\end{array}
\end{eqnarray}
with Poisson brackets induced by the bracket on $y_n(\sigma)$. Furthermore $R(y_n(\sigma), \tau)$ is generated as a commutative algebra by elements\begin{eqnarray}
\label{e:theRgenerators}
\{\eta_i^{(2r)} \mid 1\le i \le n, \ 0 < r\} \cup \{\theta_{i,j}^{(r)} \mid 1\le i < j \le n, s_{i,j} < r\}
\end{eqnarray}
where the $\theta_{i,j}^{(r)} := e_{i,j}^{(r)} + (-1)^{r+s_{i,j}} f_{i,j}^{(r)} + \cy_n(\sigma)^\tau \in R(y_n(\sigma), \tau)$. Using an inductive argument and \eqref{e:r3} we see the elements $\theta_{i,j}^{(r)}$ can also be defined via the following recursion
\begin{eqnarray}
\label{e:higherthetas}
\begin{array}{rcl}
\theta_{i,i+1}^{(r)} := \theta_i^{(r)} & \text{ for } & 1\le i < n, \ 0 < r \vspace{4pt} \\
\theta_{i,j}^{(r)} := \{\theta_{i,j-1}^{(r-s_{j-1, j})}, \theta_{j-1}^{(s_{j-1, j} + 1)}\} & \text{ for } & 1\le i < j\le n, \ s_{i,j} < r.\end{array}
\end{eqnarray}
The next lemma can be deduced from \eqref{e:r4}, \eqref{e:r5}, \eqref{e:tildedone}, \eqref{e:tildedonf}.
\begin{Lemma}
The following equalities hold in $R(y_n(\sigma),\tau)$:
\begin{eqnarray}
\label{e:dcemodideal}
\{d_i^{(2r-1)}, \ce_j^{(s)}\} + \cy_n(\sigma)^\tau = (\delta_{i,j} - \delta_{i,j+1})\sum_{t = 0}^{r-1} \eta_i^{(2t)} \theta_j^{(2(r-t-1)+s)}\\
\label{e:tdcemodideal}
\{\td_i^{(2r-1)}, \ce_j^{(s)}\} + \cy_n(\sigma)^\tau = (\delta_{i,j+1} - \delta_{i,j})\sum_{t = 0}^{r-1} \teta_i^{(2t)} \theta_j^{(2(r-t-1)+s)}
\end{eqnarray}
\end{Lemma}

Thanks to Theorem~\ref{T:shiftedyangianPBW} we see that $R(y_n(\sigma), \tau)$ comes equipped with the {\it canonical grading} $R(y_n(\sigma), \tau) = \bigoplus_{i \ge 0} R(y_n(\sigma), \tau)_i$ which places $\theta_i^{(r)}, \eta_i^{(r)}$ in degree $r$ and the Poisson bracket in degree $-1$. Another crucial feature is the {\it loop filtration} $R(y_n(\sigma), \tau) = \bigcup_{i\ge 0} \F_iR(y_n(\sigma), \tau)$ which places $\theta_i^{(r)}, \eta_i^{(r)}$ in degree $r-1$ and the bracket in degree $0$. Both of these structures are naturally inherited from $y_n(\sigma)$.

We define the following useful symbol for $r,s \in \Z$
\begin{eqnarray}
\label{e:definevarpi}
\varpi_{r,s} := (-1)^r - (-1)^s = \left\{
\begin{array}{cl} 2 & \text{ if } r \text{ even}, \ s\text{ odd},\\
0 & \text{ if } r + s \text{ even},\\
-2 &  \text{ if } r\text{ odd}, \ s \text{ even}. \end{array} \right.
\end{eqnarray}

The following is our main structural result on the Dirac reduction of the shifted Yangian.
\begin{Theorem}
\label{T:PDyangian}
Let $\sigma$ be a symmetric shift matrix, let $y_n(\sigma)$ denote the corresponding semiclassical shifted Yangian and let $\tau$ denote the automorphism of $y_n(\sigma)$ introduced in \eqref{e:tauon}. The Dirac reduction $R(y_n(\sigma), \tau)$ is Poisson generated by elements
\begin{eqnarray}
\label{e:DYgens}
\{\eta_i^{(2r)} \mid 1\le i \le n, \ 0 < r \} \cup \{ \theta_i^{(r)} \mid 1\le i < n, \ s_{i,i+1} < r \}
\end{eqnarray}
together with the following relations
\begin{eqnarray}
\label{e:dyrel1}
& & \{\eta_i^{(2r)}, \eta_j^{(2s)}\} = 0 \\ 
\label{e:dyrel2}
& &\big\{\eta_i^{(2r)}, \theta_j^{(s)}\big\} = (\delta_{i,j} - \delta_{i,j+1}) \sum_{t=0}^{r-1} \eta_i^{(2t)} \theta_j^{(2r+s -1- 2t)}
\end{eqnarray}
\begin{eqnarray}
\label{e:dyrel34}
\{\theta_i^{(r)}, \theta_i^{(s)}\} = \frac{1}{2} \sum_{t=r}^{s-1} \theta_i^{(t)} \theta_i^{(r+s-1-t)} + (-1)^{s_{i,i+1}}\varpi_{r,s} \sum_{t=0}^{(r+s-1)/2} \eta_{i+1}^{(r+s-1-2t)} \teta_{i}^{2t} & \text{ for } & r < s.\\
\label{e:dyrel5}
\big\{\theta_i^{(r+1)}, \theta_{i+1}^{(s)}\big\} - \big\{\theta_i^{(r)}, \theta_{i+1}^{(s+1)}\big\} = \frac{1}{2}\theta_i^{(r)} \theta_{i+1}^{(s)} & &
\end{eqnarray}
\begin{eqnarray}
\label{e:dyrel6}
\big\{\theta_i^{(r)}, \theta_j^{(s)}\big\} = 0 & & \text{ for } |i - j| > 1\\
\label{e:dyrel7}
\Big\{\theta_i^{(r)}, \big\{\theta_i^{(s)}, \theta_j^{(t)}\big\}\Big\} + \Big\{\theta_i^{(s)}, \big\{\theta_i^{(r)}, \theta_j^{(t)}\big\}\Big\} = 0 & & \text{ for } |i - j| =1 \text{ and } r + s \text{ odd}
\end{eqnarray}
\begin{eqnarray}
\label{e:dyrel8}
& &\Big\{\theta_i^{(r)}, \big\{\theta_i^{(s)}, \theta_j^{(t)}\big\}\Big\} + \Big\{\theta_i^{(s)}, \big\{\theta_i^{(r)}, \theta_j^{(t)}\big\}\Big\} =  \\
& & \nonumber  \ \ \ \ \ \ \ \ \ \  2(-1)^{s+s_{i,i+1}-1} \delta_{i,j+1} \sum_{m_1=0}^{m-1} \sum_{m_2=0}^{m_1} \eta_{i+1}^{(2(m-m_1-1))}  \teta_i^{(2m_2)} \theta_j^{(2(m_1-m_2) + t)} \\
& & \nonumber \ \ \ \ \ \ \ \ \ \ \ \ \ \ \ \ \ \ \ \ + 2(-1)^{s+s_{i,i+1}-1} \delta_{i+1,j} \sum_{m_1=0}^{m-1} \sum_{m_2=0}^{m-m_1-1} \teta_i^{(2m_1)}\eta_{i+1}^{(2m_2)} \theta_j^{(2(m-m_1-m_2 - 1) + t)} \\
& & \nonumber \ \ \ \ \ \ \ \ \ \ \ \ \ \ \ \ \ \ \ \ \ \ \ \ \ \ \ \ \ \ \text{ for } |i - j| =1 \text{ and } r + s = 2m \text{ even}
\end{eqnarray}
where we adopt the convention $\eta_i^{(0)} = \teta_i^{(0)} = 1$ and the elements $\teta_i^{(2r)} $ are defined via the recursion
\begin{eqnarray}
\label{e:tetatwisteddefinition}
\teta_i^{(2r)} := -\sum_{t=1}^r \eta_i^{(2t)} \teta_i^{(2r-2t)}.
\end{eqnarray}
\end{Theorem}
\begin{proof}
First we show that the elements $\teta_i^{(2r)} := \td_i^{(2r)} + \cy_n(\sigma)^\tau \in R(y_n(\sigma), \tau)$ satisfy the recursion \eqref{e:tetatwisteddefinition}. By Theorem~\ref{T:shiftedyangianPBW} the subalgebra of $y_n(\sigma)$ generated by $\{d_i^{(r)} \mid i=1,...,n, \ r > 0\}$ is a graded polynomial ring with $d_i^{(r)}$ in degree $r$. By induction $\td_i^{(r)}$ lies in degree $r$. This forces $\td_i^{(2r-1)} \in \cy_n(\sigma)$ for all $r > 0$ and so \eqref{e:dtildedefinition} equals \eqref{e:tetatwisteddefinition} modulo $\cy_n(\sigma)^\tau$, confirming the claim.

We now deduce relations \eqref{e:dyrel1}--\eqref{e:dyrel8} from \eqref{e:r2}--\eqref{e:r13}. First of all \eqref{e:dyrel1} follows immediately from \eqref{e:r2}. By \eqref{e:r4} and \eqref{e:r5} we have
\begin{eqnarray*}
\{d_i^{(2r)}, \he_j^{(s)} \} & = & (\delta_{i,j} - \delta_{i, j+1})\sum_{t=0}^{2r-1} d_i^{(t)} \Big(e_j^{(2r+s - 1 -t)} - (-1)^{s+s_{j,j+1}} f_j^{(2r+s-1-t)}\Big)\\
& = & (\delta_{i,j} - \delta_{i, j+1}) \Big( \sum_{t=0}^{r-1} d_i^{(2t)} \he_j^{(2r+s-1-2t)} + \sum_{t=0}^{r-1} d_i^{(2t+1)} \ce_j^{(2r + s - 2t-2)} \Big).
\end{eqnarray*}
Projecting to $R(y_n(\sigma), \tau) = y_n(\sigma)^\tau / \cy_n(\sigma)^\tau$ we see that the elements \eqref{e:thetaandeta} satisfy \eqref{e:dyrel2}.

Note that \eqref{e:r3} implies that
\begin{eqnarray}
\label{e:swaptheindeceseandf}
\{e_i^{(r)}, f_i^{(s)}\} = \{e_i^{(s)}, f_i^{(r)}\} \text{ for } r, s > s_{i,i+1}.
\end{eqnarray}
Together with \eqref{e:r6} and \eqref{e:r7} we deduce for all $r < s$
\begin{eqnarray*}
\{\he_i^{(r)}, \he_i^{(s)}\} &=& \{e_i^{(r)}, e_i^{(s)}\} + (-1)^{r + s} \{f_i^{(r)}, f_i^{(s)}\}
+(-1)^{s+s_{i,i+1}}\{e_i^{(r)}, f_i^{(s)}\} - (-1)^{r+s_{i,i+1}}\{e_i^{(s)}, f_i^{(r)}\} 
\end{eqnarray*}
Using relations \eqref{e:r6} and \eqref{e:r7} and substituting in \eqref{e:recoverhatcup} we see that
\begin{eqnarray}
\label{e:someballs}
\{e_i^{(r)}, e_i^{(s)}\} + (-1)^{r+s}\{f_i^{(r)}, f_i^{(s)}\} = \frac{1}{2} \sum_{t=r}^{s-1} \he_i^{(t)} \he_i^{(r+s-1-t)} \mod \cy_n(\sigma)^\tau.
\end{eqnarray}
Similarly applying \eqref{e:r3} and substituting \eqref{e:recoverhatcup} we deduce
\begin{eqnarray}
(-1)^{s+s_{i,i+1}}\{e_i^{(r)}, f_i^{(s)}\} - (-1)^{r+s_{i,i+1}}\{e_i^{(s)}, f_i^{(r)}\} = (-1)^{s_{i,i+1}}\varpi_{r,s} \sum_{t=0}^{r+s-1} d_{i+1}^{(r+s-1-2t)} \td_{i}^{2t}.
\end{eqnarray}
Projecting the right hand side into $y_n(\sigma)^\tau / \cy_n(\sigma)^\tau$, we obtain the second summation on the right hand side of \eqref{e:dyrel34}. Combining with \eqref{e:someballs} we have now checked relation \eqref{e:dyrel34}.

Using \eqref{e:shiftmatrixdefn}, \eqref{e:r3}, \eqref{e:r8}, \eqref{e:r9} we calculate
\begin{eqnarray*}
\{\he_i^{(r+1)}, \he_{i+1}^{(s)}\} - \{\he_i^{(r)}, \he_{i+1}^{(s+1)}\} & = & \{e_i^{(r+1)}, e_{i+1}^{(s)}\} - \{e_i^{(r)}, e_{i+1}^{(s+1)}\} \\ & & + (-1)^{r+s+1+s_{i,i+2}}(\{f_{i}^{(r+1)}, f_{i+1}^{(s)}\} - \{f_i^{(r)}, f_{i+1}^{(s+1)} \})\\
&=& e_i^{(r)} e_{i+1}^{(s)} + (-1)^{r+s+s_{i,i+2}} f_{i+1}^{(s)} f_i^{(r)}
\end{eqnarray*}
Substituting in \eqref{e:recoverhatcup} we see that this is equal to the right hand side of \eqref{e:dyrel5} modulo $\cy_n(\sigma)^\tau$.

Finally take $|i-j| = 1$. Expanding in terms of generators \eqref{e:Ybasicgens} and applying relations \eqref{e:r3}, \eqref{e:r12}, \eqref{e:r13} together with the Jacobi identity and \eqref{e:swaptheindeceseandf} we obtain
\begin{eqnarray*}
& & \{\he_i^{(r)}, \{\he_i^{(s)}, \he_j^{(t)}\}\} +\{\he_i^{(s)}, \{\he_i^{(r)}, \he_j^{(t)}\}\} \\
& & \ \ \ \ = (-1)^{s+t + s_{i,i+1} + s_{j,j+1}} \{e_i^{(r)}, \{f_i^{(s)}, f_j^{(t)}\}\} + (-1)^{r + t + s_{i,i+1} + s_{j,j+1}} \{e_i^{(s)}, \{f_i^{(r)}, f_j^{(t)}\}\} \\
& & \ \ \ \ \ \ \ \ \  + (-1)^{r + s_{i,i+1}} \{f_i^{(r)}, \{e_i^{(s)}, e_j^{(t)}\}\} + (-1)^{s + s_{i,i+1}} \{f_i^{(s)}, \{e_i^{(r)}, e_j^{(t)}\}\}\\
& & \ \ \ \ = (-1)^{s+t + s_{i,i+1} + s_{j,j+1}}(1 + (-1)^{r+s}) \{\{e_i^{(r)}, f_i^{(s)}\}, f_j^{(t)}\}- (-1)^{r + s_{i,i+1}}(1 + (-1)^{r+s}) \{\{e_i^{(r)}, f_i^{(s)}\}, e_j^{(t)}\}\\
\end{eqnarray*}
If $r + s$ is odd this vanishes, which proves \eqref{e:dyrel7}. Now assume that $r + s=2m$ is even. Using \eqref{e:r3} and Jacobi the last line of the previous equation reduces to
\begin{eqnarray*}
& & 2(-1)^{s+s_{i,i+1} -1}\big\{\{e_i^{(r)}, f_i^{(r)}\}, \ce_j^{(t)}\big\}  = 2(-1)^{s+s_{i,i+1} -1}\big\{\sum_{m_1=0}^{2m-1} d_{i+1}^{(2m-1-m_1)} \td_{i}^{(m_1)}, \ce_j^{(t)}\big\} \\
& & \ \ \ \ = 2(-1)^{s + s_{i,i+1}-1} \sum_{m_1=0}^{2m - 1} \big(d_{i+1}^{(2m-1-m_1)} \{\td_{i}^{(m_1)}, \ce_j^{(t)}\} + \{d_{i+1}^{(2m-1-m_1)}, \ce_j^{(t)}\} \td_{i}^{(m_1)}\big).
\end{eqnarray*}
Using the fact that $d_i^{(2r-1)} \in \cy_n(\sigma)^\tau$ for all $i, r$ we simplify this expression modulo $\cy_n(\sigma)^\tau$ to get
\begin{eqnarray*}
2(-1)^{s + s_{i,i+1}-1} \left( \sum_{m_1=0}^{m-1} d_{i+1}^{(2m-2m_1-2)} \{\td_{i}^{(2m_1+1)}, \ce_j^{(t)}\} + \sum_{m_1=0}^{m-1}\{d_{i+1}^{(2(m-m_1)-1)}, \ce_j^{(t)}\} \td_{i}^{(2m_1)}\right) + \cy_n(\sigma)^\tau
\end{eqnarray*}
Finally using \eqref{e:dcemodideal} and \eqref{e:tdcemodideal} this expression coincides with the right hand side of \eqref{e:dyrel8}.

We have shown that the generators \eqref{e:DYgens} satisfy \eqref{e:dyrel1}--\eqref{e:dyrel8}, and it remains to show that these are a complete set of relations.

Let $\h R(y_n(\sigma), \tau)$ denote the Poisson algebra with generators \eqref{e:DYgens} and relations \eqref{e:dyrel1}--\eqref{e:dyrel8}. We have shown that there is a Poisson homomorphism $\h R(y_n(\sigma), \tau)\onto R(y_n(\sigma), \tau)$ sending the elements \eqref{e:DYgens} to the elements \eqref{e:thetaandeta} with the same names. To complete the proof we show that this map sends a spanning set to a basis.

We define a loop filtration on $\h R(y_n(\sigma),\tau) = \bigcup_{i\ge 0} \F_i \h R(y_n(\sigma),\tau)$ by placing $\theta_i^{(r)}, \eta_i^{(r)}$ in degree $r-1$ and the bracket in degree $0$. Examining the top filtered degree pieces of the relations \eqref{e:dyrel1}--\eqref{e:dyrel8} with respect to the loop filtration, we see from Theorem~\ref{T:CSshiftedcurrents2} that there is a surjective Poisson homomorphism $S(\cc_n(\sigma)^\tau) \onto \gr \h R(y_n(\sigma), \tau)$. We deduce that $\gr \h R(y_n(\sigma), \tau)$ is generated as a commutative algebra by elements $\theta_{i,j}^{(r)} + \F_{r-2} \h R(y_n(\sigma), \tau), \eta_i^{(2r)} + \F_{2r-2} \h R(y_n(\sigma), \tau)$ with indexes varying in the same ranges as \eqref{e:theRgenerators}. Again the elements $\theta_{i,j}^{(r)} \in \h R(y_n(\sigma), \tau)$ are defined via the recursion \eqref{e:higherthetas}. By a standard filtration argument we deduce that the elements of $\h R(y_n(\sigma), \tau)$ with the same names as \eqref{e:theRgenerators} generate $\h R(y_n(\sigma), \tau)$ as a commutative algebra. We have shown that $\h R(y_n(\sigma), \tau)\onto R(y_n(\sigma), \tau)$ maps a spanning set to a basis, which concludes the proof.
\end{proof}

Let $\gr R(y_n(\sigma), \tau)$ denote the graded algebra for the loop filtration.
In the last paragraph of the proof of Theorem~\ref{T:PDyangian} we obtained the following important result.
\begin{Corollary}
\label{C:loopforcurrent}
If $\sigma$ is a symmetric shift matrix and $y_n(\sigma)$ is the semiclassical shifted Yangian then there is a Poisson isomorphism $S(\cc_n(\sigma)^\tau) \isoto \gr R(y_n(\sigma), \tau)$ given by
\begin{eqnarray}
\label{e:PDloopiso}
\begin{array}{lcl}
\theta_{i;r-1} \longmapsto \theta_i^{(r)} + \F_{r-2}y_n(\sigma) & \text{ for } & 1\le i<n, \ s_{i,i+1} < r ;\\
\eta_{i;2r-1} \longmapsto \eta_i^{(2r)} + \F_{2r-2}y_n(\sigma) & \text{ for } & 1\le i \le n, \ 0 < r.\vspace{-18pt}
\end{array}
\end{eqnarray}
$\hfill \qed$
\end{Corollary}

\section{Finite $W$-algebras for classical Lie algebras}
\label{s:finiteWalgebrasforclassical}

In this section we continue to work over $\C$ and we fix the following notation:
\begin{itemize}
\setlength{\itemsep}{4pt}
\item $N > 0$ is an integer and $\ve \in \{\pm 1\}$ such that $\ve^N = 1$.
\item $\g = \gl_N(\C)$ and $\g_+ \subseteq \g$ is a classical Lie subalgebra such that
$$\g_+ \cong \left\{ \begin{array}{rc} \so_N & \text{ if } \ve = 1, \\ \sp_N & \text{ if } \ve = -1.\end{array}\right.$$
\item $G = \GL_N(\C)$ and $G_+ \subseteq G$ is the connected algebraic subgroup satisfying $\g_+ = \Lie(G_+)$.
\item $\kappa : \g_+ \times \g_+ \to \C$ is the trace form associated to the natural representation $\C^N$ of $G$.
\end{itemize} 


\subsection{The symmetric shift matrix and the centraliser of a nilpotent element}
\label{ss:Dynkinandcentraliser}
Choose a partition $\lambda \vdash N$ and assume that: 
all parts are even when $\varepsilon = -1$ and
all parts are odd when $\varepsilon = 1$.

We recalled the notion of a shift matrix in Section~\ref{ss:CSpresentationsforshifts}, following \cite{BK06}. The {\it symmetric shift matrix for} $\lambda$ is the shift matrix $\sigma = (s_{i,j})_{1\leq i,j\leq n}$ defined by
\begin{eqnarray}
\label{e:symmetricalshiftmatrix}
s_{i,j} := \frac{|\lambda_i - \lambda_j|}{2} \in \Z.
\end{eqnarray}

Introduce symbols $\{b_{i,j} \mid 1\le i \le n , \ 1\le j \le \lambda_i\}$ and we identify $\C^N$ with the vector space spanned by the $b_{i,j}$. The general linear Lie algebra $\g = \gl_N$ has a basis 
\begin{equation}
\{e_{i,j;k,l} \mid 1\le i,k \le n,\ 1\le j\le \lambda_i, \ 1\le l \le \lambda_k\}
\end{equation}
where $e_{i,j; k,l} b_{r,s} = \delta_{k,r} \delta_{l, s} b_{i,j}$, so that
\begin{eqnarray}
[e_{i_1,j_1;k_1,l_1}, e_{i_2,j_2;k_2,l_2}] = \delta_{k_1, i_2} \delta_{l_1, j_2} e_{i_1, j_1; k_2, l_2} - \delta_{k_2, i_1}\delta_{l_2, j_1} e_{i_2, j_2; k_1, l_1}
\end{eqnarray}
We pick the nilpotent element in $\g$ by the rule
\begin{eqnarray}
\label{e:edefinition}
e := \sum_{i=1}^n \sum_{j = 1}^{\lambda_i-1} e_{i,j;i,j + 1}
\end{eqnarray}
It has Jordan blocks of sizes $\lambda_1,...,\lambda_n$. We also define a semisimple element $h\in \gl_N$ by
\begin{eqnarray}
\label{e:hdefinition}
h := \sum_{i=1}^n \sum_{j=1}^{\lambda_i} (\lambda_i-1 - 2j)e_{i,j;i,j}
\end{eqnarray}
It is easy to see that the pair $\{e,h\}$ can be completed to an $\sl_2$ triple $\{e,h,f\}$. We refer to the grading $\g = \bigoplus_{i\in \Z} \g(i)$ induced by $\ad(h)$ as the {\it Dynkin grading for $e$.} It is well-known that this is a good grading in the sense of \cite{BG07}. It satisfies
\begin{eqnarray}
\label{e:gooddegrees}
\deg(e_{i,j;k,l}) = 2(l-j) + \lambda_i - \lambda_k.
\end{eqnarray}
and, in particular, $e \in \g(2)$.

For $1\le i, k\le n$ and $r = s_{i,k}, s_{i,k}+ 1,..., s_{i,k} + \min(\lambda_i, \lambda_k) - 1$, we define elements
\begin{eqnarray}
\label{e:cijrdefinition}
c_{i,k}^{(r)} = \sum_{ 2r = 2(l-j) + \lambda_i - \lambda_k} e_{i,j;k,l}
\end{eqnarray}

The following fact is well known. See \cite[Lemma~7.3]{BK06} where the notation $c_{i,j}^{(r)}$ differs from ours by a shift in $r$.
\begin{Lemma}
\label{L:centraliserlemma}
The centraliser $\g^e$ has basis
\begin{eqnarray}
\label{e:gebasis}
\{c_{i,k}^{(r)} \mid 1\le i,k \le n, \ s_{i,k} \le r < s_{i,k} + \min(\lambda_i, \lambda_k)\}
\end{eqnarray}
with Lie bracket
\begin{equation}
[c_{i,j}^{(r)}, c_{k,l}^{(s)}] = \delta_{j,k}c_{i,k}^{(r + s)} - \delta_{i,l} c_{k, j}^{(r+s)}.
\end{equation}
Furthermore $\g^e$ is a Dynkin graded Lie subalgebra with $c_{i,k}^{(r)}$ lying in degree $2r$. $\hfill \qed$
\end{Lemma}

\subsection{Symplectic and orthogonal subalgebras}
\label{ss:symplecticandorthogonal}
Consider the matrix
\begin{eqnarray}
\label{e:Jdefinition}
J:= \sum_{i=1}^n \sum_{j=1}^{\lambda_i} (-1)^{j}  e_{i,j; i, \lambda_i + 1 - j}
\end{eqnarray}
This block diagonal matrix can be described as follows. For each index $i = 1,...,n$ there is a block of size $\lambda_i$, each of these blocks has alternating entries $\pm 1$ on the antidiagonal and zeroes elsewhere. Thus when $\ve = 1$ the matrix $J$ is symmetric and when $\ve = -1$ it is anti-symmetric. It follows that 
\begin{eqnarray}
\label{e:Jinverse}
J^{-1} = \ve J.
\end{eqnarray}

Now define an involution of $\g = \gl_N$ by the rule
\begin{eqnarray}
\label{e:taudefinition}
\tau : X \mapsto - J^{-1} X^\top J
\end{eqnarray}

\begin{Lemma}
\label{L:tauaction}
For all admissible indexes we have:
\begin{enumerate}
\setlength{\itemsep}{4pt}
\item[(i)] $\tau(e_{i,j;k,l}) = (-1)^{j - l - 1}  e_{k, \lambda_k + 1 - l; i, \lambda_i + 1 - j}$;
\item[(ii)] $\tau(e) = e$ and $\tau(h)=h$;
\item[(iii)] $\tau(c_{i,k}^{(r)}) = (-1)^{r-\frac{\lambda_k - \lambda_i}{2} - 1}c_{k, i}^{(r)}.$
\end{enumerate}
\end{Lemma}
\begin{proof}
Using \eqref{e:Jinverse} and multiplying matrices we have
$\tau(e_{i,j;k,l}) = -\ve (-1)^{\lambda_k + j - l}  e_{k, \lambda_k + 1 - l; i, \lambda_i + 1 - j}$ and so (i) follows from the fact that $\varepsilon(-1)^{\lambda_i} = 1$ under our assumptions on $\varepsilon$ and $\lambda$ fixed at the start of Section~\ref{ss:Dynkinandcentraliser}. 
Now (ii) and (iii) follow by applying (i) to \eqref{e:edefinition}, \eqref{e:hdefinition} and \eqref{e:cijrdefinition}.
\end{proof}
We decompose $\g = \g_+ \oplus \g_-$ into eigenspaces for $\tau$ so that $\g_+$ is the space of $\tau$-invariants. Then we have
\begin{eqnarray}
\label{e:fixedpointLietype}
\g_+ \cong \left\{\begin{array}{cl} \so_N & \text{ if } \ve = 1; \\
															\sp_N & \text{ if } \ve = -1.

\end{array}\right.
\end{eqnarray}
By Lemma~\ref{L:tauaction}(ii) we have $\{e,h\}\subseteq \g_+$, and $\g_-$ is a $\g_+$-module. Write $G_+$ for the connected component of the group of elements $g\in \GL_N$ such that $gJg^\top = J$. This is a classical group satisfying $\Lie(G_+) = \g_-$.

As $\tau$ preserves $\g(-2)$ we also have that $\g(-2)=\g_+(-2)\oplus \g_-(-2)$. Writing $f=f_1+f_2$ with $f_1\in \g_+(-2)$ and $f_2\in \g_-(-2)$ and using the fact that $h\in \g_+$ we deduce that $h=[e,f_1]$ and $[e,f_2]=0$. Since the Dynkin grading is good for $e$ in $\g$ this yields $f_2=0$. We conclude that the $\sl_2$-triple $\{e,h,f\}$ is contained  in $\g_+$.

Our next result explains the relationship between the shifted current algebras and the centralisers described above. We refer the reader to Theorem~\ref{T:CSshiftedcurrents2} for a recap of notation.
\begin{Lemma}
\label{L:shiftedcurrentandcentraliser}
\begin{enumerate}
\setlength{\itemsep}{4pt}
\item There is a surjective Lie algebra homomorphism
\begin{eqnarray*}
\cc_n(\sigma) \longtwoheadrightarrow \g^e ; \\
e_{i,j}t^r \longmapsto c_{i,j}^{(r)}
\end{eqnarray*}
 where $c_{i,j}^{(r)} := 0$ for $r \ge s_{i,j} + \min(\lambda_i, \lambda_j)$.
The kernel is Poisson generated by $e_{1,1}t^r$ with $r\ge \lambda_1$.
\item If $\sigma$ is symmetric the homomorphism from (1) restricts to $\cc_n(\sigma)^\tau \onto \g^e_+$.
\begin{itemize}
\item[(i)] For $\ve = 1$ the kernel is Poisson generated  by $\{\eta_{1; 2r-1} \mid 2r-1 \ge \lambda_1\}$.
\item[(ii)] For $\ve=-1$ the kernel is Poisson generated by those same elements as the $\ve = 1$ case, along with $$\{\theta_{i; \lambda_i + s_{i,i+1}} \mid i=1,...,n-1\}.$$
\end{itemize} 
\end{enumerate}
\end{Lemma}
\begin{proof}
Part (1) is explained in \cite[Lemma~2.6]{GT19c}. Compare \eqref{e:tauonshiftedcurrents} with Lemma~\ref{L:tauaction}(iii) to see that the map is $\tau$-equivariant, hence it is a well-defined map $\cc_n(\sigma)^\tau \onto \g^e_+$.

We go on to describe the kernel in (2). First of all observe that the elements listed there are $\tau$-fixed elements of the kernel of $\cc_n(\sigma)\onto \g^e$; for $\ve = -1$ we use \eqref{e:shiftedcurrentfreerels2}, \eqref{e:shiftedcurrentfreerels3}. Now observe that $\g_+^e$ has a spanning set consisting of elements of the form $c_{i,j}^{(r)} + \tau(c_{i,j}^{(r)})$. By Lemma~\ref{L:centraliserlemma} we can show that the elements in (2) generate the kernel by checking that the ideal $\mi$ which they generate contains
\begin{eqnarray}
\label{e:inthekernel}
\{\eta_{i;r} \mid i=1,...,n, \ r \ge \lambda_i\} \cup \{\theta_{i;r} \mid i=1,..,n-1, \ r \ge s_{i,i+1} + \lambda_i \}.
\end{eqnarray}
First take $\ve = 1$ so that $\lambda_1$ is odd. Using relation \eqref{e:dyrel2} we see that $\mi$ contains $\{\eta_{1;\lambda_1}, \theta_{1;r}\} = \theta_{1; \lambda_1 + r}$ for $r \ge s_{1,2}$, whilst \eqref{e:dyrel5} shows that $\eta_{1; \lambda_2 + r} - \eta_{2; \lambda_2 + r} \in \mi$ for all $r \ge 0$. In particular $\eta_{2;\lambda_2+r} \in \mi$ for $r\ge 0$. By Theorem~\ref{T:CSshiftedcurrents2} the subalgebra of $\cc_n(\sigma)^\tau$ generated by $\theta_{i;r}, \eta_{i;r}$ with $2\le i$ is isomorphic to a current Lie algebra of smaller rank so the description of the kernel follows by induction.

For $\ve = -1$, $\lambda_1$ is even and the ideal $\mi$ contains $\{\eta_{1;\lambda_1+1}, \theta_{1; s}\} = \theta_{1;\lambda_1 + 1 + s}$ for $s \ge s_{1,2}$. Together with the additional generator $\theta_{1;\lambda_1 + s_{1,2}}$ this gives all elements of the form $\theta_{1;r}$ listed in \eqref{e:inthekernel}. Now the induction proceeds in the same way as the case $\ve = 1$.
\end{proof}

\subsection{Generators of the $W$-algebra}
\label{ss:generatorsWalgebra}

Here we recall formulas for generators of $S(\gl_N, e)$, due to Brundan and Kleshchev \cite[\textsection 9]{BK06}, see also \cite[\textsection 3.3]{BK08}. Their notation is slightly different, however it is a simple exercise to translate between the two settings.  In this section we continue to assume that all parts of $\lambda$ are odd when $\ve  = 1$ and all parts are even when $\ve  = -1$.

Now for $1 \le i,k \le n$, $0 \le x < n$ and $r > 0$, we let
\begin{equation}\label{e:invaraintsdef}
t_{i,k;x}^{(r)} := \sum_{s=1}^r (-1)^{r-s} \sum_{\substack{(i_m, j_m, k_m, l_m) \\ m =1,...,s}} (-1)^{\hash \{q=1,\dots,s-1 \mid k_q \le x\}}
e_{i_1, j_1; k_1, l_1} \cdots e_{i_s, j_s; k_s, l_s} \in S(\g(\geez))
\end{equation}
where the sum is taken over all indexes such that for $m=1,...,s$
\begin{eqnarray}
\label{e:indexbounds}
1 \le i_1,...,i_s, k_1,...,k_s \le n, 1\le j_m \le \lambda_{i_m} \text{ and } 1\le l_m \le \lambda_{k_m}
\end{eqnarray}
satisfying the following six conditions:
\begin{itemize}
\setlength{\itemsep}{4pt}
\item[(a)] $\sum_{m=1}^s (2l_m - 2j_m + \lambda_{i_m} - \lambda_{k_m}) = 2(r - s)$;
\item[(b)] $2l_m- 2j_m + \lambda_{i_m} - \lambda_{k_m} \ge 0$ for each $m = 1,\dots,s$;
\item[(c)] if $k_m > x$, then $l_m < j_{m+1}$ for each $m = 1,\dots,s-1$;
\item[(d)] if $k_m \le x$ then $l_m \ge j_{m+1}$ for each $m = 1,\dots,s-1$; 
\item[(e)] $i_1 = i$, $k_s = k$;
\item[(f)] $k_m = i_{m+1}$ for each $m = 1,\dots,s-1$.
\end{itemize}

Keeping $r > 0$ fixed and choosing $s \in \{1,...,r\}$ we make the notation $\X_{i,k}^{(r, s)}$ for the set of ordered sets $(i_m, j_m, k_m, l_m)_{m = 1}^s$ in the range \eqref{e:indexbounds} satisfying the conditions (a)--(f) and consider the map
\begin{eqnarray}
\upsilon(i_m,j_m,k_m,l_m) := (k_{s+1-m}, \lambda_{k_{s+1-m}} + 1 - l_{s+1-m}, i_{s+1-m}, \lambda_{i_{s+1-m}} + 1 - j_{s+1-m})
\end{eqnarray}
defined on ordered sets in the range \eqref{e:indexbounds}.
\begin{Lemma}
\label{L:upsilonbijects}
The map $\upsilon$ defines a bijection 
$\X_{i,k}^{(r,s)} \overset{1\text{-}1}{\longrightarrow} \X_{k, i}^{(r,s)}.$
\end{Lemma}
\begin{proof}
The index bounds \eqref{e:indexbounds} ensure that $\X_{i,k}^{(r,s)}$ is finite. Since $\upsilon^2$ is the identity map it will suffice to show that $\upsilon$ is well defined $\X_{i,k}^{(r,s)} \to \X_{k, i}^{(r,s)}$. We fix $(i_m, j_m, k_m, l_m)_{m = 1}^s \in \X_{i,k}^{(r,s)}$ and check that conditions (a)--(f) hold for $\upsilon(i_m, j_m, k_m, l_m)$. For example, 
\begin{eqnarray*}
\sum_{m=1}^s \big(2(\lambda_{i_{s+1-m}} + 1 - j_{s+1-m}) - 2(\lambda_{l_{s+1-m}} + 1 - k_{s+1-m}) + (\lambda_{k_{s+1-m}} - \lambda_{i_{s+1-m}})\big)\ \ \ \ \ \ \ \ \\
\ \ \ \ \ \ \ \ \ \ = \sum_{m=1}^s (2l_m - 2j_m + \lambda_{i_m} - \lambda_{k_m})   = 2(r - s)
\end{eqnarray*}
This shows that (a) for $(i_m, j_m, k_m, l_m)_{m = 1}^s$ implies (a) for $\upsilon(i_m, j_m, k_m, l_m)_{m = 1}^s$. Conditions (b)--(f) can be checked similarly.
\end{proof}

Now we define a map $\theta : \g^e \to S(\g(\geez))$ by
\begin{eqnarray}
\label{e:intermediatemap}
\begin{array}{rcl}
c_{i,i}^{(r)}  \longmapsto (-1)^{r-1}t_{i,i;i-1}^{(r-1)} & & \text{ for } i=1,...,n \text{ and } 0 \le r < \lambda_i;\\
c_{i,k}^{(r)} \longmapsto (-1)^{r-1}t_{i,k; i}^{(r-1)} & & \text{ for } 1\le i < k \le n \text{ and } r-s_{i,k} = 0,...,\min(\lambda_i, \lambda_k) -1;\\
c_{k, i}^{(r)} \longmapsto (-1)^{r-1}t_{k,i;i}^{(r-1)} & & \text{ for } 1\le i < k \le n \text{ and } r-s_{j,k} = 0,...,\min(\lambda_i, \lambda_k) -1.
\end{array}
\end{eqnarray}

Comparing \eqref{e:gooddegrees} with Lemma~\ref{L:centraliserlemma} and Lemma~\ref{L:tauaction} we see that $\tau$ induces involutions on $\g^e$ and $S(\g(\geez))$, and these will all be denoted $\tau$. The hardest part of the following result is the assertion that the image of the map $\theta$ described in \eqref{e:intermediatemap} lies in $S(\g,e)$, which follows from a result of Brundan and Kleshchev. We add to this the observation that $\theta$ is $\tau$-equivariant.
\begin{Proposition}
\label{P:tauequivariance}
$\theta : \g^e \to S(\g,e)$ is $\tau$-equivariant. In particular we have
\begin{eqnarray}
\tau(t_{i,k;m}^{(r-1)}) = (-1)^{r-\frac{\lambda_k - \lambda_i}{2} - 1} t_{k, i; m}^{(r-1)} & & \text{ for }  1\le i \le k \le n \text{ and all } m, r. \label{e:tauclaim2}
\end{eqnarray}
\end{Proposition}
\begin{proof}
The assertion that the image of $\theta$ lies in $S(\g,e)$ is an immediate consequence of \cite[Corollary~9.4]{BK06}, upon taking the top graded term with respect to the Kazhdan filtration of the finite $W$-algebra (see also \cite[\textsection 4]{GT19b} for a short summary).
Thanks to Lemma~\ref{L:tauaction}(iii) the $\tau$-equivariance will follow from \eqref{e:tauclaim2}.  Fix indexes $1\le i\le k \le n, \ 0\le x < n$ and $0 < r$. Choose $s\in \{1,...,r\}$. Let $(i_m, j_m, k_m,l_m)_{m=1}^s \in \X_{i,i}^{(r,s)}$. Thanks to Lemma~\ref{L:tauaction}(i) we have
\begin{eqnarray}
\label{e:tauclaim3}
\tau(\prod_{m=1}^s e_{i_m, j_m; k_m, l_m}) &=& (-1)^{\sum_{m=1}^s (j_m - l_m + 1)} \prod_{t=1}^s e_{\upsilon(i_m, j_m, k_m, s_m)}
\end{eqnarray}
Conditions (a), (e) (f) imply that
$\sum_m (j_m - l_m + 1) = \frac{\lambda_{i_1} - \lambda_{k_s}}{2} - r.$
Hence the sign on the right hand side of \eqref{e:tauclaim3} is equal to $(-1)^{r-\frac{\lambda_k - \lambda_i}{2}}$. Furthermore condition (f) ensures that
$\hash \{ m =1,...,s-1 \mid k_m \le x\} = \hash \{m =2,...,s \mid i_m \le x\}.$
Combining these observations together with Lemma~\ref{L:upsilonbijects} we have
\begin{eqnarray*}
\tau(t_{i,k;x}^{(r)}) &=& \sum_{s=1}^r (-1)^{r-s} \sum_{(i_m, j_m, k_m, l_m)_m \in \X_{i,k}^{(r,s)}} (-1)^{\hash\{ m =1,...,s-1 \mid k_m \le x\}}\tau (\prod_m e_{i_m, j_m, k_m, l_m}) \\
& = & (-1)^{r-\frac{\lambda_k - \lambda_i}{2}} \sum_{s=1}^r (-1)^{r-s} \sum_{(i_m, j_m, k_m, l_m)_m \in \X_{k,i}^{(r,s)}} (-1)^{\hash\{ m =1,...,s-1 \mid k_m \le x\}}\prod_m e_{i_m, j_m, k_m, l_m}\\
& = & (-1)^{r-\frac{\lambda_k - \lambda_i}{2}} t_{k,i;x}^{(r)}
\end{eqnarray*}
This completes the proof.
\end{proof}
\begin{Remark}
\label{R:explicitPBWmap}
Comparing  the linear terms appearing in \eqref{e:invaraintsdef} with the basis \eqref{e:cijrdefinition} for $\g^e$ one can check that $\theta$ satisfies the properties of Theorem~\ref{T:PBW} (Cf. \cite[Lemma~4.2]{GT19b}).
\end{Remark}

\subsection{The semiclassical Brundan--Kleshchev homomorphism}
\label{ss:semiclassicalyangiansandWalgebras}

Recall that $\g_+ \subseteq \g = \gl_N$ is a classical Lie subalgebra constructed in Section~\ref{ss:symplecticandorthogonal}, that $e\in \g_+$ is a nilpotent element, and $\sigma$ is a choice of symmetric shift matrix introduced in \eqref{e:symmetricalshiftmatrix}. In Section~\ref{ss:generatorsWalgebra} we observed that the automorphism $\tau$ on $\g$ which fixes $\g_+$ induces a natural automorphism on $S(\g,e)$. Furthermore in \eqref{e:tauon} we defined a Poisson involution of $y_n(\sigma)$ which is also denoted $\tau$.

The relationship between these objects is described by the following result.
\begin{Proposition}
\label{P:semiclassicalBK}
There is a $\tau$-equivariant surjective Poisson homomorphism
\begin{eqnarray}
\label{e:yangiantofwa}
\varphi : y_{n}(\sigma) \longtwoheadrightarrow S(\g,e)
\end{eqnarray}
determined by
\begin{eqnarray}
\label{e:dgenerators}
\begin{array}{rcl}
d_i^{(r)} \longmapsto  t_{i,i;i-1}^{(r)} & \text{ for } & 1 \le i \le n, \ r >0;  \\ [3pt] 
e_i^{(r)} \longmapsto t_{i,i+1;i}^{(r)} & \text{ for } & 1 \leq i < j \leq n, \ r >s_{i,j} ;  \\ [3pt] 
f_i^{(r)} \longmapsto t_{i+1, i; i}^{(r)} & \text{ for } & 1 \leq i < j \leq n, \ r > s_{j,i}.
\end{array}
\end{eqnarray}
The kernel is the Poisson ideal generated by $\{d_1^{(r)} \mid r > \lambda_1\}$. If we equip $y_n(\sigma)$ with the canonical grading doubled (see Theorem~\ref{T:shiftedyangianPBW}) and equip $S(\g,e)$ with the Kazhdan grading then this homomorphism is graded.
\end{Proposition}
\begin{proof}
In \cite[\textsection 2]{BK06} Brundan and Kleshchev introduced the shifted Yangian $Y_n(\sigma)$, of which $y_n(\sigma)$ is the semiclassical limit under the canonical filtration. In \cite[Theorem~10.1]{BK06} they construct a homomorphism to the quantum finite $W$-algebra (see also \cite[Theorem~4.3]{GT19b} where their result is recalled in notation similar to that of the present paper). It is straightforward to see that the doubled canonical filtration lines up with the Kazhdan filtration. The map defined by \eqref{e:dgenerators} is just the semiclassical limit of this filtered algebra homomorphism, and this proves that the map is a surjective Poisson homomorphism. The $\tau$-equivariance can be checked by comparing Proposition~\ref{P:tauequivariance} with formulas \eqref{e:tauon}.

Another consequence of \cite[Theorem~10.1]{BK06} is that the elements $\{d_1^{(r)} \mid r > \lambda_1\}$ lie in the kernel of $\varphi$. Write $I$ for the Poisson ideal generated by these elements. Combining Theorem~\ref{T:PBW} and Lemma~\ref{L:centraliserlemma} we can show that $I = \ker \varphi$ by demonstrating that the quotient $y_n(\sigma) / I$ is generated as a commutative algebra by
\begin{eqnarray}
\label{e:truncatedgens}
\begin{array}{rcl}
& \{e_{i,j}^{(r)} \mid 1\le i < j \le n, \ s_{i,j} < r \le s_{i,j} + \min(\lambda_i ,\lambda_j) \} \cup \{d_i^{(r)} \mid 1\le i \le n, \ 0 < r \le \lambda_{i} \} & \\ & \cup \{f_{i,j}^{(r)} \mid 1\le i < j \le n, \ s_{j,i} < r \le s_{j,i} + \min(\lambda_i ,\lambda_j) \}. &
\end{array}
\end{eqnarray}
By Theorem~\ref{T:shiftedyangianPBW} we can identify $S(\cc_n(\sigma)) = \gr y_n(\sigma)$ as Poisson algebras, with respect to the loop filtration. The associated graded ideal $\gr I$ contains elements $\{d_{i;r} \mid r \ge \lambda_1\}$, and so applying Lemma~\ref{L:shiftedcurrentandcentraliser}(1) we see that $\gr y_n(\sigma)/I$ is generated by by the top graded components of the elements \eqref{e:truncatedgens}. Applying a standard filtration argument we deduce that $y_n(\sigma)/ I$ is generated by the requisite elements.
\end{proof}

\subsection{Dirac reduction of the semiclassical Brundan--Kleshchev homomorphism}
\label{ss:symplecticrelations}

Recall that $\sigma$ is a symmetric shift matrix defined in \eqref{e:symmetricalshiftmatrix}. Thanks to \eqref{e:notrightexact}, Theorem~\ref{T:firstmain} and Proposition~\ref{P:semiclassicalBK} we have a Poisson homomorphism
\begin{eqnarray}
\label{e:varphiforclassical}
\varphi : R(y_n(\sigma), \tau) \longtwoheadrightarrow S(\g_+, e).
\end{eqnarray}
The purpose of this section is to describe the kernel. Combined with Theorem~\ref{T:PDyangian} this will complete the proof of Theorem~\ref{T:Slicepresentation}. First we deal with orthogonal types.

\begin{Proposition}
\label{P:Walgpresentation1}
When $\ve = 1$ the kernel of \eqref{e:varphiforclassical} is Poisson generated by $$\{\eta_1^{(2r)} \mid \ 2r > \lambda_1\}.$$
\end{Proposition}
\begin{proof}
The elements $\eta_1^{(2r)}$ with $2r > \lambda_1$ are certainly in the kernel of $\varphi$, thanks to Proposition~\ref{P:semiclassicalBK}. Let $I$ denote the Poisson ideal of $R(y_n(\sigma), \tau)$ which they generate. It follows from Lemma~\ref{L:shiftedcurrentandcentraliser}(2)(i) that the associated graded algebra $\gr (R(y_n(\sigma), \tau)/I)$ with respect to the loop filtration is generated (as a commutative algebra) by the top filtered component of the following elements
\begin{eqnarray}
\label{e:orthtruncatedgens}
\{\theta_{i,j}^{(r)} \mid 1\le i < j \le n, \ s_{i,j} < r \le s_{i,j} + \lambda_i \} \cup \{\eta_i^{(2r)} \mid 1\le i \le n, \ 0 < r \le \lambda_{i} \}.
\end{eqnarray}
Hence $R(y_n(\sigma), \tau)/I$ is generated by these elements \eqref{e:orthtruncatedgens}. Using Theorem~\ref{T:PBW} and Remark~\ref{R:explicitPBWmap} we see that $S(\g_+, e)$ is a polynomial algebra generated by the images of these elements under $\varphi$. It follows that the induced map $R(y_n(\sigma), \tau) / I \onto S(\g_+, e)$ is an isomorphism, which completes the proof.
\end{proof}

Now we fix $\ve = -1$ and we introduce a family of elements of $R(y_n(\sigma), \tau)$; they lie in the kernel of \eqref{e:varphiforclassical} and provide filtered lifts of the elements appearing in Lemma~\ref{L:shiftedcurrentandcentraliser}(2)(ii).

Define elements  $\{ \ttheta_i^{(s_{i,i+1} + \lambda_i)} \mid i=1,...,n-1\} \subseteq R(y_n(\sigma),\tau)$
as follows: first let $\od_1^{(\lambda_1+1)} := d_1^{(\lambda_1+1)} \in y_n(\sigma)$ and then inductively define 
\begin{eqnarray}
\od_i^{\lambda_i+1} := \{\{\od_i^{(\lambda_i + 1)}, \he_i^{(s_{i,i+1} + 1)}\}, \he_i^{(s_{i,i+1}+1)}\} \in y_n(\sigma).
\end{eqnarray}
Note that $\tau(\he_i^{(s_{i,i+1}+1)}) = \he_i^{(s_{i,i+1}+1)}$ and so $\tau(\od^{(\lambda_i+1)}) = -\od^{(\lambda_i+1)}.$
Therefore we may define
\begin{eqnarray}
\label{e:additionalgenerators}
\ttheta_i^{(\lambda_i + s_{i,i+1} + 1)} := \{\od_i^{(\lambda_i+1)}, \ce_i^{(s_{i,i+1} + 1)}\} + \cy_n(\sigma)^\tau \in y_n(\sigma)^\tau/\cy_n(\sigma)^\tau = R(y_n(\sigma), \tau)
\end{eqnarray}

\begin{Example}
\label{Ex:firstbadrel}
The first occurrence of one of these new elements is easy to calculate. We have
\begin{eqnarray}
\ttheta_1^{(\lambda_1+s_{1,2} + 1)} = \sum_{t=0}^{\lambda_1/2} \eta_1^{(2t)} \theta_1^{(\lambda_1 - 2t + s_{1,2} + 1)}.
\end{eqnarray}
For $i > 1$ the expressions are more complicated, but it would be interesting to obtain a closed formula. We will not pursue this in the present article.
\end{Example}

\begin{Proposition}
\label{P:Walgpresentation-1}
For $\ve = -1$ the kernel of \eqref{e:varphiforclassical} is Poisson generated by
\begin{eqnarray}
\label{e:extrakernel}
\Big\{\eta_1^{(2r)} \mid 2r > \lambda_1\Big\}\cup \Big\{ \ttheta_i^{(\lambda_i + s_{i,i+1} + 1)} \mid i=1,...,n-1\Big\}.
\end{eqnarray}
\end{Proposition}

\begin{proof}
Let $I$ be the Poisson ideal of $R(y_n(\sigma), \tau)$ generated by \eqref{e:extrakernel}. Recall that the kernel of the map $y_n(\sigma) \to S(\g,e)$ is generated by $d_1^{(r)}$ with $r > \lambda_1$ (Proposition~\ref{P:semiclassicalBK}). It follows that the elements \eqref{e:extrakernel} are projections to $R(y_n(\sigma), \tau)$ of certain elements of $\ker(y_n(\sigma) \to S(\g,e))^\tau$, and so these elements all lie in the kernel of the map \eqref{e:varphiforclassical}. In other words $\varphi$ factors through $R(y_n(\sigma),\tau) \to R(y_n(\sigma), \tau)/I$.

Similar to the proof of Proposition~\ref{P:Walgpresentation1} it suffices to show that $R(y_n(\sigma), \tau)/I$ is generated as a commutative algebra by the elements \eqref{e:orthtruncatedgens}. An easy calculation in $y_n(\sigma)$, using Corollary~\ref{C:loopforcurrent}, shows that the top filtered component of $\ttheta_i^{(s_{\lambda_i + s_{i,i+1} + 1})}$ with respect to the loop filtration is $\theta_{i; \lambda_i + s_{i,i+1}}$. Now we can apply Lemma~\ref{L:shiftedcurrentandcentraliser}(2)(ii) to see that $\gr R(y_n(\sigma), \tau)$ is generated by the top filtered components of the elements \eqref{e:orthtruncatedgens}. The argument concludes in exactly the same manner as Proposition~\ref{P:Walgpresentation1}.
\end{proof}

Thanks to Lemma~\ref{L:centraliserlemma} the Lie algebra $\g_+^e$ is spanned by elements $c_{i,j}^{(r)} + \tau(c_{i,j}^{(r)})$; note that $\tau$ here refers to the automorphism of $\gl_N$ given in \eqref{e:taudefinition}. Using \eqref{e:invaraintsdef} along with Theorem~\ref{T:PBW} and Propositions~\ref{P:Walgpresentation1} and \ref{P:Walgpresentation-1} we see that $S(\g_+,e)$ is generated as a commutative algebra by the images under $\varphi$ of elements \eqref{e:orthtruncatedgens}. 
By slight abuse of notation we denote the images by the same names.
\begin{Corollary}
\label{C:madic}
Let $\ve = \pm 1$. If $\m_0$ is the unique graded maximal ideal of $S(\g_+, e)$ then we have a Poisson isomorphism $S(\g_+^e) \isoto \gr_{\m_0} S(\g_+,e)$ to the $\m_0$-adic graded algebra defined by
\begin{eqnarray}
\label{e:anotherlabel}
\setlength{\itemsep}{4pt}
\begin{array}{rcl}
c_{i,i+1}^{(r)} + \tau(c_{i,i+1}^{(r)})\longmapsto \theta_{i}^{(r+1)} + \m_0^2 & \text{ for } & 1\le i < n, \ s_{i,j} \le r < s_{i,j} + \lambda_i;\\
c_{i,i}^{(r)} + \tau(c_{i,i}^{(r)})\longmapsto \eta_{i}^{(r+1)} + \m_0^2 & \text{ for } & 1\le i \le n, \ 0 \le r < \lambda_i.
\end{array}
\end{eqnarray}
\end{Corollary}
\begin{proof}
The graded algebra $\gr_{\m_0} S(\g_+, e)$ is Poisson, and it is straightforward to check, using Theorem~\ref{T:PDyangian}, that \eqref{e:anotherlabel} gives a Lie algebra homomorphism $\g_+^e \onto \gr_{\m_0} S(\g_+, e)$. By the universal property of the Lie--Poisson algebra we get a Poisson homomorphism $S(\g_+^e) \onto \gr_{\m_0} S(\g_+, e)$. By Theorem~\ref{T:PBW} we see that this map is an isomorphism.
\end{proof}

\part{One dimensional representations of $W$-algebras}

\section{Generalities on Poisson schemes}
\label{s:poissonschemes}
\subsection{Conic degenerations of affine schemes}
\label{ss:degenerationsofschemes}
We begin the second Part of the paper by recording two basic results, allowing us to compare the number of irreducible components of certain complex schemes of finite type. The first records an obvious bound on the dimensions of components of subschemes, whilst the second allows us to bound the number of components of a deformation of a reduced conic affine scheme.

We define $\P$ be the set of all finite sequences of all non-negative integers of arbitrary length. We order $\P$ lexicographically as follows: if $d = (d_1,...,d_n)$ and $d' = (d_1',...,d_m')$ then we say that $d > d'$ if there is an index $j \in \{1,...,\min(n,m)\}$ such that $d_i = d'_i$ for $i=1,...,j$ and $d_{j+1} > d'_{j+1}$. Here we adopt the convention $d_{m+1}' = d_{n+1} = 0$

If $X$ is a complex scheme of finite type then we write $X = \bigcup_{i=1}^l X_i$ for the decomposition into irreducible components, ordered so that $\dim X_i \ge \dim X_j$ for $i < j$. We define {\it the dimension vector of $X$} to be the sequence $d(X) := (\dim X_1, \dim X_2, ..., \dim X_l) \in \P$. 

The following useful fact will be used in several later arguments.
\begin{Lemma}
\label{L:closedembeddings}
Let $X, Y$ be complex schemes of finite type with a closed embedding $Y \to X$. Then $d(X)\ge d(Y)$ with equality if and only if the underlying reduced schemes are isomorphic.
\end{Lemma}

Let $A$ be a finitely generated $\C$-algebra. We say that $A = \bigcup_{i\ge 0} \Fc_i A$ is a {\it standard filtration of $A$} if there is a surjection $\C[x_1,...,x_n] \onto A$ for some $n$, along with integers $m_1,...,m_n \ge 0$ such that $\Fc_m A$ is spanned by the image of $\{x_1^{k_1} \cdots  x_n^{k_n} \mid k_1 m_1 + \cdots + k_n m_n \le m\}$. If $X = \Spec(A)$ is the associated affine scheme the we refer to $\C X = \Spec(\gr A)$ as the {\it asymptotic cone of $X$}.
\begin{Lemma}
\label{L:reduceddichotomy} 
If $\C X$ is reduced then $\hash \Comp(X) \le \hash \Comp(\C X)$.
\end{Lemma}
\begin{proof}
If $\gr A$ is reduced then so is $A$. We begin by choosing a presentation for the standard filtration. Let $R = \C[x_1,...,x_n] = \bigoplus_{i\ge 0} R_i$ be a graded  polynomial ring with  $x_i$ in degree $m_i$. Write $\F_i R = \bigoplus_{j=0}^i R_j$.
We let $\phi : R\to A$ be the homomorphism inducing the standard filtration, set $I := \Ker (\phi)$ and denote the minimal prime ideals over $I$ by $\p_1,...,\p_m \subseteq A$, so that $I = \bigcap_i \p_i$. Using \cite[Proposition~7.6.13]{MR01}, for example, we see that there is a natural isomorphism $\gr A \cong R / \gr I$ and we view $X$ and $\C X$ as subschemes of $\Spec R = \mathbb A^n$.

For $g \in \F_i R$ we write $\bar g = g + \F_{i-1} R \in \gr R$. For any ideal $J$ of $R$ we write $V(J) \subseteq \mathbb A^n$ for the corresponding variety of closed points. We have inclusions $\gr (J) \gr (K) \subseteq \gr (JK) \subseteq \gr(J) \cap \gr(K)$ for any ideals $J, K \subseteq R$ (see \cite[\textsection 5.3]{Pr02}), and so 
\begin{eqnarray}
\label{e:Vdecomposes}
V(\gr I) = \bigcup_{i=1}^m V(\gr \p_i).
\end{eqnarray}

We prove the contrapositive of the lemma, so assume that $m = \hash\Comp(X) > \hash\Comp(\C X)$. Then, a fortiori, we must have $V(\gr \p_j) \subseteq \bigcup_{i\neq j} V(\gr \p_i)$ for some $j$, say $j = 1$. Equivalently $\bigcap_{i\neq 1} \sqrt{\gr \p_i} \subseteq \sqrt{\gr \p_1}$, which implies $\bigcap_{i\neq 1} {\gr \p_i} \subseteq \sqrt{\gr \p_1}$. Since $I \subsetneq \p_2 \p_3 \cdots \p_m$ it follows that $\gr I \subsetneq \gr(\p_2 \cdots \p_m)$ and so we may choose $g_2,...,g_m$ with $g_i \in \p_i$ such that $g := g_2\cdots g_m$ satisfies $\bar g \notin \gr I$. On the other hand, we have $\bar g \in \gr(\p_2\cdots \p_m) \subseteq \bigcap_{i=2}^m \gr (\p_i) \subseteq \bigcap_{i=1}^m \sqrt{\gr (\p_i)} = \sqrt{\gr (I)}$, where the last equality follows from \eqref{e:Vdecomposes}. Thus we conclude that $\bar g \in \sqrt{\gr (I)} \setminus \gr I$, so that $\C X$ is not reduced. This completes the proof.
\end{proof}

\subsection{Completions and nilpotent elements of graded algebras}
If $\p\in \Spec(A)$ then, as usual, $A_\p$ denotes the localisation and $A_\p^\c$ the completion at the maximal ideal of $A_\p$. The kernel of $A\to A_\p$ is the set of elements annihilating some element of $A\setminus \p$, whilst the kernel of $A \to A_\p^\c$ is $\bigcap_{k>0} \p^k$. When $\p$ is a maximal ideal it follows from Krull's intersection theorem that these kernels are equal (see the remark following \cite[Theorem~10.17]{AM69}).

Now let $A = \bigoplus_{i\ge 0} A_i$ be a finitely generated, connected graded algebra with unique maximal graded ideal $\m_0$. The connected grading induces a one parameter family of automorphisms of $A$, and a $\C^\times$-action on $\mSpec(A)$ contracting to the unique fixed point $\m_0$. For $a\in A$ make the notation
$$\Xi_a := \Big\{\m \in \mSpec(A) \mid a\notin \ker(A \to A_\m^\c) \Big\}.$$
\begin{Lemma}
\label{L:technicallemma}
The following hold:
\begin{enumerate}
\item $\bigcap_\m \bigcap_{i> 0} \m^i = 0$ where the intersection is taken over all maximal ideals of $A$.
\item If $a \in A_i$ for some $i$ then $\Xi_a$ is closed and conic for the contracting $\C^\times$-action.
\end{enumerate}
\end{Lemma}
\begin{proof}
Since $A$ is finitely generated over $\C$ the intersection $\bigcap_{\m \in \mSpec(A)} \m = \Rad(A)$ is a graded ideal, however for every $i > 0$ we have $A_i \cap \m_0^{i+1} = 0$ and so (1) follows.

If $\m \in \mSpec(A)$, we remarked above that the kernel of $A \to A_\m^\c$ is equal to the kernel of $A \to A_\m$. Writing $\Ann(a) = \{b\in A \mid ab = 0\}$, it follows that $\ker(A \to A_\m^\c) = \{a \in A \mid ab = 0 \text{ for some } b\notin \m\} = \{a\in A\mid \Ann(a) \nsubseteq \m\}$. We have shown that $\Xi_a = \{\m\in \mSpec(A) \mid \Ann(a) \subseteq \m\}$, i.e. $\Xi_a$ is the subvariety of $\mSpec(A)$ cut out by $\Ann(a)$, hence closed. Finally it remains to see that $\Xi_a$ is stable under the $\C^\times$-action, and this follows directly from two facts: $\Ann(a)$ is a graded ideal whenever $a$ is homogeneous, and $\C^\times$-stable sets are precisely those defined by graded ideals.
\end{proof}

\begin{Lemma}
\label{L:reducedlemma}
Suppose that $A = \bigoplus_{i>0} A_i$ is a finitely generated, connected graded algebra with unique maximal graded ideal $\m_0$. The following are equivalent:
\begin{enumerate}
\setlength{\itemsep}{4pt}
\item[(i)] $A$ is reduced.
\item[(ii)] $A_{\m_0}^\c$ is reduced.
\item[(iii)] $A_\m^\c$ is reduced for every maximal ideal $\m\in \Spec(A)$.
\end{enumerate}
Furthermore if $\gr_{\m_0} A$ is reduced then these equivalent conditions hold.
\end{Lemma}
\begin{proof}
We prove (ii) $\Rightarrow$ (i) by contraposition. So suppose that $0 \ne f \in \Rad(A)$ is a nonzero element. Without loss of generality we may assume $f$ is homogeneous. By Lemma~\ref{L:technicallemma}(1) we have $\ker(A \to \prod_{\m\in \mSpec(A)} A_\m^\c) = 0$, which means that $f$ maps to a nonzero element of $A_\m^\c$ for some $\m$, implying $\m \in \Xi_f$. Using Lemma~\ref{L:technicallemma}(2) we see $\Xi_f$ is conic and closed, so it must contain $\m_0$, which implies that $A_{\m_0}^\c$ has nilpotent elements. This proves (ii) $\Rightarrow$ (i).

To see (i) $\Rightarrow$ (iii) we first of all observe that the property of being reduced is preserved by localisation for any commutative ring, and then apply \cite[\textsection 32, Remark~1]{Ma86}. Clearly (iii) $\Rightarrow$ (ii).

We prove that if $A$ is non-reduced then $\gr_{\m_0}A$ is non-reduced. Let $0 \ne f \in \Rad(A)$ and suppose (without loss of generality) that $f$ is homogeneous. There are two cases: either $f \in \m_0^i \setminus \m_0^{i+1}$ for some $i > 0$, or $f \in \bigcap_{i>0} \m_0^i$. In the first case $f + \m_0^{i+1}$ is a nonzero nilpotent element of $\gr_{\m_0} A$, which proves that $\gr_{\m_0} A$ is non-reduced.

We now show that the second case cannot happen. Suppose $f\in \bigcap_{i > 0} \m_0^i = \ker(A \to A_{\m_0}^\c)$, or equivalently, $\m_0 \notin \Xi_f$. By Lemma~\ref{L:technicallemma}(2) we must have $\Xi_f = \emptyset$ which implies that $f\in \ker(A \to A_\m^\c)$ for all maximal ideals $\m$. By Lemma~\ref{L:technicallemma}(1) we see $f \in \bigcap_{\m} \bigcap_{i> 0} \m^i = 0$. This contradiction completes the proof.
\end{proof}

\subsection{Commutative quotients of Poisson schemes}
\label{ss:poissonschemesandabeliansations}

If $A$ is a Poisson algebra then the {\it abelianisation $A^\ab$} is the largest Poisson commutative quotient of $A$. It is constructed by factoring by the ideal generated by $\{A, A\}$. 

When $X$ is a complex manifold it can be stratified into immersed submanifolds known as symplectic leaves (see \cite{We83} for an introduction). Therefore the closed points of a regular complex Poisson scheme can be decomposed into a disjoint union of symplectic leaves. More generally the closed points of a singular affine Poisson scheme can be decomposed into symplectic leaves by an iterative procedure \cite[\textsection 3.5]{BG03}.

The stratification by leaves can be coarsened to a stratification by rank, which is especially transparent in the case of affine schemes: if $A$ is finitely generated by elements $x_1,...,x_n$ and carries a Poisson structure, then we can form the matrix $\pi = (\{x_i, x_j\})_{1\leq i,j\leq n}$ and consider the subschemes $X_0 \subseteq X_1 \subseteq \cdots \subseteq X$ where $X_k$ is defined by the ideal of $A$ generated by all $(k+1) \times (k+1)$ minors of $\pi$.

It follows from \cite[Proposition~3.6]{BG03} that the locally closed set $X_k \setminus X_{k-1}$ is the union of all symplectic leaves of dimension $k$. In particular $X_0$ is the union of symplectic leaves of dimension zero. Note that $X_0 = \Spec(A^\ab)$ by definition. These observations prove the following lemma.


\begin{Lemma}
\label{L:0dleavesandabelianspec}
The following reduced subschemes of $X$ coincide:
\begin{enumerate}
\setlength{\itemsep}{4pt}
\item $\Spec(A^\ab)_\red$;
\item The union of all zero dimensional symplectic leaves of $\Spec(A)$.  $\hfill \qed$
\end{enumerate}
\end{Lemma}

 If $A$ is a Poisson algebra and $I$ is a Poisson ideal then the completion $A_I^\c$ can be equipped with a Poisson structure in a unique way such that $A \to A_I^\c$ is Poisson: each $A/I^i$ is a Poisson algebra and $A_I^\c$ is an inverse limit in the category of Poisson algebras. 

The next lemma states that Poisson abelianisation and completion commute. 

\begin{Lemma}
\label{L:abelianisationvscompletion}
Let $A$ be finitely generated and Poisson, and pick $\m \in \mSpec(A^\ab)$. Then there is a natural isomorphism
$$(A_\m^{\wedge})^\ab \isoto (A^\ab)_\m^\wedge.$$
\end{Lemma}
\begin{proof}
According to Lemma~\ref{L:0dleavesandabelianspec} the point $\m$ is a zero dimensional symplectic leaf, and it follows that $\m$ is a Poisson ideal. Hence $A_\m^\c$ is a Poisson algebra. We let $B$ denote the ideal of $A$ generated by the Poisson brackets $\{A, A\}$, let $\hB_\m$ be the ideal of $A^\c_\m$ generated by $\{A^\c_\m, A^\c_\m\}$ and let $B_\m^\c$ be the ideal $A^\c_\m \otimes_A B$, which we identify with an ideal of $A_\m^\c$ using \cite[Proposition~10.13]{AM69}. Considering the exact sequence of $A$-modules $B\to A \to A^\ab$, we deduce $(A^\ab)_\m^\c \cong A_\m^\c / B_\m^\c$ from \cite[Proposition~10.12]{AM69}.

We must show that $B_\m^\c = \hB_\m$. We certainly have an inclusion $B_\m^\c \subseteq \hB_\m$ and so the claim will follow from the universal property of $(A_\m^\c)^\ab$ if we can show that $A_\m^\c / B_\m^\c$ is an abelian Poisson algebra. Since this is isomorphic to $(A^\ab)_\m^\c$, which is a projective limit of abelian Poisson algebras, the claim follows.
\end{proof}

\section{Sheets and induction}
\label{s:sheetsandinduction}
\subsection{Lusztig--Spaltenstein induction and sheets}
\label{ss:LSinduction}

Let $G$ be a complex connected reductive group and $P \subseteq G$ a parabolic subgroup with Levi decomposition $P = LR$, where $R$ denotes the unipotent radical, and write $\l = \Lie(L)$, $\r = \Lie(R)$. If $\v \subseteq \g$ is any quasi-affine subvariety then we write $\v_\reg$ for the set of elements $v\in \v$ such that $\dim \Ad(G)v$ attains the maximal value.

For any choice of nilpotent orbit $\O_0 \subseteq \l$ and $z\in \z(\l)$ it is not hard to see that $\Ad(G) (z + \O_0 + \r)$ contains a dense $G$-orbit. This orbit is {\it (Lusztig--Spaltenstein) induced from $(\l, \O_0, z)$} and is denoted $\Ind_{\l,z}^\g(\O_0)$. Remarkably the induced orbit depends only on the $G$-conjugacy class of $(\l, \O_0, z)$, not on the choice of $P$ admitting $L$ as a Levi factor \cite[Lemma~4.1]{Lo22}. We call this $G$-orbit an {\it induction datum}, and when $z = 0$ we say that the orbit of $(\l, \O_0)$ is the induction datum with induced orbit $\Ind_\l^\g(\O_0)$, suppressing $z$.

If an orbit cannot be obtained by induction from a proper Levi subalgebra then it is called {\it rigid}, otherwise it is called {\it induced}. Rigid orbits are necessarily nilpotent. We say that an induction datum $(\l, \O_0)$ is {\it rigid} if $\O_0$ is a rigid orbit in $\l$. When $G$ is almost simple of classical type the conjugacy classes of Levi subalgebras and nilpotent orbits can be described combinatorially, and induction can be totally understood in terms of partitions \cite[\textsection 7]{CM93}. 

The sheets of a Lie algebra are the irreducible components of the rank varieties which consist of points whose orbit has a fixed dimension. These were classified in terms of rigid induction data by Borho \cite{Bo81}. For each induction datum $(\l, \O_0)$ define
\begin{eqnarray*}
& & \D(\l, \O_0):= \Ad(G)(\z(\l)_\reg + \O_0 + \r)\\
& & \Sc(\l,\O_0) := \overline{\D(\l,\O_0)}_\reg.
\end{eqnarray*}
\begin{Theorem} \cite[Satz~4.3, Satz~5.6]{Bo81}
\label{T:Borhostheorem}
The sheets of $\g$ are the sets $\Sc(\l, \O_0)$ where $(\l, \O_0)$ varies over conjugacy classes of rigid data. Furthermore $\dim \Sc(\l, \O_0) = \dim \z(\l) + \dim \Ind_{\l}^\g(\O_0)$.
\end{Theorem}


\subsection{Classical finite $W$-algebras and Slodowy slices}
\label{ss:classicalandslices}
Resume the notation $\kappa, \{e,h,f\}, \chi$ used in Section~\ref{ss:classicalfiniteWalg}, so $\g = \bigoplus_{i\in \Z} \g(i)$ is the $\ad(h)$-grading.

The {\it Slodowy slice} is the affine variety $e+ \g^f$. Using $\kappa$ we can identify $\g$ with $\g^*$ as $G$-modules, so that $S(\g)$ is identified with the coordinate ring $\C[\g]$. Then $S(\g) \g(\lemw)_\chi$ is identified with the defining ideal of
$$e + \g(\lemw)^\perp = \{x\in \g \mid \kappa(x - e, \g(\lemw)) = 0 \} = e + \g(\le1).$$
It follows from \cite[Lemma~2.1]{GG02} that the adjoint action map of $G$ restricts to an isomorphism
$$G(\lez) \times e + \g^f \isoto e + \g(\le1).$$
Via this isomorphism we identify $S(\g,e) = \C[e+\g^f]$, and thus we equip $e+\g^f$ with a Poisson structure.

Let $\gamma : \C^\times \to G$ be the cocharacter with differential $d_1\gamma(1) = h$. There is a $\C^\times$-action on $\g$ given by $t\mapsto t^{2} \Ad(\gamma(t)^{-1})$ which restricts to a contracting $\C^\times$-action on $e+\g^f$. The grading on $\C[e+\g^f]$ induced by this action coincides with the Kazhdan grading on $S(\g,e)$.

Thanks to \cite[\textsection 3.2]{GG02} we know that the symplectic leaves of $e+\g^f$ coincide with the irreducible components of the intersection of adjoint orbits with the slice. Since $S(\g,e)$ is Poisson graded in degree $-2$ this shows that the $\C^\times$-action permutes the symplectic leaves of $e + \g^f$.

\subsection{The Katsylo variety and the tangent cone}

The Slodowy slice $e + \g^f$ reflects the local geometry of $\g$ in a neighbourhood of the adjoint orbit of $e$. Let $\Sc_1,...,\Sc_l$ be the sheets of $\g$ containing $e$. Recall from the introduction that we define the {\it Katsylo variety} to be
\begin{eqnarray}
\label{e:Katsylovariety}
e + X := (e + \g^f) \cap \bigcup_{i=1} \Sc_i.
\end{eqnarray}
Let $\m_0$ denote the maximal ideal of $\C[e+X]$ corresponding to $e$.
One of the main tools in this paper is the tangent cone $\TC_e(e+X) = \Spec (\gr_{\m_0} \C[e+X])$, which we equip with the reduced scheme structure.

\begin{Proposition}
\label{P:Katsyloproperties}
\begin{enumerate}
\setlength{\itemsep}{4pt}
\item $\Spec S(\g,e)^\ab_\red = e + X$ as reduced schemes.
\item Let $\g$ be a classical Lie algebra. We have bijections
$$\Comp(\! \ \bigcup_{i=1}^l \Sc_i) \onetoone \Comp(e+X) \onetoone \Comp\TC_e(e+X) .$$
The first bijection reduces dimension by $\dim \Ad(G) e$ whilst the second is dimension preserving.
\end{enumerate}
\end{Proposition}
\begin{proof}
By Lemma~\ref{L:0dleavesandabelianspec} we know that the reduced scheme  associated to $\Spec S(\g,e)^\ab$ is the union of zero dimensional symplectic leaves of $e + \g^f$. By \cite[\textsection 3.2]{GG02} we know that the symplectic leaves of $e + \g^f$ are the irreducible components of the non-empty intersections $\O \cap (e + \g^f)$ where $\O \subseteq \g$ is an adjoint orbit. Since $e + \g^f$ intersects orbits transversally it follows that the zero dimensional leaves are precisely the components of $\O \cap (e+\g^f)$ where $\dim \O = \dim \Ad(G)e$. Since the contracting $\C^\times$-action permutes the leaves of $e + \g^f$, these are precisely the orbits lying in the sheets containing $e$. This proves (1).

To prove the first bijection in (2) it suffices to show that for every sheet $\Sc_i$ the intersection $\Sc_i \cap (e + \g^f)$ is irreducible. It was proven in \cite[Theorem~6.2]{Im05} that $\Sc_i \cap (e + \g^f)$ can be expressed as the image of a morphism from an irreducible variety, which gives the first bijection.  The main result of {\it loc. cit.} shows that the varieties $\Sc_i$ are smooth. Since they are smoothly equivalent to $\Sc_i \cap (e+\g^f)$ the latter are also smooth, and so $$\TC_e(e + X) = \bigcup_{i=1}^l \T_e(\Sc_i \cap e + \g^f).$$ Now the second bijection will follow if we can show that $\T_e(\Sc_i \cap e + \g^f)$ is never contained in $\T_e(\Sc_j \cap e + \g^f)$ for $i \ne j$ which is a consequence of the fact that the sheets $\Sc_1,...,\Sc_l$ intersect transversally at $e$ \cite{Bu}. Finally the claims about dimensions follow from the fact that $e + \g^f$ is transversal to adjoint orbits and the sheets are smooth.
\end{proof}

\section{Abelian quotients of finite $W$-algebras }
\label{s:quantumabelianquotients}

\subsection{The finite $W$-algebra}
\label{ss:finiteWalgebras}
We now recall the definition of the (quantum) finite $W$-algebra. Once again $G$ is a connected complex reductive group with simply connected derived subgroup. A nondegenerate trace form on $\g$ is denoted $\kappa$. We pick an $\sl_2$-triple $(e,h,f)$ in $\g = \Lie(G)$ and consider the grading $\g = \bigoplus_{i\in \Z} \g(i)$ given by $\ad(h)$-eigenspaces. Write $\chi := \kappa(e, \cdot) \in \g^*$. Recall the notation $\g(\lemw)$ and $G(<0)$.

The generalised Gelfand--Graev module is defined to be
$$Q := U(\g) / U(\g) \g(\lemw)_\chi = U(\g) \otimes_{U(\g(<\!-1))} \C_\chi$$
where $\C_\chi$ is the one dimensional representation of $\g(\lemw)$ afforded by $\chi$. The finite $W$-algebra is defined to be
$$U(\g,e) := Q^{G(< 0)}.$$
It inherits a natural algebra structure from $U(\g)$ via $\bar{u}_1 \bar{u}_2 = \overline{u_1 u_2}$ where $\bar{u}$ denotes the projection of $u \in U(\g)$ to $Q$, and $\bar u_1, \bar u_2 \in Q^{G(<0)}$.
\begin{Remark}
The finite $W$-algebra is more commonly defined as a quantization of a Hamiltonian reduction with respect to a certain subgroup of $G$ lying between $G(\!<\!\!-1)$ and $G(\lez)$. In \cite[Theorem~4.1]{GG02} Gan--Ginzburg show that this definition is equivalent to the one given here: in their notation our definition corresponds to the isotropic space $\ell = 0$.
\end{Remark}

The Kazdan filtration on $U(\g)$ is defined by placing $\g(i)$ in degree $i + 2$. This descends to a non-negative filtration on $Q$ and $U(\g,e)$ and we write $\gr U(\g,e)$ for the associated graded algebra. The graded algebra $\gr U(\g,e)$ naturally identifies with a subalgebra of $\gr Q = S(\g) / S(\g) \g(\lemw)_\chi$, and we have the following fundamental fact (see \cite[Proposition~6.3]{Pr02}, \cite[Theorem~4.1]{GG02}).
\begin{Lemma}
\label{L:quantizes}
$\gr U(\g,e) = S(\g,e)$ as Kazhdan graded algebras.
\end{Lemma}
One of the main objects of study in this paper is the maximal abelian quotient $U(\g,e)^\ab$, which is defined to be the quotient by the derived ideal which is generated by all commutators $[u,v]$ with $u,v \in U(\g,e)$. 

The Kazhdan filtration descends to $U(\g,e)^\ab$ and we write $\gr U(\g,e)^\ab$ for the associated Kazhdan graded algebra. Our main results describe the structure of the following affine schemes
\begin{eqnarray*}
& & \E(\g,e) := \Spec U(\g,e)^\ab,\\
& & \C\E(\g,e) := \Spec \gr (U(\g,e)^\ab).
\end{eqnarray*}
The closed points of the first of these parameterises the one dimensional representations of $U(\g,e)$, whilst the second is the asymptotic cone of the first (cf. Section~\ref{ss:degenerationsofschemes}).

\subsection{Bounding the asymptotic cone}
The next result is equivalent to \eqref{e:coneimmersedinKatsylo}, and we view it as a version of Premet's theorem \cite[Theorem~1.2]{Pr10} for the asymptotic cone of $\E(\g,e)$. Our methods are adapted from his work. 
\begin{Theorem}
\label{T:semiclassicalcomponents}
If $\g = \Lie(G)$ is the Lie algebra of a complex reductive group then there is a surjective homomorphism $$S(\g,e)^\ab \onto \gr U(\g,e)^\ab$$ which induces an isomorphism of reduced algebras.
\end{Theorem}

First of all we prove the existence of the surjection
\begin{Lemma}
\label{L:abquotsurjection}
There is a surjective homomorphism $S(\g,e)^\ab \onto \gr U(\g,e)^\ab$.
\end{Lemma}
\begin{proof}
Let $\theta : \g^e \to S(\g,e)$ be a PBW map coming from Theorem~\ref{T:PBW} and let $\Theta : \g^e \to U(\g,e)$ be a filtered lift of $\theta$. Using the Leibniz rule for the biderivation $[\cdot, \cdot]$ we see that the derived ideal $D$ of $U(\g,e)$ is generated by the set $\{[\Theta(u), \Theta(v)] \mid u,v \in \g^e\}$. Similarly the bracket ideal $B$ in $S(\g,e)$ is generated by $\{\{\theta(u),\theta(v)\} \mid u,v\in \g^e\}$. If $u, v$ are homogeneous for the $\ad(h)$-grading on $\g$ then the image of $[\Theta(u), \Theta(v)]$ in the graded algebra $\gr U(\g,e) = S(\g,e)$ is $\{\theta(u), \theta(v)\}$, using the identification of Lemma~\ref{L:quantizes}. This shows that $B \subseteq \gr D$. Hence we have $\gr U(\g,e)^\ab = S(\g,e) / \gr D \onto S(\g,e) / B = S(\g,e)^\ab$.
\end{proof}

In order to show that this surjection gives an isomorphism of reduced algebras we use reduction modulo $p$. This process was pioneered by Premet in the theory of $W$-algebras and used systematically to great effect in subsequent work \cite{Pr07, Pr10, To17, PT21}. Rather than repeating the details in full we state the important properties of the modular reduction procedure which we need for our proof.

Let $h$ denote the greatest Coxeter number as we vary over simple factors of the derived subgroup $(G, G)$. 
Pick a Chevalley $\Z$-form $\g_\Z \subseteq \g$ and write $\g_\Z = \mathfrak{h}_\Z \oplus \bigoplus_{\alpha \in \Phi} \Z x_\alpha$. Let $T\subseteq G$ denote the complex torus with Lie algebra $\mathfrak{h}_\Z \otimes_\Z \C$. When $\k$ is an algebraically closed field we let $G_\k$ denote the reductive algebraic group with Lie algebra $\g_\k := \g_\Z \otimes_\Z \k$ and let $T_\k$ denote the algebraic torus in $G_\k$ with Lie algebra $\mathfrak{h}_\Z \otimes_\Z \k$.

It follows from the observations of \cite[\textsection 3.3]{PT21} that for every nilpotent orbit $\O \subseteq \g$ there is an element $e\in \g_\Z \cap \O$ and a cocharacter $\lambda \in X_*(T)$ such that for all algebraically closed fields $\k$ of characteristic $p > h$:
\begin{enumerate}
\item The Bala--Carter labels of the adjoint orbits of $e$ and $e_{\k} := e\otimes 1 \in \g_{\k}$ coincide. 
\item The differential $d_1\lambda(\C)$ is a semisimple element of $\g$ lying in an $\sl_2$-triple containing $e$.
\item After canonically identifying to cocharacter lattices $X_*(T) = X_*(T_\k)$ the grading on $\g_\k$ induced by $\lambda$ is good for $e$, meaning $e_\k\in \g_\k(2)$ and $\g_\k^e \subseteq \g_\k(\ge 0)$.
\end{enumerate}
We mention that slightly weaker assumptions are sufficient for (1), (2), (3). One may work under the standard hypotheses of \cite[\textsection 6.3]{Ja04}, however we avoid these technicalities for the sake of simplicity.

Fix an orbit $\O \subseteq \g$ and a choice of $e, \lambda$ as above. When the characteristic of $\k$ is $p > h$ there is a $G$-equivariant trace form $\kappa_\k$, for example the Killing form is non-degenerate on $[\g, \g]$ and $\g = [\g, \g] \oplus \z(\g)$ under out hypothesis. Write $\chi_\k$ for the element of $\g_\k^\ast$ corresponding to $e_\k$ via $\kappa_\k$. Using the grading $\g_\k = \bigoplus_{i\in \Z} \g_\k(i)$ coming from $\lambda$ we let $\v$ denote a homogeneous complement to $[\g_\k,e_\k]$ in $\g_\k$. If $\Sc_1,...,\Sc_l$ denote the sheets of $\g_\k$ containing $e_\k$ then we define
$$e_\k + X_\k = (e_\k + \v) \cap \bigcup_{i=1}^l \Sc_i.$$
In the following we write $e + X_\C$ for the complex Katsylo variety \eqref{e:Katsylovariety}, and write $X_\k^\vee\subseteq \g^*_\k$ for the image of $X_\k$ under the $G$-equivariant isomorphism $\g_\k \to \g_\k^*$ coming from the trace form.
\begin{Lemma}
\label{L:dimensionpreservingbijections}
For $\Char(\k) \gg 0$ there are dimension preserving bijections
\begin{eqnarray}
\label{e:firstbij}
\Comp \C\E(\g,e) & \onetoone & \Comp \k \E(\g_{\k},e_\k) \\
\label{e:secondbij}
\Comp X_\C & \onetoone & \Comp X_{\k}.
\end{eqnarray}
\end{Lemma} 
\begin{proof}
The first bijection \eqref{e:firstbij} can be obtained by reciting the proof of \cite[Lemma~3.1]{Pr10} verbatim, whilst the second \eqref{e:secondbij} was constructed in the proof of \cite[Theorem~3.2]{Pr10}.
\end{proof}

Since $\g_\k = \Lie(G_\k)$ is algebraic there is a $G_\k$-equivariant restricted structure on $\g_\k$. The $p$-centre $Z_p(\g_\k)$ is a central subalgebra of $U(\g_\k)$ which identifies with $\k[(\g^*_\k)^{(1)}]$, the coordinate ring on the Frobenius twist of the dual space of $\g_\k$, as $G_\k$-algebras.

The modular finite $W$-algebra is defined in precisely the same manner as the complex analogue: $U(\g_\k, e_\k) = Q_\k^{G_\k(< 0)}$ where $Q_\k = U(\g_\k) / U(\g_\k) \g_\k(\lemw)_\chi$ and $\g_\k(\lemw)_\chi = \{ x - \chi_\k(x) \mid x\in \g_\k(\lemw)\}$ \cite{Pr10, GT19a}. If $\phi : U(\g_\k) \to Q_\k$ is the natural projection then the $p$-centre of $U(\g_\k, e_\k)$ is defined to be $(\phi Z_p(\g_\k))^{G_\k(<0)}$; see \cite[\textsection 8]{GT19a} for more detail. It follows from Lemma~8.2 of {\it loc. cit.} that $Z_p(\g_\k,e_\k) = \k[(\chi_\k + \v^\vee)^{(1)}]$ where $\v^\vee = \kappa_\k(\v, \cdot) \subseteq \g^*$, and the maximal ideals of the $p$-centre will be referred to as $p$-characters. For $\eta\in \chi_\k + \v^\vee$ we define the reduced finite $W$-algebra $U_\eta(\g_\k,e_\k)$ to be the quotient of $U(\g_\k,e_\k)$ by the ideal generated by the corresponding maximal ideal of $Z_p(\g_\k,e_\k)$ (identifying $\chi_\k + \v^\vee$ with its Frobenius twist, as sets).
\begin{Lemma}
\label{L:finitedominantcharp}
There is a finite, dominant morphism $$\k \E(\g_{\k}, e_{\k}) \to X_{\k}.$$
\end{Lemma}
\begin{proof}
Consider the homomorphism $Z_p(\g_\k,e_\k) \to U(\g_\k,e_\k)^\ab$ and denote the kernel by $K$. Write $V(K) \subseteq \mSpec Z_p(\g_\k,e_\k)$ for the algebraic subvariety defined by $K$. By definition the points of $V(K)$ correspond to the $p$-characters $\eta \in \chi_\k + \v^\vee$ such that $U_\eta(\g_\k, e_\k)$ admits a one dimensional module. In \cite[\textsection 2.6]{Pr10} Premet demonstrated the following relationship between the reduced enveloping algebra and reduced finite $W$-algebra
\begin{eqnarray}
\label{e:Premetsthm}
U_\eta(\g_\k) \cong \Mat_{p^{d_{\chi_\k}}} U_\eta(\g_\k, e_\k)
\end{eqnarray}
where $d_{\chi_\k} := \frac{1}{2} \dim \Ad^*(G_\k) \chi_\k$, under the assumption $p \gg 0$. In \cite[Remark~9.4]{GT19a} the isomorphism \eqref{e:Premetsthm} was recovered under weaker hypotheses ($p > h$ suffices). It follows that the points of $V(K)$ correspond to $p$-characters $\eta \in \chi_\k + \v^\vee\subseteq \g_\k^*$ such that $U_\eta(\g_\k)$ admits a module of dimension $p^{d_{\chi_\k}}$.

In \cite{Pr95} Premet provided his first proof of the second Kac--Weisfeiler conjecture (he later refined his argument using $W$-algebras, see Theorem~2.3(ii) and \textsection 2.6 of \cite{Pr02}). From this, we know that every $U_\eta(\g_\k)$-module has dimension divisible by $p^{d_\eta}$. On the other hand, \cite[Theorem~1.1]{PT21} states that every $U_\eta(\g_\k)$ has a module of dimension $p^{d_\eta}$. Therefore
$$V(K) = \{\eta \in \chi_\k + \v^\vee \mid \dim \Ad^*(G_\k) \chi_\k = \dim \Ad^*(G_\k) \eta\}.$$
We claim that this set is equal to $(\chi_\k + X^\vee_\k)^{(1)}$. In the current setting, $\chi_\k + \v^\vee$ admits a contracting $\k^\times$-action defined in the same manner as the one parameter group of automorphisms appearing in Section~\ref{ss:classicalandslices}, using the cocharacter determined by the grading in (3) above. This contracting action preserves the sheets of $\g_\k^*$ and this implies the claim.

Let $R$ denote the reduced algebra corresponding to $Z_p(\g,e)/K$. Since $Z_p(\g,e) = \k[(\chi_\k + \v^\vee)^{(1)}]$ as Kazhdan filtered algebras we see that $R = \k[(\chi_\k + X_\k^\vee)^{(1)}]$ is also an identification of Kazhdan filtered algebras. Since $\chi_\k + X_\k^\vee$ is stable under the contracting $\k^\times$-action inducing the Kazhdan grading, it follows that $R \cong \gr R$. We conclude that there is a finite injective algebra homomorphism $R \hookrightarrow \gr U(\g,e)^\ab$, which completes the proof.
\end{proof}

\begin{proofofsemiclassicalcomponentmap}
Recall that finite morphisms are closed, and so a finite dominant morphism is surjective.  Now it follows from Lemma~\ref{L:finitedominantcharp} that there is a surjective map $\Comp \k\E(\g_\k, e_\k) \to \Comp X_\k$ such that for every $Z\in \Comp X_\k$ there exists an element of $\Comp \k\E(\g_\k, e_\k)$ mapping to $Z$ of the same dimension. Combining with Lemma~\ref{L:dimensionpreservingbijections} we deduce that there is surjection $\Comp \C\E(\g,e) \onto \Comp(e + X_\C)$ which respects dimensions in the same manner. In particular we have $d(\C\E(\g,e)) \ge d(e + X_\C)$, where $d$ denotes the dimension vector defined in Section~\ref{ss:degenerationsofschemes}. Thanks to Lemma~\ref{L:closedembeddings}, Proposition~\ref{P:Katsyloproperties}(1) and Lemma~\ref{L:abquotsurjection} the embedding $\C\E(\g,e) \into \Spec S(\g,e)^\ab = e + X_\C$ is an isomorphism on the underlying reduced schemes. $\hfill \qed$
\end{proofofsemiclassicalcomponentmap}

The next result follows immediately from Theorem~\ref{T:semiclassicalcomponents}.
\begin{Corollary}
\label{P:relatingabelianisations}
If $S(\g,e)^\ab$ is reduced then so are $\gr U(\g,e)^\ab$ and $U(\g,e)^\ab$. $\hfill \qed$
\end{Corollary}

\section{The abelian quotient of the classical $W$-algebra}
\label{s:abquotsection}

In this section we study the finite $W$-algebras associated to classical Lie algebras, and so we refresh the notation once again. Let $N > 0$ is an integer and $\ve \in \{\pm 1\}$ such that $\ve^N = 1$. We let $G \subseteq \GL_N$ be a classical simple algebraic group and $\g = \Lie G$ such that
$$\g \cong \left\{ \begin{array}{rc} \so_N & \text{ if } \ve = 1, \\ \sp_N & \text{ if } \ve = -1.\end{array}\right.$$

\subsection{Partitions for nilpotent orbits in classical Lie algebras}
\label{ss:partitionsfornilpotentorbits}

We refer to the $G$-orbits in $\g$ consisting of nilpotent elements as {\it nilpotent orbits}, and similarly for $\GL_N$-orbits in $\gl_N$. Nilpotent $\GL_N$-orbits in $\gl_N$ are classified by partitions $\mathcal{P}(N) := \{\lambda \vdash N\}$, where the parts of a partition $\lambda$ correspond to the sizes of the Jordan blocks of elements in the orbit.

For each nilpotent $\GL_N$-orbit $\O_\lambda \subseteq\g$ the intersection $\O_\lambda \cap \g$ is either empty or it is a union of either one or two nilpotent $G$-orbits. Therefore an approximate classification of nilpotent $G$-orbits is achieved by describing the partitions $\lambda \vdash N$ for which $\O_\lambda \cap \g\ne \emptyset$. The set of such partitions is denoted $\P_\ve(N)$, and they are characterised as follows.
\begin{Lemma}
\label{L:involutionforlambda}
Let $\lambda = (\lambda_1,...,\lambda_n) \vdash N$. Then $\lambda \in \PeN$ if and only if there exists an involution $i \mapsto i'$ on the set $\{1,...,n\}$ such that:
\begin{enumerate}
\setlength{\itemsep}{4pt}
\item $\lambda_i = \lambda_{i'}$ for all $i = 1,...,n$;\smallskip
\item $i = i'$ if and only if $\varepsilon(-1)^{\lambda_i} = 1$;\smallskip
\item $i' \in \{i-1, i, i+1\}$. $\hfill \qed$
\end{enumerate}
\end{Lemma}
Whenever we choose $\lambda \in \PeN$ we will assume that a choice of involution has been fixed in accordance with Lemma~\ref{L:involutionforlambda}. We also adopt the convention that our partitions are non-decreasing $\lambda_1 \le \cdots \le \lambda_n$. For completeness we mention that $\lambda \in \PeN$ corresponds to one orbit unless $\ve = 1$ and all parts of $\lambda$ are even, which case there are two $G$-orbits of nilpotent elements of $\g$ with Jordan blocks given by $\lambda$ and these are permuted by the outer automorphism group of $\g$, see \cite[\textsection 5.1]{CM93} for example.

\subsection{Rigid, singular and distinguished partitions}
We introduce three classes of partition which correspond to important families of nilpotent orbit in classical Lie algebras of type {\sf B, C, D}.

Let $\lambda = (\lambda_1,...,\lambda_n) \in \PeN$ and adopt the convention that $\lambda_0 = 0$ and $\lambda_{n+1} = \infty$. A {\it 2-step} is a pair of indexes $(i, i+1)$ such that:
\begin{itemize}
\item $\ve(-1)^{\lambda_i} = \ve(-1)^{\lambda_{i+1}} = -1$;
\item $\lambda_{i-1} < \lambda_i \le \lambda_{i+1} < \lambda_{i+2}$.
\end{itemize}
We note that $\lambda_0 < \lambda_1$ and $\lambda_n < \lambda_{n+1}$ always hold due to our conventions. We write $\Delta(\lambda)$ for the set of 2-steps for $\lambda$. A 2-step $(i,i+1) \in \Delta(\lambda)$ is called {\it bad} if one of the following conditions holds:
\begin{itemize}
\item $\lambda_i - \lambda_{i-1}  > 0$ is even;
\item $\lambda_{i+2} - \lambda_{i+1} > 0$ is even.
\end{itemize}
We note that the second condition never holds for $i = n-1$, whilst if the first condition holds for $i = 1$ then we necessarily have $\ve = -1$. If the first of these two conditions holds then we refer to $i-1$ as the {\it bad boundary} of the 2-step whilst if the second condition holds then $i+2$ is the bad boundary. We say that a partition is {\it singular} if it admits a bad 2-step, and non-singular otherwise.

We say that a partition is {\it rigid} if the following two conditions hold:
\begin{itemize}
\item $\Delta(\lambda) = \emptyset$;
\item $\lambda_i - \lambda_{i-1} < 2$ for all $i = 1,...,n$.
\end{itemize}

We say that a partition is {\it distinguished} if the following two conditions hold:
\begin{itemize}
\item $\lambda_i < \lambda_{i+1}$ for all $i=1,...,n$;
\item $i = i'$ for all $i$.
\end{itemize}

The following result may be gathered from \cite[Theorem~3]{PT14}, \cite[Lemma~4.2]{Ja04} and \cite[Corollary~7.3.9]{CM93}.
\begin{Lemma} 
\label{L:partitionmeaning}
Let $e$ be a nilpotent element with Jordan block sizes given by the partition $\lambda \in \PeN$ in a classical Lie algebra preserving a non-degenerate form $\Psi$ on $\C^N$ such that $\Psi(u,v) = \ve \Psi(v,u)$ for $u,v\in \C^N$. The following hold:
\begin{enumerate}
\setlength{\itemsep}{4pt}
\item[(i)]  $e$ lies in a unique sheet of $\g$ if and only if $\lambda$ is non-singular.
\item[(ii)] $e$ is distinguished in the sense of Bala--Carter theory if and only if $\lambda$ is distinguished.
\item[(iii)] $e$ is rigid in the sense of Lusztig--Spaltenstein induction if and only if $\lambda$ is rigid. $\hfill \qed$
\end{enumerate}
\end{Lemma}

\subsection{The Kempken--Spaltenstein algorithm}

Keep fixed a choice of $\lambda = (\lambda_1,...,\lambda_n) \in \PeN$. We say that $i$ is an {\it admissible index for $\lambda$} if Case 1 or Case 2 occurs:
\begin{eqnarray*}
\textbf{Case 1:} & & \lambda_i - \lambda_{i-1} >1;\\
\textbf{Case 2:} & & (i-1, i) \in \Delta(\lambda) \text{ and } \lambda_{i-1} = \lambda_{i}.
\end{eqnarray*}
When $i$ is an admissible index we define $\lambda^{(i)} \in \P_\epsilon(N - 2(n-i+1))$ as follows:
\begin{eqnarray*}
\textbf{Case 1:} & &  \lambda^{(i)} = (\lambda_1, \lambda_2,...,\lambda_{i-1}, \lambda_i - 2, \lambda_{i+1}-2, ..., \lambda_{n}-2); \\
\textbf{Case 2:} & & \lambda^{(i)} = (\lambda_1, \lambda_2,..., \lambda_{i-2}, \lambda_{i-1} - 1, \lambda_{i}-1, \lambda_{i+1}-2,..., \lambda_{n}-2)
\end{eqnarray*}
Now we extend this definition inductively from indexes to sequences. We say that $(i)$ is an admissible sequence for $\lambda$ if $i$ is an admissible index. We inductively extend these definitions to sequences of indexes by saying that $\i = (i_1,...,i_l)$ is an {\it admissible sequence for $\lambda$} and define $\lambda^\i = (\lambda^{(i_1,...,i_{l-1})})^{(i_l)}$ if $\i' = (i_1,...,i_{l-1})$ is an admissible sequence for $\lambda$ and $i_l$ is an admissible index for $\lambda^{\i'}$. A sequence $\i$ is called {\it maximal admissible} if $\lambda^\i$ does not admit any admissible indexes. These definitions first appeared in \cite{PT14}, where the opposite ordering on the parts of $\lambda$ was used.

We define {\it admissible multisets} to be the multisets taking values in $\{1,...,n\}$ obtained from admissible sequences by forgetting the ordering. If $\i$ is an admissible sequence then write $[\i]$ for the corresponding admissible multiset.

For $\ve = \pm 1$ let $\g$ be a classical Lie algebra in accordance with Lemma~\ref{L:partitionmeaning}, and let $\O \subseteq \g$ be a nilpotent orbit with partition $\lambda$. According to Proposition~8 and Corollary~8 of \cite{PT14}, for each addmisible multiset $[\i]$, there is a procedure for choosing:
\begin{enumerate}
\item a Levi subalgebra $\l_\i \cong \gl_{i_1} \times \cdots \gl_{i_l} \times \g_+$, where $\g_+$ is a classical Lie algebra of the same Dynkin type as $\g$ and natural representation of dimension $N - 2\sum i_j$.
\item a nilpotent orbit $\O_\i = \{0\} \times \cdots \{0\} \times \O_{\lambda^\i}$ where $\O_{\lambda^\i}\subseteq \g_+$ has partition $\lambda^\i$.
\end{enumerate}
These choices satisfy $\O = \Ind_{\l}^\g(\O_\i)$. Furthermore $[\i] \mapsto (\l_\i, \O_\i)$ sets up a one-to-one correspondence between maximal admissible multisets and rigid induction data for $\O$. Combining with Theorem~\ref{T:Borhostheorem} this proves.
\begin{Proposition}
\label{P:KSclassifiessheets}
The sheets of $\g$ containing $e$ are in bijection with maximal admissible multisets for $\lambda$. The dimension of sheet corresponding to $\i$ is $|\i| + \dim (\Ad(G)e)$. $\hfill \qed$
\end{Proposition}

\subsection{Distinguished elements and induction}

The proof of the first main theorem will be reduced to the distinguished case, and the following result is one of the crucial steps.

\begin{Lemma}
\label{L:distinguishedinduction}
Let $\O\subseteq \g$ be a nilpotent orbit with partition $\lambda \in \PeN$ in a classical Lie algebra of rank $r$. There exists a sequence of indexes $i_1,...,i_l$ and a classical Lie algebra $\tilde \g$ of rank $r + \sum_j (n- i_j + 1)$ of the same Dynkin type as $\g$ such that:
\begin{enumerate}
\setlength{\itemsep}{4pt}
\item $\gl_{n-i_1+1} \times \cdots \times \gl_{n-i_l+1} \times \g$ embeds as a Levi subalgebra $\l \subseteq \tilde\g$;
\item Identifying $\O$ with the as a nilpotent orbit $\{0\} \times \cdots  \times \{0\} \times \O \subseteq \l$ we have that $\Ind_\l^{\tilde\g}(\O)$ is a distinguished orbit in $\tilde\g$.
\end{enumerate}
Furthermore we can assume that the first part of the partition of $\Ind_\g^{\tilde\g}(\O)$ is arbitrarily large.
\end{Lemma}
\begin{proof}
If $i = 2,...,n$ is an index such that $i' = i-1$ then we replace $\lambda$ by
$$(\lambda_1,...,\lambda_{i-2}, \lambda_{i-1}+1, \lambda_{i} + 1, \lambda_{i+1} + 2, ..., \lambda_n + 2).$$
By iterating this procedure we can eventually obtain a partition with $i = i'$ for all $i$. If $i = 1,2,...,n$ is an index such that $\lambda_i = \lambda_{i-1}$ then we replace our new partition by
$$(\lambda_1,...,\lambda_{i-1}, \lambda_i + 2, \lambda_{i+1}+2,...,\lambda_n + 2) .$$
Iterating this procedure we eventually replace $\lambda$ with a distinguished partition $\tilde\lambda \in \P_\ve(N + \sum_{i\in \i} 2(n-i+1))$ where $\i \subseteq \{1,...,n\}$ is some multiset consisting of indexes which were chosen in the above iterations. By applying the latter procedure at index $i = 1$ we can make the first part of $\tilde\lambda$ as large as we choose.

Let $\tilde \g$ be the classical Lie algebra of the same type as $\g$ with natural representation of dimension $N + \sum_{j\in \i} 2(n-j+1)$. By Lemma~\ref{L:partitionmeaning} every orbit with partition $\tilde\lambda$ is distinguished. Since $\tilde\lambda^\i = \lambda$ the remarks of the previous section show that there is a unique orbit $\widetilde{\O} \subseteq \tilde\g$ with partition $\lambda$ and a unique induction datum $(\l, \O)$ such that $\l$ and $\O$ satisfy the properties described in the current proposition, and $\widetilde{\O} = \Ind_{\l}^{\tilde\g}(\O)$.
 \end{proof}

\subsection{The combinatorial Katsylo variety}
\label{ss:combKat}

In this section we introduce an affine algebraic variety $X_\lambda$, determined by a choice of partition, which we call the {\it combinatorial Katsylo variety.} In Theorem~\ref{T:distinguishedreduced} we shall use this variety to relate the $\m_0$-adic graded algebra of $S(\g, e)^\ab$ with the tangent cone $\TC_e(e+X)$ of the Katsylo variety, where $\m_0 \subseteq S(\g, e)$ denotes the maximal Kazhdan graded ideal and the orbit of $e\in \g$ has partition $\lambda$.

Let $\lambda = (\lambda_1,...,\lambda_n) \in \PeN$ such that $i = i'$ for all $i$, all parts of $\lambda$ are distinct, and $\lambda_1 > 1$. In particular $\lambda$ is distinguished. This condition ensures that $\lambda_i - \lambda_{i-1} > 0$ is even for all $i =2,...,n$ and that $(i,i+1) \in \Delta(\lambda)$ for $i=1,...,n-1$. We set
\begin{eqnarray}
\label{e:partdifferences}
\begin{array}{lcc}
s_1 := \lfloor \frac{\lambda_1}{2}\rfloor; & & \vspace{3pt}\\
s_i := \frac{\lambda_{i} - \lambda_{i-1}}{2} & \text{ for } & i=2,...,n.
\end{array}
\end{eqnarray}
Introduce a set of variables
\begin{eqnarray}
\label{e:kXlambdagenerators}
S_\lambda := \left\{x_{i, r}\mid 1 \leq i \leq n, \ 1 \leq r \leq s_i \right\} \cup \left\{y_j \mid j=1,...,n-1\right\}
\end{eqnarray}
and we consider the following collection of quadratic elements $Q_\lambda$ in the polynomial ring $\C[S_\lambda]$
\begin{eqnarray}
\label{e:kXlambdarelations}
\begin{array}{rcl} Q_\lambda & := & \{x_{i,1} y_i \mid i=1,...,n-1, \ \lambda_{i} - \lambda_{i-1}  > 0 \text{ is even}\}  \\ & & \ \ \cup \ \{x_{i+2,1} y_i, y_i y_{i+1} \mid i=1,...,n-2\}.\end{array}
\end{eqnarray}
Finally we define $X_\lambda$ to be the algebraic variety with coordinate ring
\begin{eqnarray}
&\C[X_\lambda] := \C[S_\lambda] / (Q_\lambda).
\end{eqnarray}

The main goal of this section is to describe the irreducible components of $X_\lambda$, and calculate their dimensions. For this purpose it suffices to consider a special partition: define $\alpha = (\alpha_1,...,\alpha_n)$ by
\begin{eqnarray}
\alpha = \left\{\begin{array}{cl} (2, 4,..., 2n) & \text{ if }\ve = -1; \\ (3,5,...,2n+1) & \text{ if }\ve = 1.
\end{array} \right.
\end{eqnarray}
It follows directly from the definitions that $\C[X_\lambda] \cong \C[X_\alpha] \otimes \C[x_{i,r} \mid 1\le i\le n, \ 1 < r \le s_i]$, which implies that 
\begin{eqnarray}
\label{e:lambdareducestoalpha}
X_\lambda \cong X_\alpha \times \mathbb{A}^{\lfloor \lambda_n/2\rfloor - n}.
\end{eqnarray}
It is now sufficient to study $X_\alpha$ in order to understand the structure of $X_\lambda$ for any $\lambda$.

Let $P_n$ be the set of subsets $S\subseteq \{1,...,n-1\}$ such that for all $i = 1,...,n-2$ either $i \in S$ or $i+1\in S$ or both. For $S \in P_n$ we let $S^c := \{1,...,n-1\} \setminus S$ be the complementary set. We define a map from $P_n$ to the set of subsets of $\{x_{1,1},...,x_{n,1}, y_1,...,y_{n-1}\} \subseteq \C[X_\alpha]$ by 
\begin{eqnarray}
\label{e:subsetstoideals}
\iota : S \mapsto \{y_i \mid i \in S\} \cup \{x_i \mid i \in S^c \text{ or } i-2 \in S^c\}.
\end{eqnarray}
\begin{Lemma}
\label{L:iotaminimalprimes}
$S \mapsto (\iota S)$ is a bijection from $P_n$ to the set of minimal prime ideals of $\C[X_\alpha]$ and
\begin{eqnarray}
\dim \C[X_\alpha] / (\iota S) = |\{x_1,...,x_n, y_1,...,y_{n-1}\} \setminus \iota S|
\end{eqnarray}
\end{Lemma}
\begin{proof}
The torus $T:= (\C^\times)^{2n-1}$ acts on $\C[x_{1,1},...,x_{n,1}, y_1,...,y_{n-1}]$ by automorphisms rescaling generators, and this action descends to $\C[X_\alpha]$. Since $T$ is connected it preserves the irreducible components of $X_\alpha$, hence preserves the minimal primes of $\C[X_\alpha]$. Since the $T$-weight spaces on $\C[x_{1,1},...,x_{n,1}, y_1,...,y_{n-1}]$ are one dimensional, it follows that each minimal prime $\p$ is generated by $\p \cap \{x_{1,1},...,x_{n,1}, y_1,...,y_{n-1}\}$.

By inspection we see that for each $S \in P_n$ the set $\iota S$ contains at least one of the factors of each quadratic relation \eqref{e:kXlambdarelations} in $\C[X_\alpha]$ and is minimal with respect to this property. This implies that $\iota$ maps $P_n$ to minimal primes.

The injectivity of $S \mapsto (\iota S)$ is clear and it remains to show that every minimal prime $\p$ is in the image. Suppose that $\p$ is generated by some subset $S \subseteq \{x_{1,1},...,x_{n,1},y_1,...,y_{n-1}\}$. Using the primality of $\p$ it is straightforward to check that $S' := \{1\le i \le n-1 \mid y_i \in S\} \in P_n$ and $(\iota S') \subseteq \p$. Now by minimality we have $\p = (\iota S')$.
\end{proof}

\begin{Proposition}
\label{P:quadfactors}
There is a bijection $\i \mapsto X_\lambda^\i$ from maximal admissible multisets for $\lambda$ to the set of irreducible components of $X_\lambda$. Furthermore $\dim X_\lambda^\i = |\i|.$
\end{Proposition}
\begin{proof}
Applying Case 1 of the Kempken--Spaltenstein algorithm exactly $(\lambda_i - \lambda_{i-1})/2 - 1$ times at each admissible index $i$ we reduce the partition $\lambda$ to $\alpha$. This shows that the maximal admissible multisets for $\lambda$ are of the form $\i \cup \j$ where $\i$ is a maximal admissible multiset for $\alpha$ and $\j$ is the multiset taking values in $\{1,...,n\}$ in which $j$ occurs with multiplicity $(\lambda_j - \lambda_{j-1})/2 - 1$. Combining this with \eqref{e:lambdareducestoalpha} it suffices to prove the current proposition in the case $\lambda = \alpha$.

By Lemma~\ref{L:iotaminimalprimes} the irreducible components of $X_\alpha$ are parameterised by $P_n$. In order to prove the first claim of the proposition we construct a bijection from $P_n$ to the set of maximal admissible multisets.

The admissible multisets for $\alpha$ have multiplicities $\le 2$. Thanks to \cite[Lemma~6]{PT14} we know that a maximal admissible multiset is totally determined by the collection of indexes in $\{2,...,n\}$ which occur with multiplicity 2: this is the set of indexes such that Case 2 occurs at some point in the KS algorithm. Therefore the collection of indexes of multiplicity 2 is only constrained by the fact that if $i$ has multiplicity 2 then neither $i-1$ nor $i+1$ have multiplicity 2. It follows that if $S \in P_n$ then there is a unique maximal admissible sequence $\i$ for $\alpha$ such that the indexes in $\{2,...,n\} \setminus (S+1)$ occur with multiplicity 2. This gives the desired bijection.

It remains to prove the claim regarding dimensions. Let $S \subseteq P_n$ and enumerate $\{1,...,n-1\} \setminus S = (i_1,...,i_l)$. By the remarks of the previous paragraph the maximal admissible sequence associated to $S \in P_n$ is determined as follows: first we let $\i' = (i_1+1, i_1+1, i_2+1, i_2+1,...,i_l+1, i_l+1)$ and then we let $\i$ be the maximal admissible sequence obtained from $\i'$ by applying Case 1 of the KS algorithm to $\alpha^{\i'}$ as many times as possible. The set of indexes where Case 1 can be applied to $\alpha^{\i'}$ are precisely
\begin{eqnarray}
\{1\leq i \leq n \mid \alpha^{\i'}_i - \alpha_{i-1}^{\i'} = 2\} = \{1\le i \le n \mid i \ne i_t \text{ and } i-2 \ne i_t \text{ for all } t\}
\end{eqnarray}
and it follows that the length of the maximal admissible sequence associated to $S$ is
\begin{eqnarray}
\label{e:lengthformula}
n-1 - |S| +  \hash \{1\le i \le n \mid i, i-2 \in S\}.
\end{eqnarray}
According to \eqref{e:subsetstoideals} and Lemma~\ref{L:iotaminimalprimes} the dimension of the component of $X_\alpha$ corresponding to $S$ is
\begin{eqnarray}
\label{e:dimcompformula}
2n-1 - |S| - \hash \{1\le i \le n \mid i\in S^c \text{ or } i-2 \in S^c\}
\end{eqnarray}
It is straightforward to see that \eqref{e:lengthformula} and \eqref{e:dimcompformula} are equal.
\end{proof}

\subsection{The semiclassical abelianisation I: the distinguished case}

\label{ss:scabelian:distinguishedcase}

Let $\lambda = (\lambda_1,...,\lambda_n) \in \PeN$ be a distinguished partition satisfying $\lambda_1 > 1$ and let $e\in \g$ be a nilpotent element with partition $\lambda$. The indexes $s_1,...,s_{n-1}$ defined in \eqref{e:partdifferences} allow us to construct an $n\times n$ symmetric shift  matrix $\sigma = (s_{i,j})_{1\le i,j\le n}$ via 
\begin{eqnarray}
s_{i+1, i} = s_{i,i+1}:= s_{i+1} \text{ for } i=1,...,n-1.
\end{eqnarray}
By \eqref{e:shiftmatrixdefn} these values determine the entire matrix, which coincides with the matrix constructed from $\lambda$ in \eqref{e:symmetricalshiftmatrix}. Next we describe generators of $\gr_{\m_0} S(\g,e)^\ab$ and certain relations between them.
\begin{Lemma}
\label{L:gensrelstangentconelemma}
Let $\m_0 \subseteq S(\g,e)^\ab$ be unique maximal Kazhdan graded ideal. Then $\gr_{\m_0} S(\g,e)^\ab$ is generated by elements
\begin{eqnarray}
\label{e:distinguishedabelianquotientgenerators}
\{\otheta_{i}^{(s_{i+1}+1)} \mid i = 1,...,n-1\} \cup \{\oeta_i^{(2r)} \mid i=1,...,n, \ r=1,..., s_{i}\}
\end{eqnarray}
These elements satisfy the following relations
\begin{eqnarray}
\label{e:quadinaq1}
\otheta_i^{(s_{i+1}+1)} \otheta_{i+1}^{(s_{i+2}+1)} = 0 & \text{ for } & i=1,...,n-2;\\
\label{e:quadinaq1.5}
\otheta_1^{(s_{2}+1)} \eta_{1}^{(2s_{1})} = 0 & \text{ for } & \ve = -1;\\
\label{e:quadinaq2}
\otheta_i^{(s_{i+1}+1)} (\oeta_i^{(2s_i)} - \oeta_{i-1}^{(2s_i)}) = 0 & \text{ for } & i=2,...,n-1 ;\\
\label{e:quadinaq3}
\otheta_i^{(s_{i+1}+1)} (\oeta_{i+2}^{(2s_{i+2})} - \oeta_{i+1}^{(2s_{i+2})}) = 0 & \text{ for } & i=1,...,n-2
\end{eqnarray}
where the elements $\eta_i^{(r)}$ are taken to be zero if they lie outside the range prescribed by \eqref{e:distinguishedabelianquotientgenerators}.
\end{Lemma}
\begin{proof}
As explained in Section~\ref{ss:symplecticrelations} there is a surjective Poisson homomorphism $\varphi: R(y_n(\sigma), \tau) \onto S(\g,e)$. In the following proof we abuse notation and identify $S(\g,e)$ with the corresponding quotient of $R(y_n(\sigma), \tau)$. This justifies the use of generators and relations appearing in Theorem~\ref{T:PDyangian}.

We let $\theta_{i,j}^{(r)}$ and $\eta_i^{(r)}$ be the generators of $S(\g,e)$ described in \eqref{e:theRgenerators}, and write $\otheta_{i,j}^{(r)} := \theta_{i,j}^{(r)} + \m_0^2$ and $\oeta_i^{(r)} := \eta_i^{(r)} + \m_0^2 \in\gr_{\m_0} S(\g,e)$. 

It follows from the proofs of Propositions~\ref{P:Walgpresentation1} and \ref{P:Walgpresentation-1} that $\gr_{\m_0} S(\g,e)$ is generated as a commutative algebra by elements
\begin{eqnarray}
\label{e:orthtruncatedgenslater}
\{\otheta_{i,j}^{(r)} \mid 1\le i < j \le n, \ s_{i,j} < r \le s_{i,j} + \lambda_i \} \cup \{\oeta_i^{(2r)} \mid 1\le i \le n, \ 0 < r \le \lambda_{i} \}.
\end{eqnarray}
Since natural map $\pi : S(\g,e) \onto S(\g,e)^\ab$ is a Kazhdan graded homomorphism the unique maximal graded ideal of $S(\g,e)$ maps to the unique maximal graded ideal of $S(\g,e)^\ab$, and we indulge in another abuse of notation by writing $\m_0$ for either ideal. Thus $\pi$ induces a homomorphism $\pi_0 : \gr_{\m_0} S(\g,e) \onto \gr_{\m_0} S(\g,e)^\ab$. We shall show that the kernel of $\pi_0$ contains the left-hand sides of \eqref{e:quadinaq1}--\eqref{e:quadinaq3}, as well as the following elements
\begin{eqnarray}
\label{e:gensingradedmap}
\{\otheta_{i,j}^{(r)} \mid j \ne i + 1 \text{ or } r > s_{i+1}+1\} \cup \{\oeta_i^{(2r)} \mid r > s_i\}.
\end{eqnarray}
Since \eqref{e:orthtruncatedgenslater} is the union of \eqref{e:distinguishedabelianquotientgenerators} and \eqref{e:gensingradedmap}, this will complete the proof.

Let $B \subseteq S(\g,e)$ be the defining ideal of $S(\g,e)^\ab$, which is generated by the Poisson brackets in $S(\g,e)$. Then the kernel of $\pi_0$ is equal to $\gr_{\m_0} B$. Corollary~\ref{C:madic} allows us to identify $\gr_{\m_0} S(\g,e)$ with $S(\g^e)$. Under this identification the vector space $\theta(\g^e) \subseteq S(\g,e)$ spanned by elements \eqref{e:orthtruncatedgenslater} identifies with the subspace $\g^e \subseteq S(\g^e)$. Furthermore the same Corollary also implies that $\{\theta(\g^e), \theta(\g^e)\} + \m_0^2$ identifies with $[\g^e, \g^e] \subseteq \g^e\subseteq S(\g^e)$. We can express this more informally by saying that the linear terms of Poisson brackets in $S(\g,e)$ are just the Lie brackets in $\g^e$, which is well-known in general (see \cite[Theorem~2.11]{DSKV16}).

Now \cite[Theorem~6]{PT14} implies that the elements \eqref{e:gensingradedmap} identify with a spanning set for $[\g^e, \g^e]$, although we warn the reader that the notation there is different. It follows that every element of \eqref{e:gensingradedmap} can be expressed as $b + \m_0^2$ for some $b\in B$, and so these elements lie in $\gr_{\m_0} B = \ker\pi_0$ as required.

In order to complete the proof we show that the left-hand sides of \eqref{e:quadinaq1}--\eqref{e:quadinaq3} lie in $\gr_{\m_0} B$. Using \eqref{e:dyrel5} we see that $\theta_i^{(s_{i+1}+1)} \theta_{i+1}^{(s_{i+2}+1)} \in B$ and it follows that the left-hand side of \eqref{e:quadinaq1} lies in $\gr_{\m_0} B$.

Now suppose that $i = 1$ and that $\ve = -1$, so that $\lambda_1$ is even. Thanks to Proposition~\ref{P:Walgpresentation-1} we know that $\ttheta_1^{(\lambda_1+s_{1}+1)} = 0$ in $S(\g,e)$. Using Example~\ref{Ex:firstbadrel} and \eqref{e:dyrel2} we see that 
$$\sum_{t=0}^{\lambda_1/2} \eta_1^{(2t)} \theta_1^{(\lambda_1 - 2t + s_{1,2} + 1)} - \sum_{t=0}^{\lambda_1/2-1} \eta_1^{(2t)}\{\eta_1^{(2)}, \theta_{1}^{(\lambda_1 - 2t + s_{1,2})}\} = \eta_1^{(\lambda_1)} \theta_1^{(s_{2}+1)} \in B.$$
Note that $\lambda_1 = 2s_1$ in this case, since $\lambda_1$ is even.  Projecting $\eta_1^{(\lambda_1)} \theta_1^{(s_{1}+1)}$ into $\gr_{\m_0} S(\g,e)$ we see that the left-hand side of \eqref{e:quadinaq1.5} lies in $\gr_{\m_0} B$.

The relations \eqref{e:quadinaq2}, \eqref{e:quadinaq3} can be deduced similarly by considering the quadratic terms appearing in \eqref{e:dyrel2}, \eqref{e:dyrel8} and projecting into $\gr_{\m_0} S(\g,e)$. Since the arguments for these two cases are almost identical we just provide the details for \eqref{e:quadinaq3}.

Before we proceed we record a single relation which is a special case of \eqref{e:dyrel2}. For $r > s_{i+1} + 1$ we have
\begin{eqnarray}
\label{e:specialbracket}
\{\eta_i^{(2)}, \theta_i^{(r-1)}\} = \theta_i^{(r)}
\end{eqnarray}

Fix $i=1,...,n-2$. Using \eqref{e:dyrel8} we examine the expression $\big\{\theta_{i+1}^{(s_{i+2}+1)}, \{\theta_{i+1}^{(s_{i+2}+1)}, \theta_{i}^{(s_{i+1}+1)}\}\big\}$. There is a unique linear term $\theta_{i}^{(2s_{i+2}+s_{i+1} + 1)}$ corresponding to $m = m_1-1$ and $m_2 = 0$. Thanks to \eqref{e:specialbracket} we can subtract an element of $B$ to eliminate this linear term. Now consider the quadratic terms. If we choose any term such that $m_2 < m_1$ then this term has a factor of $\theta_{i}^{(2(m_1-m_2) + s_{i+1} + 1)}$ and so once again we can eliminate these particular quadratic terms using \eqref{e:specialbracket}. The the only remaining quadratic terms are $\eta_{i+2}^{(2s_{i+1})} \theta_{i}^{(s_{i+1}+1)}$ and $\tilde\eta_{i+1}^{(2s_{i+1})} \theta_{i}^{(s_{i+1}+1)}$ which correspond to $m_1 = m_2 = 0$ and $m_1 = m_2 = m-1$ respectively. Finally one checks using \eqref{e:tetatwisteddefinition} that $\tilde\eta_{i+1}^{(2s_{i+1})} = - \eta_{i+1}^{(2s_{i+1})}$ modulo $\m_0^2$, and so $$\eta_{i+2}^{(2s_{i+1})} \theta_{i}^{(s_{i+1}+1)} + \tilde\eta_{i+1}^{(2s_{i+1})} \theta_{i}^{(s_{i+1}+1)} + \m_0^3 = \otheta_i^{(s_{i+1}+1)} (\oeta_{i+2}^{(2s_{i+2})} - \oeta_{i+1}^{(2s_{i+2})}) \in \gr_{\m_0} B.$$
This completes the proof of relation \eqref{e:quadinaq3}.

The proof of \eqref{e:quadinaq2} is almost identical, instead examining $\big\{\theta_{i}^{(s_{i+1}+1)}, \{\theta_{i}^{(s_{i+1}+1)}, \theta_{i+1}^{(s_{i+2}+1)}\}\big\}$, and using \eqref{e:dyrel8} again.
\end{proof}

\begin{Theorem}
\label{T:distinguishedreduced}
Let $\g$ be a classical Lie algebra and $e \in \g$ a distinguished nilpotent element. Let $\m_0\subseteq S(\g,e)^\ab$ be the maximal Kazhdan graded ideal. There is an isomorphism $$\gr_{\m_0} S(\g,e)^\ab \isoto \C[\TC_e(e+X)].$$ In particular $S(\g,e)^\ab$ is reduced.
\end{Theorem}
\begin{proof}
Let $A$ denote the algebra with generators \eqref{e:distinguishedabelianquotientgenerators} and relations
\begin{eqnarray}
\label{e:yetanotherlabel}
\begin{array}{rcl}
\otheta_i^{(s_{i+1}+1)} \otheta_{i+1}^{(s_{i+2}+1)} = 0 & \text{ for } & i=1,...,n-2;\\
\otheta_1^{(s_2+1)} \oeta_{1}^{(2s_{1})} = 0 & \text{ for } & \ve = -1;\\
\otheta_i^{(s_{i+1}+1)} \oeta_i^{(2s_i)}  = 0 & \text{ for } & i=2,...,n-1 ;\\
\otheta_i^{(s_{i+1}+1)} \oeta_{i+2}^{(2s_{i+2})} = 0 & \text{ for } & i=1,...,n-3.
\end{array}
\end{eqnarray}

Before we proceed, we outline the argument. First of all we will show that $A\cong \C[X_\lambda]$ where $X_\lambda$ is the combinatorial Katsylo variety from Section~\ref{ss:combKat}. By Lemma~\ref{L:gensrelstangentconelemma} we have that $A \onto \gr_{\m_0} S(\g,e)^\ab$ and since the reduced algebra of $S(\g,e)^\ab$ is $\C[e + X]$ by Proposition~\ref{P:Katsyloproperties}(1), we see that $\C[X_\lambda] \onto \C[\TC_e(e + X)]$. We then combine deductions of the previous sections to see that the components of $X_\lambda$ and $\TC_e(e+X)$ have the same dimensions. Since $X_\lambda$ is reduced we can apply Lemma~\ref{L:closedembeddings} to deduce that we have isomorphisms $\C[X_\lambda] \cong \gr_{\m_0} S(\g,e)^\ab \cong \C[\TC_e(e+X)]$, from which the current theorem follows.\vspace{3pt}

{\it Step (i):} Comparing \eqref{e:yetanotherlabel} with \eqref{e:kXlambdarelations} we see there is an isomorphism $A \isoto \C[X_\lambda]$ defined by 
\begin{eqnarray*}
& & \oeta_i^{(2r)} \longmapsto x_{i,s_i+1-r}; \\
& & \otheta_i^{(s_{i+1}+1)} \longmapsto y_i.
\end{eqnarray*}

{\it Step (ii):} Now we show that $A \onto \gr_{\m_0} S(\g,e)^\ab$. We adopt the convention $\oeta_i^{2r} = 0$ when $r > s_i$ or $i = 0$. The algebra endomorphism of $A$ defined by $\oeta_i^{(2s_{i})} \mapsto \oeta_i^{(2s_{i})} + \oeta_{i-1}^{(2s_i)}$ for all $i=1,...,n$ and acting identically on the other generators, is a unimodular subsitution. Therefore it induces an automorphism. By slight abuse of notation we denote the images of the generators by the same symbols. With this new set of generators the relations are precisely \eqref{e:quadinaq1}-\eqref{e:quadinaq3}. By Lemma~\ref{L:gensrelstangentconelemma} and part (i) we have a surjection $A = \C[X_\lambda] \onto \gr_{\m_0} S(\g,e)^\ab$.

{\it Step (iii):} Recall that we equip$\TC_e(e+X)$ with the reduced scheme structure. By Corollary~\ref{C:madic} we have 
\begin{eqnarray}
\label{e:somesurjs}
\C[X_\lambda] \onto \gr_{\m_0} S(\g,e)^\ab \onto \C[\TC_e(e+X)]
\end{eqnarray}
 Applying Proposition~\ref{P:Katsyloproperties}(2), Proposition~\ref{P:KSclassifiessheets} and Proposition~\ref{P:quadfactors} we have
$$d(X_\lambda) = d(e + X) = d(\TC_e(e+X)).$$
Now by Lemma~\ref{L:closedembeddings} we see that the surjections \eqref{e:somesurjs} are isomorphisms on the underlying varieties of closed points. Since $\C[X_\lambda]$ is reduced these maps are algebra isomorphisms and $\gr_{\m_0} S(\g,e)^\ab$ is reduced. Applying Lemma~\ref{L:reducedlemma} we see that $S(\g,e)^\ab$ is reduced, as claimed.
\end{proof}

\subsection{The semiclassical abelianisation II: the general case}

The following result uses the local geometry of Poisson manifolds to infer reducedness of $S(\g,e)^\ab$ from the case where $e$ is distinguished, where we can apply Theorem~\ref{T:distinguishedreduced}.

\begin{Theorem}
\label{T:generalreduced}
Let $\g$ be a classical simple Lie algebra and let $e \in \g$ be any nilpotent element. Then $S(\g,e)^\ab$ is reduced.
\end{Theorem}
\begin{proof}
{\it Part (i):} In type {\sf A} the fact that $S(\g,e)^\ab$ is reduced can be deduced from \cite[Theorem~3.3]{Pr10}, and so we let $\g$ be simple of type {\sf B}, {\sf C} or {\sf D} throughout the proof. We let $\O \subseteq \g$ be the adjoint orbit of $e$ in $\g$, and let $\tilde \g$ and $\widetilde\O$ be the Lie algebra and distinguished nilpotent orbit introduced in Lemma~\ref{L:distinguishedinduction}. Pick $\tilde e \in \tilde\g$ and an $\sl_2$-triple $(\tilde e, \tilde h, \tilde f)$. Write $\widetilde{G}$ for the simply connected, connected complex algebraic group with $\tilde{\g} = \Lie(\widetilde{G})$. Recall that $\l =\gl_{n-i_1+1} \times \cdots \times \gl_{n-i_l+1} \times \g$ embeds as a Levi subalgebra of $\tilde{\g}$. If we choose $e\in \O \subseteq \g\subseteq \l$ then $\O$ identifies with the adjoint $L$-orbit and we have the following Poisson isomorphism
$$S(\l, e) \cong S(\gl_{i_1}) \otimes \cdots \otimes S(\gl_{i_l}) \otimes S(\g, e)$$
Therefore $S(\g,e)^\ab$ is reduced if and only if $S(\l,e)^\ab$ is reduced. This holds if and only only if $(S(\l,e)_{\m_0}^\c)^\ab$ is reduced, where $\m_0$ is the maximal ideal of $e$, thanks to Lemma~\ref{L:reducedlemma} and Lemma~\ref{L:abelianisationvscompletion}.\vspace{3pt}

{\it Part (ii):} Now pick a regular element $s\in \z(\l)$ so that $\tg^s \cong \l$. Write $x = s + e \in \tilde{\g}$. We will explain below that $x$ can be viewed as a point of $\tilde e + \tilde\g^{\tilde f} = \Spec S(\tilde\g, \tilde e)$, and by slight abuse of notation we will also write $\m_0$ for the maximal ideal of $x$ in $S(\tilde \g, \tilde e)$. We claim that $S(\l, e)_{\m_0}^\c \cong S(\tilde\g,\tilde e)_{\m_0}^\c$ as Poisson algebras. Note that $(S(\tilde{\g}, \tilde e)_{\m_0}^\c)^\ab$ is reduced by Lemma~\ref{L:reducedlemma} and Theorem~\ref{T:distinguishedreduced}. Thanks to part (i) the current proof will follow from the claim.\vspace{3pt}

{\it Part (iii):} Pick an $\sl_2$-triple $(e,h,f)$ for $e$ inside $\l$, then define $y = s + f \in \tilde{\g}$. By $\sl_2$-theory the affine variety $x + \tilde{\g}^y$ is a transverse slice to the adjoint orbit $\Ad(\widetilde G) x$ at the point $x$. Also pick an $\sl_2$-triple $(\tilde{e}, \tilde{h}, \tilde{f})$ in $\tilde{\g}$. After replacing the latter triple by some conjugate we can actually assume that $x\in \tilde{e} + \tilde{\g}^{\tilde{f}}$. By construction we have $\Ind_\l^{\tilde{\g}}(\O) = \Ad(\widetilde{G}) \tilde{e}$. It follows that $x$ lies in a sheet of $\tilde{\g}$ containing $\tilde{e}$, and so we have $x \in \tilde{e} + \widetilde{X}$, where the latter variety is defined in the same way as \eqref{e:Katsylovariety}.

Now we choose small neighbourhoods $U_x \subseteq \tilde e + \tilde\g^{\tilde f}$ and $V_x \subseteq x + \tilde{\g}^y$ of $x$ in the complex topologies. Similarly let $W_e \subseteq e+ \l^f$ be a small neighbourhood of $e$. It is well-known that $U_x, V_x, W_e$ all carry transverse Poisson structures on the ring of analytic functions, see \cite[\textsection 5.3.3]{LPV13} for example. Write $\C^\an(M)$ for the ring of analytic functions on a complex manifold $M$. There is a natural restriction map $\C[e + \l^f] \to \C^\ab(W_e)$ which is a Poisson homomorphism. Thanks to \cite[Proposition~3]{Ser55} this induces an isomorphism $S(\l, e)_{\m_0}^\c \cong \C^\an(W_e)_e^\c$ of complete Poisson algebras. Similarly we have a Poisson isomorphism $S(\tilde{\g}, \tilde{e})_{\m_0}^\c \cong \C^\an(U_x)_x^\c$.

Now the claim will follow from the existence of an isomorphism $\C^\an(U_x)_x^\c \cong \C^\an(W_e)_e^\c$ of complete Poisson algebras. First of all we observe that $x\in \tilde{e} + \widetilde{X}$ implies that both $\tilde{e} + \tilde{\g}^{\tilde f}$ and $x + \tilde{\g}^y$ are transverse slices to $\Ad(\widetilde{G})x$ at $x$. Therefore by \cite[Proposition~5.29]{LPV13} we have an isomorphism $U_x \to V_x$ of analytic Poisson manifolds sending $x$ to $x$. Furthermore by \cite[Proposition~2.1]{DSV07} we see that there is a similar isomorphism $V_x \to W_e$ sending $x$ to $e$. This completes the proof.
\end{proof}

\subsection{The abelianisation of the finite $W$-algebra via deformation theory}
\label{ss:abelianisationviadeformation}

Here we prove the first main theorem (Theorem~\ref{T:main}), which states that Premet's component map \eqref{e:Premetsmap} is a bijection. As we explained after the statement of that theorem, the result is know in type {\sf A} and so we must complete the proof for the other classical types.

First of all we apply Proposition~\ref{P:relatingabelianisations} and Theorem~\ref{T:generalreduced} to see that both $\gr U(\g,e)^\ab$ and $U(\g,e)^\ab$ are reduced. Now apply Lemma~\ref{L:reduceddichotomy}, Proposition~\ref{P:Katsyloproperties}(1) and Theorem~\ref{T:semiclassicalcomponents} to see that
$$\hash \Comp \E(\g,e) \le \hash \Comp \C\E(\g,e) = \hash \Comp(e + X).$$
Since Premet's map \eqref{e:Premetsmap} is surjective and restricts to a dimension preserving bijection from some subset, the theorem follows.

\begin{Remark}
Premet asked whether the abelian quotient $U(\g,e)^\ab$ of a finite $W$-algebra is reduced \cite[Question 3.1]{Pr10}. Combining Proposition~\ref{P:relatingabelianisations} with Theorem~\ref{T:generalreduced} we have given an affirmative answer in the case of classical Lie algebras.

The problem of understanding whether $U(\g,e)^\ab$ is reduced for exceptional Lie algebra is rather subtle. The methods of this paper will certainly not work in general: in the introduction to \cite{Pr14} it is explained that there are four orbits in exceptional Lie algebra such that the associated finite $W$-algebra is known to admit precisely two one dimensional representations (these correspond to the first four columns of \cite[Table~1]{PT21}). For these $W$-algebras it is not hard to see that the reduced algebra associated to $U(\g,e)^\ab$ is isomorphic to $\C[x]/(x^2-1)$ as filtered algebras, with the generator $x$ in some positive degree. This ensures that $\gr U(\g,e)^\ab$ is not reduced, and Proposition~\ref{P:relatingabelianisations} implies that $S(\g,e)^\ab$ admits nilpotent elements.
\end{Remark}

\section{The orbit method}

\subsection{Quantizations and deformations of symplectic singularities} 
\label{ss:quantsanddefs}

We say that an affine Poisson algebra $A$ is {\it graded of degree $d > 0$} if $A = \bigoplus_{i \ge 0} A_i$ is a graded commutative algebra such that $A_0 = \C$ and the Poisson bracket is in degree $-d$, meaning $\{A_i, A_j\} \subseteq A_{i+j-d}$. In this case the variety $X := \Spec(A)$ is {\it conical} since the grading induces a contacting $\C^\times$-action with unique fixed point $\bigoplus_{i>0} A_i \in X$, and the contracting action rescales the Poisson bivector. The smooth locus of $X$ will be denoted $X^\reg$. From now on we fix such a Poisson algebra $A$.

A {\it filtered Poisson deformation} of $A$ is a pair $(\A, \iota)$ consisting of a filtered Poisson algebra $\A = \bigcup_{i\ge 0} \A_i$ such that the Poisson bracket satisfies $\{\A_i, \A_j\} \subseteq \A_{i+j-d}$ and an isomorphism of graded Poisson algebras $\iota : \gr \A \to A$. 

A {\it filtered quantization} of $A$ is pair $(\A, \iota)$ where $\A$ is a not necessarily commutative, filtered algebra satisfying $[ \A_i, \A_j] \subseteq \A_{i+j - d}$ such that $\iota : \gr \A \to A$ is a Poisson isomorphism (see \cite[2.2]{ACET20} for example).

Let $B$ be a non-negatively graded connected algebra. By a {\it graded Poisson $B$-algebra} we shall mean a graded $B$-algebra $\A$ with Poisson structure such that the structure map $B\to \A$ is graded and factors through the Poisson centre. A {\it graded Poisson deformation of $A$ over base $B$} is a pair $(\A, \iota)$ consisting of a flat graded Poisson $B$-algebra $\A$ and an isomorphism $\iota : \A \otimes_B \C \to A$.

Now let $B$ be a connected graded algebra concentrated in non-negative degrees. We can then equip $B$ with the obvious filtration whose $j$th filtered piece is the sum of the first $j$ graded pieces, and $\gr B \cong B$ as algebras. A {\it filtered $B$-algebra} is a filtered algebra $\A$ equipped with a filtered map $B \to \A$ which factors through the centre. We say that $(\A, \iota)$ is a {\it filtered quantization (of a Poisson deformation) of $A$ over base $B$} if $\A$ is a flat filtered $B$-algebra and $(\gr\A, \iota)$ is a graded Poisson deformation of $A$ over $\gr B = B$, and the map $\iota$ is an isomorphism $\iota : \gr\A \otimes_{\gr B} \C \to A$ of Poisson algebras.

For more detail on the above definitions, and a precise description of morphisms, we refer the reader to \cite[\textsection 2.2, 2.3, 3.1]{Lo22} or \cite[\textsection 2]{ACET20}.

\begin{Remark}
\label{R:filtvsgraded}
Filtered Poisson deformations of $A$ correspond bijectively to graded Poisson deformations of $A$ over base $\C[h]$ with $h$ in degree 1, by the Rees algebra construction \cite[\textsection 2.3]{Lo22}.
\end{Remark}

Whilst surveying the construction of graded Poisson deformations of Poisson algebras, and their quantizations, below we shall occasionally refer to deformations and quantizations of a Poisson scheme $X$ (the distinction between a deformation of an algebra and a scheme is important when $X$ is not affine). We refer the reader to \cite[\textsection 2.2]{Lo22} for a brief introduction to these notions.

Following Beauville \cite{Be00}, we say that $X = \Spec(A)$ has {\it symplectic singularities} if:
\begin{itemize}
\item  $X$ is normal;
\item  the Poisson bivector $\omega$ on $X$ has full rank on $X^\reg$, which is therefore a smooth symplectic variety;
\item there exists a projective resolution $\widehat X \to X$ such that the pullback of $\omega|_{X^\reg}$ extends to a regular (not necessarily symplectic) bivector on $\widehat X$.
\end{itemize}

From now on we also assume that $A = \C[X]$ has symplectic singularities and is graded in degree $d > 0$. In this case the variety $X$ is known as a {\it conical symplectic singularity}. 

\begin{Example}
\label{Ex:orbitCSS}
Let $G$ be a complex reductive group and $\g = \Lie(G)$. If $\O\subseteq \g$ is a nilpotent orbit then the affinisation $\Spec \C[\O]$ is a conical symplectic singularity. It is well-known that the affinisation is isomorphic to the normalisation of the closure (see \cite[Proposition~8.3]{Ja04}). The conical structure on $\O$ arises from the vector space structure on $\g$ (see \cite[Lemma~2.10]{Ja04}), whilst the presence of symplectic singularities follows from the work of Panyushev \cite{Pa91}.

 Now suppose that $\widetilde \O \to \O$ is a $G$-equivariant finite cover of a nilpotent orbit. Then the affinisation $\Spec \C[\widetilde\O]$ is also a conical symplectic singularity (see \cite[Lemma~2.5]{Lo18}). We remark that the conical structure comes from lifting the $\C^\times$-action, and possibly replacing $\C^\times$ with a cover.
\end{Example}

A conical symplectic singularity $X = \Spec(A)$ admits a very nice deformation theory: there is a classification of filtered Poisson deformations and filtered quantizations of $A$, due to Namikawa \cite{Na10, Na11} and Losev \cite{Lo22}, which we now recall.

Recall that a normal variety $\wt X$ is said to be $\Q$-factorial if every Weil divisor admits a non-zero integer multiple which is Cartier. We refer the reader to \cite[Proposition~2.3]{Lo22} and \cite{BCHM} for the defining properties of a $\Q$-factorial terminalisation.

Let $\wt X \to X$ be a $\Q$-factorial terminalisation and define $\Pbb := H^2(\wt X^\reg, \C)$. This is known as the {\it Namikawa--Cartan space} of $X$. Then according to \cite[Proposition~2.6]{Lo22} there exists a universal graded Poisson deformation of $\wt X$ over the base $\Pbb$, which should be understood as scheme-theoretic deformation, in the sense of \cite[\textsection 2.2]{Lo22}. This means the following: if $\wt X_B$ is a graded Poisson deformation of $\wt X$ over $\Spec(B)$ there exists a unique $\C^\times$-equivariant morphism $\Spec(B)\to \Pbb$ and a (non-unique) isomorphism $\Spec(B) \times_\Pbb \wt X_{\Pbb} \isoto \wt X_B$. In other words every graded Poisson deformation of $\wt X$ can be obtained from $\wt X_\Pbb$ by base change.

A universal graded Poisson deformation of $A$ can be produced from $\wt X_\Pbb$. First of all we describe an important finite group $W_X$ associated to $X$ which controls isomorphisms of deformations. The following facts are explained in more detail in \cite[\textsection 2.3]{Lo22}. A result of Kaledin \cite{Ka06} states that $X$ has only finitely many symplectic leaves. Let $\L_1,...,\L_k$ be the collection of leaves of codimension 2. Another of the main results of {\it op. cit.} states that there exists a formal slice $\Sigma_i$ to each leaf $\L_i$. By \cite[Proposition~1.3]{Be00} we know that $X$ has rational singularities, and so the formal neigbourhood of $0 \in \Sigma_i$ is a Du Val singularity, i.e. the formal completion of a rational double point of type ADE. Let $\wt W_i$ be the Weyl group associated to the Dynkin diagram of the singularity $\Sigma_i$. The fundamental group of $\L_i$ acts on $\wt W_i$ by diagram automorphisms and we let $W_i$ denote the fixed points. The {\it Namikawa--Weyl group} of $X$ is $W_X := \prod_{i=1}^k W_i$.

We are now ready to describe a Poisson deformation of $X$ with a remarkable property. Let $X_\Pbb := \Spec \C[\wt X_\Pbb]$ and $W = W_X$.
\begin{Proposition}
\label{P:universalPoissondef}
\cite[Proposition~2.9, Proposition~2.12, Corollary~2.13]{Lo22}\\
The algebras $A_\Pbb := \C[X_\Pbb]^W$ and $B := \C[\Pbb]^W$ satisfy the following properties:
\begin{enumerate}
\item $A_\Pbb$ is a finitely generated algebra and a free $B$-module.
\item $A_\Pbb \otimes_B \C = A$ where $\C$ denotes the unique one dimensional graded $B$-module. 
\item For every finitely generated positively graded, connected algebra $B'$ and graded Poisson $B'$ algebra $A'$ such that $A' \otimes_{B'} \C \cong \C[X]$ there exists a unique graded algebra homomorphism $B \to B'$ and a (not necessarily unique) $B'$-linear graded Poisson isomorphism
\begin{eqnarray}
\label{e:Poissdefuniversalprop}
A \otimes_{B} B' \isoto A'
\end{eqnarray}
which intertwines the isomorphisms $A \otimes_B \C \cong \C[X] \cong A' \otimes_{B'} \C$. 
\end{enumerate}
We call $A_\Pbb$ {\it the universal graded Poisson deformation of $A$}.
\end{Proposition}

Similar to the situation for Poisson deformations, the smooth locus of the $\Q$-factorial terminalisation $\widetilde X$ of $X$ admits a ``universal quantization'' (Cf. \cite[Proposition~3.1]{Lo22}). This is a sheaf of filtered, flat $\C[\Pbb]$-algebras in the conical topology on $\widetilde X$. The pushforward of this sheaf of algebras via $\widetilde X^\reg \to \widetilde X$ is denoted $\D_\Pbb$ \cite[Corollary~3.2]{Lo22}.

It follows that there is a map from $\C[\Pbb]$ to the global sections of $\D_\Pbb$. The algebra $\C[\Pbb]$ is graded whilst $\Gamma(\D_\Pbb)$ is filtered. With respect to the natural filtration induced on $\C[\Pbb]$ the map $\C[\Pbb] \to \Gamma(\D_\Pbb)$ is strictly filtered in the sense of \cite[\textsection 2.2]{ACET20}.

\begin{Lemma}
\label{L:DPbbLemma}
\cite[Proposition~3.3]{Lo22}\\
Write $W = W_X$. We have the following: 
\begin{enumerate}
\item $\Gamma(\D_\Pbb)$ is a filtered quantization of $\C[X_\Pbb]$ and the associated graded homomorphism of $\C[\Pbb] \to \Gamma(\D_\Pbb)$ is $\C[\Pbb] \to \C[X_\Pbb]$, appearing in Proposition~\ref{P:universalPoissondef}. 
\item $W$ acts on $\Gamma(\D_\Pbb)$ by filtered automorphisms preserving $\C[\Pbb]$. The induced action on $\gr (\C[X_\Pbb]) = \C[X_\Pbb]$ coincides with the action on $\C[X_\Pbb]$.
\end{enumerate}
\end{Lemma}

Crucially the filtered quantizations of Poisson deformations of $A$ are classified by the same space as the graded Poisson deformations of $A$.
\begin{Proposition}
\label{P:universalQuant}
\cite[Proposition~3.5]{Lo22}
Let $\A_\Pbb := \Gamma(\D_\Pbb)^W$ and $B:= \C[\Pbb]^W$. Let $B'$ be any finitely generated, positively graded commutative algebra and let $A'$ be a graded, flat Poisson $B'$-algebra such that $A' \otimes_{B'} \C = A$. Also let $\A'$ be a $B'$-algebra which is filtered quantization of $A'$, such that the associated graded map of $B' \to \A'$ is $B' \to A'$. Then there is a unique filtered algebra homomorphism $B \to B'$ and a (not necessarily unique) isomorphism
\begin{eqnarray}
\label{e:quantuniversalprop}
\A_\Pbb \otimes_{B} B' \isoto A'
\end{eqnarray}
Furthermore the associated graded homomorphism of \eqref{e:quantuniversalprop} is \eqref{e:Poissdefuniversalprop}. 
\end{Proposition}
We call $\A_\Pbb$ the {\it universal filtered quantization of Poisson deformations of $A$}.

For clarity we remark that $\A_\Pbb$ is not a filtered quantization of $A$, rather it is a filtered quantization of $\C[X_\Pbb]$. However filtered quantizations of $A$ can be obtained by specialising $\A_\Pbb$ over points of $\Pbb$ (Cf. Remark~\ref{R:filtvsgraded}).

\subsection{Birational induction and the orbit method}
\label{ss:birationalandinduction}
Let $G$ be a connected complex reductive algebraic group and write $\g = \Lie(G)$. In this section we recap some ideas from \cite[\textsection 4 \& \textsection 5]{Lo22} and explain the construction of the orbit method map.

We identify $\g$ with $\g^*$ in what follows, and so $\C[\g]$ is equipped with a natural Poisson structure. Furthermore the action of $G$ on $\g$ is Hamiltonian.

Recall from Section~\ref{ss:LSinduction} that an {\it induction datum} is a triple $(\l, \O_0, z)$ consisting of a Levi subalgebra, a nilpotent orbit therein and $z\in \z(\l)$. When $z=0$ we usually omit this element, and say that $(\l, \O_0)$ is an induction datum.  After fixing a datum we can choose a parabolic subalgebra $\p = \Lie(P)$ admitting $\l$ as a Levi factor, with $\r =\Rad(\p)$, and construct a homogeneous bundle $G \times^P (z + \overline{\O}_0 + \r)$ over $G/P$. This bundle admits a map to $\g$ via the adjoint representation of $G$, which we denote $\pi$. The image of $\pi$ is the closure of an adjoint orbit, and this orbit depends only on the induction datum, not on the choice of $P$ \cite[Lemma~4.1]{Lo22}. This is the {\it (Lusztig--Spaltenstein) induced orbit}, denoted $\Ind_{\l, z}^\g(\O_0)$.
\begin{itemize}
\item When $\pi$ is generically finite onto its image we say that the orbit datum is {birational}. We say that $\Ind_{\l, z}^\g(\O_0)$ is {\it birationally induced} from $(z, \l, \O_0)$
\item  Furthermore we say that $\O \subseteq \g$ is birationally rigid if it cannot be birationally induced from a proper Levi subalgebra.
\item If $\pi$ is birational and $\O_0$ is birationally rigid then we say that the datum is {\it birationally minimal}.
\end{itemize}

One of the key properties of birational induction is the following.
\begin{Theorem}
\cite[Theorem~4.4]{Lo22}
\label{T:biratind}
For every orbit $\O \subseteq \g$ there is a unique birationally minimal induction datum $(\l, \O_0, z)$ such that $\O = \Ind_{\l, z}^\g(\O_0)$.
\end{Theorem}

For background we refer the reader to \cite{Am20}, where the theory of birational induction was studied in the setting of conjugacy classes in groups.

Now pick an induction datum such that $\O_0 \subseteq \l$ is birationally rigid. Also choose a parabolic subgroup $P$ such that $\p = \Lie(P)$ admits $\l$ as a Levi factor and nilradical $\r$. Let $X_0 := \Spec \C[\O_0]$ be the affinisation of $\O_0$ (this coincides with the normalisation of $\overline\O$, by \cite[Proposition~8.3]{Ja04}) and consider the bundle $G \times^P (X_0 \times \r)$. This contains a dense open orbit, say $G/H$ for some subgroup $H$, and this is a finite $G$-equivariant cover of $\Ind_{\l}^\g(\O_0)$. Thanks to Example~\ref{Ex:orbitCSS} the affinisation $X = \Spec \C[G/H]$ is a conical symplectic singularity. In this setting one can describe the Namikawa--Cartan space, the Namikawa--Weyl group of $X$ and the universal graded Poisson deformation of $X$.

\begin{Lemma}
\label{L:Namikawainvariants}
\cite[Proposition~2.9 \& 4.7]{Lo22}
If $G$ is semisimple then:
\begin{enumerate}
\item $\Pbb = \z(\l)$;
\item $W_X$ is a normal subgroup of $W(\l, \O_0) := Z_N(\O_0)/L$ where $N = N_G(L)$ is the normaliser of $L$, and $Z_N(\O_0)$ denotes the set of elements of $N$ which stabilise the orbit $\O_0$.
\item The scheme $\widetilde X_\Pbb$ is isomorphic to 
$\check X := G \times^P (\z(\l) \times X_0 \times \r)$ as a Poisson scheme over $\Pbb$ with $\C^\times$-action.
\end{enumerate}
\end{Lemma}

We note that the $\C^\times$-action on $\check X$ arises from the fact that $\check X$ is a graded deformation of $G \times^P (X_0 \times \r)$, in the sense of \cite[\textsection 2.2]{Lo22} (see the proof of \cite[Proposition~4.7]{Lo22}).

Now pick an orbit $\O \subseteq \g$, and denote the corresponding birational induction datum by $(\l, \O_0, z)$. Let $G/H$ be the dense orbit in $G \times^P (X_0 \times \r)$, as above. It follows from \cite[Proposition~4.7(2)]{Lo22} that the affinisation of the dense orbit in $G\times^P (\{z\} \times X_0 \times \r)$ is isomorphic to the fibre of the universal graded Poisson deformation of $\C[G/H]$ corresponding to the parameter $z \in \Pbb$ (terminology of Proposition~\ref{P:universalPoissondef}). We denote the regular functions on the fibre by $\C[X_z]$.

Thanks to Proposition~\ref{P:universalQuant} the above construction has a quantum analogue: let $\A_z$ be the fibre of the universal filtered quantization of $\C[G/H]$ corresponding to parameter $z \in \Pbb$.

The action of $G$ on $G/H$ is Hamiltonian and the comoment map $\g \to \C[X_0]$ lifts to a quantum comoment map $\g \to \A_z$. This gives rise to an algebra homomorphism $U(\g) \to \A_z$ and the kernel is denoted $\J(\O)$.

\begin{Lemma}
\label{L:CPP}
The ideal $\J(\O)$ is a completely prime, primitive ideal.
\end{Lemma}
\begin{proof}
Since $\A_z$ is a filtered quantization of $\C[G/H]$ and the latter is an integral domain it follows that $\J(\O)$ is completely prime.

Again let $(\l, \O_0, z)$ be the birationally minimal orbit datum inducing to $\O$. The comoment map $\C[\g] \to \C[G/H]$ factors through the comoment map $\C[\g] \to \C[\overline{\Ind_\l^\g(\O_0)}]$, and since the induced orbit is contained in the nilpotent cone, which is the vanishing locus of the non-constant homogeneous invariant polynomials $\C[\g]^G_+$,  it follows that the kernel of the map $\C[\g]^G \to \C[G/H]$ has codimension 1. The associated graded map of quantum comoment $U(\g) \to \A_z$ is $\C[\g] \to \C[G/H]$, and it follows that the kernel of the map $Z(\g) \to \A_z$ has codimension 1. This means that the prime ideal $\J(\O)$ admits a central character. Now by \cite[\textsection 8.5.7]{Di96} it follows that $\J(\O)$ is a primitive ideal.
\end{proof}

Hence this construction yields a map $$\J : \g/ G \to \Prim U(\g),$$ which is injective for classical Lie algebras \cite[Theorem~5.3]{Lo22}.

Recall the notation $\Prim_\O U(\g)$ from the introduction. The following refinement of Theorem~\ref{T:Losevsmap} uses the notation of Section~\ref{s:quantumabelianquotients}. It is the main result of this section.
\begin{Theorem}
\label{T:Losevrefined}
Let $e \in \O \subseteq \g$ be an element of a nilpotent orbit in a classical Lie algebra. The following sets coincide:
\begin{enumerate}
\item The image of $\J$ intersected with $\Prim_\O U(\g)$.
\item The ideals $\Ann_{U(\g)}(Q\otimes_{U(\g,e)} \C_\eta)$ where $\C_\eta\in U(\g,e)\lmod$ is one dimensional.
\end{enumerate}
\end{Theorem}

The proof of the theorem will occupy the rest of the paper. 

The first step is to describe the associated variety of the primitive ideal $\J(\O)$. This was stated in \cite[\textsection 5.4]{Lo22}, but we provide a proof for completeness. Another proof has appeared recently in \cite[Proposition~6.1.2(1)]{LMM21}.
\begin{Lemma}
\label{L:associatedvariety}
For $\O \in \g/G$ we let $\Sc$ be a sheet of $\g$ containing $\O$. Then
$$\VA(\J(\O)) = \overline{\Sc \cap \N(\g)}.$$
\end{Lemma}
\begin{proof}
Let $(\l, \O_0, z)$ be the birationally minimal orbit datum inducing to $\O$. For $t\in \C$  we can consider the orbit $\O(t)$ which is dense in the image of the generalised Springer map $G \times^P (tz + \O_0 + \n) \to \g$. By definition we have $\O(1) = \O$ and $\O(0) = \Ind_\l^\g(\O_0)$, whilst $\dim \O(t)$ is constant. It follows that $\O$ and $\Ind_\l^\g(\O_0)$ lie in a sheet. Since $\O_0$ is nilpotent and every sheet contains a unique nilpotent orbit, $\Sc \cap \N(\g) = \Ind_\l^\g(\O_0)$ for every sheet containing $\O$.

Now write $G/H$ for the dense orbit in $G\times^P (X_0 \times \n)$. Since $\gr$ is exact on strictly filtered vector spaces, the graded ideal $\gr \J(\O)$ is the kernel of the comoment map $\C[\g] \to \C[G/H]$. This factors through the comoment map for $\C[\g] \to \C[\overline{\Ind_\l^\g(\O_0)}]$, which shows that $\sqrt{\gr \J(\O)}$ is the defining ideal of $\overline{\Ind_\l^{\g}(\O_0)}$. This completes the proof.
\end{proof}

\subsection{Harish-Chandra bimodules and Losev's dagger functors}
\label{ss:quantumHam}

Now  we keep fixed $e\in \N(\g)$, the adjoint orbit $\O := \Ad(G)e$, an $\sl_2$-triple $\{e,h,f\}$ for $e$ and write $G^e(0)$ for the pointwise stabiliser of the triple. Note that the notation $\g^e(0) = \Lie G^e(0)$ is consistent with that of Section~\ref{ss:classicalfiniteWalg}.

In \cite{Lo10a} Losev constructed a map $\mathcal{I} \mapsto \mathcal{I}^\dagger$ from two-sided ideals of $U(\g,e)$ to two-sided ideals of $U(\g)$, with remarkable properties. In his subsequent work \cite{Lo11} this construction was upgraded to a pair of adjoint functors between the categories of Harish--Chadra $U(\g)$-bimodules and $G^e(0)$-equivariant Harish--Chandra $U(\g,e)$-bimodules, as we now recall.

The category $\HC U(\g)$ of Harish-Chandra $U(\g)$-bimodules consists of $U(\g)$-bimodules, finitely generated on both sides, which are locally finite for $\ad(\g)$. This implies that there is a $G$-action which differentiates to $\ad(\g)$.
If $\O := G\cdot e$ then we denote by $\HC_{\overline{\O}} U(\g)$ the full subcategory of bimodules supported on $\overline{\O}$, see \cite[\textsection 1.3]{Lo11} for the precise definition.

It follows from \cite{GG02} that $U(\g,e)$ admits $G^e(0)$-action by filtered automorphisms, and a quantum comoment map
\begin{eqnarray}
\label{e:centmaps}
\g^e(0) \hookrightarrow U(\g,e)
\end{eqnarray}
This map is injective, thanks to \cite[Lemma~2.4]{PrJI}. Therefore one can define the category $\HC^{G^e(0)}U(\g,e)$ of $G^e(0)$-equivariant Harish-Chandra $U(\g,e)$-bimodules to be the category of $U(\g,e)$-bimodules, finitely generated on both sides, with a compatible $G^e(0)$-action which differentiates to the $\ad \g^e(0)$-action (see \cite[\textsection 1.3]{Lo11}). We denote the subcategory of finite dimensional objects by $\HC^{G^e(0)}_{\text{fin}} U(\g,e)$.

In \cite{Lo11} Losev constructed a pair  $(\bullet_\dagger, \bullet^\dagger)$ of adjoint functors $$\bullet_\dagger : \HC_{\overline \O} U(\g) \leftrightarrows \HC^{G^e(0)}_{\text{fin}} U(\g,e) : \bullet^\dagger$$ which behave nicely when applied to quantizations of nilpotent orbit covers (see \cite[\textsection 5]{Lo22}). The definition of these functors depends upon the decomposition theorem from \cite{Lo10a} which we will not describe here. Instead we recall from \cite[\textsection 3.5]{Lo11} that there is an isomorphism 
\begin{eqnarray}
\label{e:Whittakerdagger}
M_\dagger \cong (M/M\g(\lemw)_\chi)^{\ad(\g(<0))}
\end{eqnarray}
for any $M \in \HC(\g)$, which is in the spirit of the definitions used throughout this paper. One immediate consequence which we use often is that $U(\g)_\dagger \cong U(\g,e)$.

Let $(\l_0, \O_0)$ be an induction datum for $\O$ such that $\O_0$ is birationally rigid in $\l_0$. This leads to to a finite cover $G/H \onto \O$ discussed in Section~\ref{ss:birationalandinduction}. Let $\A$ denote the universal filtered quantization of $\C[G/H]$ described in Proposition~\ref{P:universalQuant}. We can assume without loss of generality that $(G^e)^\circ \subseteq H \subseteq G^e$. By Lemma~\ref{L:Namikawainvariants} this is a $\C[\z(\l_0)]^{W}$-algebra, where $X$ denotes the affinisation of $G/H$ and $W$ is a subgroup of the relative Weyl group $N_G(L)/L$.

The image of $\C[\z(\l_0)]^W \to \A$ is central and so for every ideal $I \subseteq \C[\z(\l_0)]^W$ we see that $\A/I\A \in \HC U(\g)$. Thus there is a natural $G^e(0)$-action on $(\A/I\A)_\dagger$.

The following fact will be used a few times in the sequel.

\begin{Lemma} \cite[Lemma~5.2(1)]{Lo22}
\label{L:quantdagger}
For any maximal ideal $I \subseteq \C[\z(\l_0)]^W$ we have $(\A / I\A)_\dagger \cong \C[G/H]$ as $G^e(0)$-algebras.
\end{Lemma}

For brevity we write $B = \C[\z(\l_0)]^{W}$ and for a closed point $z\in \Spec B$ we write $I_z \subseteq B$ for the corresponding maximal ideal, and $\C_z = B/I_z$ for the one dimensional $B$-module, so that $\A_z = \A/I_z \A$.

Since $B \subseteq \A$ is central, the image of the map $B \to \A/\A\g(\lemw)_\chi$ lies in the space of $\ad(\g(\lez))$-invariants, giving an algebra homomorphism $B \to \A_\dagger$. The multiplication on $\A_\dagger$ is induced from that on $\A$ it follows that the image of the map is central.
\begin{Lemma}
\label{L:daggerandspecialisation}
$\A_\dagger \otimes_{B} \C_z \cong  (\A\otimes_{B} \C_z)_\dagger$. 
\end{Lemma}
\begin{proof}
By \eqref{e:Whittakerdagger} we have
$\A_\dagger \otimes_{B} \C_z \cong (\A / \A \g(\lemw)_\chi))^{\ad \g(<0)} / (I_z \A / \A \g(\lemw)_\chi))^{\ad \g(<0)}.$
Since $\ad\g(\lez)$ is exact on the category of Whittaker modules \cite[Theorem~6.1]{GG02} it will suffice to show that
$$(\A/\A\g(\lemw)_\chi) / (I_z\A / \A\g(\lemw)_\chi) \cong (\A / I_z \A) / (\A\g(\lemw)_\chi / I_z \A)$$
This is easy to see: both sides are isomorphic to $\A / \A(\g(\lemw)_\chi + I_z)$.
\end{proof}
By the lemma we can (and shall) identify $(\A_z)_\dagger$ with $(\A_\dagger)/I_z \A_\dagger$ in what follows.

\begin{Corollary}
\label{C:Adaggercommutative}
The algebra $\A_\dagger$ is commutative.
\end{Corollary}
\begin{proof}
Pick $a\in \A_\dagger$. Let $\ad(a) : \A_\dagger \to \A_\dagger$ be the map $b \mapsto [a,b]$. For any closed point $z\in \Spec B$ we let $\ad_z(a) : (\A_\dagger)_z \to (\A_\dagger)_z$ be the induced map on the quotient. Since $\A_z$ is a filtered quantization of $\C[G/H]$ it follows from Lemma~\ref{L:quantdagger} that $(\A_z)_\dagger = (\A_\dagger)_z$ is a commutative algebra, and so $\ad_z(a) = 0$.

We conclude $\ad(a) \A_\dagger \subseteq I_z \A_\dagger$ for all maximal ideals $z\in \Spec B$.

By Proposition~\ref{P:universalPoissondef}(1) we see that $\gr \A$ is a free $\gr B$-module, and it follows by a standard filtration argument that $\A$ is a free $B$-module. We deduce that $\bigcap_z (I_z \A) = 0$, and this implies that $\bigcap_z (I_z \A_\dagger) = (\bigcap_z I_z \A)_\dagger = 0$ by \eqref{e:Whittakerdagger} and Lemma~\ref{L:daggerandspecialisation}. Finally we have $\ad(a) \A_\dagger \subseteq \bigcap_z (I_z \A_\dagger) = 0$, and the proof is complete.
\end{proof}

\begin{Proposition}
\label{P:daggerprop}
The $G^e(0)$-action on $\A_\dagger$ factors through the component group $\Gamma := G^e(0) / G^e(0)^\circ$ and
\begin{eqnarray}
\label{e:ogpogg}
\A_\dagger^{\Gamma} := (\A_\dagger)^{G^e(0)} \cong \C[\z(\l_0)]^{W}.
\end{eqnarray}
\end{Proposition}
\begin{proof}
As $B$ is central in $\A$, it follows that $B$ is $\ad(\g)$-invariant, hence $G$-invariant. Therefore the image of $B \to \A_\dagger$ is $G^e(0)$-invariant. Now consider the homomorphism $\varphi : B \to \A_\dagger^{G^e(0)}$. We shall show that this is an isomorphism.

Since $G^e(0)$ is a reductive group the functor of $G^e(0)$-invariants is exact, and so by Lemma~\ref{L:daggerandspecialisation} we have $\A_\dagger^{G^e(0)} \otimes_B \C_z \cong (\A\otimes_B \C_z)_\dagger^{G^e(0)}$. The latter is nonzero by Lemma~\ref{L:quantdagger}. If $0\ne b\in \Ker \varphi$ then, since $B$ is reduced, we can choose a maximal ideal $I_z \subseteq B$ with $b \notin I_z$ and we find $\A_\dagger^{G^e(0)} \otimes_B \C_z = 0$. This contradiction shows that $\varphi$ is injective.

Now consider the cokernel $C$ of $\varphi$ in the category of $B$-modules. We claim that $\varphi$ is surjective after specialisation at any point $z \in \Spec(B)$. By Lemma~\ref{L:quantdagger} we have $(\A/I_z\A)^{G^e(0)} = \C$. Therefore the specialisation is just $\C = B/I_z \to \C$. If this map is zero then we can choose $b\in B$ such that $\C b + I_z = B$ and $\varphi(b) \subseteq I_z \A_\dagger^{G^e(0)}$. This would force $\varphi(B) \subseteq I_z \A_\dagger^{G^e(0)} \subsetneq \A_\dagger^{G^e(0)}$ which would contradict the fact that $\varphi(1) =  1$. Therefore $\varphi$ is surjective after specialisation, as claimed.

Since $(\bullet)\otimes_B \C_z$ preserves cokernels, we see that $C/ I_z C = 0$ for all $z\in \mSpec(B)$. We shall show that this implies $C= 0$.

It follows from Proposition~\ref{P:universalPoissondef}(1), Lemma~\ref{L:DPbbLemma} and Proposition~\ref{P:universalQuant} that $\A$ is a free $B$-module, by a standard filtration argument. Hence $\A$ is locally finite over $B$. It follows from \eqref{e:Whittakerdagger} that $(\A_\dagger)^{G^e(0)}$ is a subquotient of $\A$ as a $B$-module, and so it is locally finite over $B$, hence $C$ is also a locally finite $B$-module.

Now every finite dimensional $B$-submodule decomposes as a direct sum of its generalised eigenspaces for a set of generators for $B$. Since $C$ is the direct limit of its finite dimensional $B$-submodules, the same holds for $C$. To be precise, we let $C_z := \{c \in C \mid I_z^k c = 0 \text{ for some } k\in \Nbb\}.$
Then we have
$$C = \bigoplus_{z\in \Spec(B)} C_z.$$
Now we aim to show that $C_z = 0$ for each $z\in \mSpec(B)$.

Since $C$ is locally finite the same is true for $C_z$. Also $I_z C = C$ implies $I_z C_{z'} = C_{z'}$ for all $z,z' \in \mSpec(B)$. In particular, $I_z C_z = C_z$ for all such $z$.

The finite dimensional indecomposable modules of $C_z$ are each of the form $B/I_z^k$ for some $k \ge 0$, thus every finite dimensional submodule of $C_z$ is a direct sum of modules of this form. It is easy to see that the limit of a directed system of modules of Krull dimension zero has Krull dimension zero. Therefore $C_z$ is necessarily of the form $\bigoplus_{i \in I} B/I_z^{k_i}$ for some index set $I$ and non-negative integers $(k_i)_{i\in I}$. Finally $I_z C_z = C_z$ implies that $k_i = 0$ for all $i$ and that $C_z = 0$ as required. We have now shown that $\varphi$ is surjective.

It remains to show that the $G^e(0)$-action on $\A_\dagger$ factors through $\Gamma$. To this end we demonstrate that the differential of the $G^e(0)$-action is trivial. For $x\in \g^e(0)$ this action coincides with the map $\ad(x) : \A_\dagger \to \A_\dagger$ given by $a \mapsto xa - ax$, where we identify $x$ with its image in $\A_\dagger$ under the composition of \eqref{e:centmaps} with $\mu_\dagger$. By Corollary~\ref{C:Adaggercommutative} this map is zero, and the proof is complete.
\end{proof}

\subsection{The commutative quotient of the $W$-algebra and the orbit method}
\label{ss:commquotandorbitmethod}
Now we relate $\Spec U(\g,e)^\ab$ to the orbit method map. Let $(\l_1,\O_1),...,(\l_s, \O_s)$ be the collection of induction data for $\O := G \cdot e$ such that $\O_i$ is birationally rigid in $\l_i$. Let $P_i$ be a parabolic subgroup of $G$ with Levi factor $L_i$ such that $\Lie(L_i) = \l_i$  and $\r_i := \Rad\Lie P_i$. We let   $G/H_i$ be the dense $G$-orbit in $G \times^{P_i} (\O_i \times \r_i)$. Also write $X_i$ for the normalisation of $\overline{\O}_i$.

The universal quantization of $\C[G/H_i]$ described in Proposition~\ref{P:universalPoissondef} is denoted $\A^i$, and let $W_i$ be the Namikawa--Weyl group of $\Spec \C[G/H_i]$ (Cf. Example~\ref{Ex:orbitCSS}), so that $\A^i$ is a $B_i := \C[\z(\l_i)]^{W_i}$-algebra by Lemma~\ref{L:Namikawainvariants}. Each $\A^i$ comes equipped with a quantum comoment map $\mu_i : U(\g) \to \A^i$ \cite[\textsection 5.2]{Lo22}.

Let $B:= \prod_{i=1}^s B_i$ and consider the $B$-algebra $\A := \prod_{i=1}^s \A_i$. There is a homomorphism $\mu := \prod_{i} \mu_i : U(\g) \to \A$ which gives
\begin{eqnarray}
\label{e:quantumcom}
\mu_\dagger : U(\g,e) \longrightarrow \A_\dagger
\end{eqnarray}

\begin{Lemma}
\label{L:losevreduction}
$\mu_\dagger$ factors through $U(\g,e) \to U(\g,e)^\ab$ and the image lies in $\A_\dagger^{G^e(0)}$. Furthermore the following sets coincide:
\begin{enumerate}
\setlength{\itemsep}{4pt}
\item $\J(\g/ G) \cap \Prim_\O U(\g)$,
\item The set of ideals $\Ann_{U(\g)} Q\otimes_{U(\g,e)} \C_\eta \subseteq \Prim U(\g)$ where $\C_\eta$ is a one dimensional $U(\g,e)$-module appearing in the image of the map of maximal spectra $$\mu_\dagger^* : \Spec \A_\dagger^{G^e(0)} = \Spec B \to \Spec U(\g,e)^\ab.$$
\end{enumerate}
\end{Lemma}
\begin{proof}
The first claim follows directly from Corollary~\ref{C:Adaggercommutative}.

We now prove that (1) $=$ (2). The set (1) is equal to the collection of kernels of the maps $U(\g) \to \A\otimes_B \C_z$ as we vary over all closed points $z\in \Spec B$. For any fixed $z$ (using Lemma~\ref{L:daggerandspecialisation}) the kernel of the map $\mu_\dagger : U(\g,e) \to \A_\dagger \otimes_B \C_z$ is the annihilator of a one dimensional $U(\g,e)$-module, thanks to our previous observations. We denote it $\C_\eta$. Our chosen $z$ lies in $\Spec B_i$ for some $i$ and we claim that the kernel $K = \Ker(U(\g) \to \A \otimes_{B} \C_z)$ is equal to the annihilator of $Q\otimes_{U(\g,e)} \C_\eta$. This will complete the proof.

As we mentioned above, the functor $\bullet^\dagger$ on Harish--Chandra bimodules was preceded by Losev's construction of a map $\mathcal{I} \mapsto \mathcal{I}^\dagger$ from ideals of $U(\g,e)$ to ideals of $U(\g)$ \cite{Lo10a}. Also note that two sided ideals of $U(\g)$ can be regarded as Harish--Chadra $U(\g)$-bimodules, and that the map $\mathcal{J} \mapsto \mathcal{J}_\dagger$ sends ideals of $U(\g)$ to ideals of $U(\g,e)$ \cite[Theorem~1.3.1]{Lo11}.

By {\it loc. cit.} we know that $K_\dagger = \Ann_{U(\g,e)} \C_\eta$. By Lemmas~\ref{L:CPP} and \ref{L:associatedvariety}, we know that $K$ is a primitive ideal with associated variety $G\cdot e$, so by \cite[Theorem~1.2.2(viii)]{Lo10a} the ideals $I \subseteq U(\g,e)$ which satisfy $I^\dagger = K$ are precisely minimal primes over $K_\dagger$. Since $K_\dagger$ is primitive it is prime, and in particular $(K_\dagger)^\dagger = K$. Now \cite[Theorem~1.2.2(ii)]{Lo10a} implies that $$K = (\Ann_{U(\g,e)} \C_\eta)^\dagger = \Ann_{U(\g)} (Q \otimes_{U(\g,e)} \C_\eta).$$
This concludes the proof.
\end{proof}

Finally we prove the main theorem of this section.

\begin{proofoforbitmethod}
Thanks to Lemma~\ref{L:losevreduction} it suffices to show that $\mu_\dagger^* : \Spec \A_\dagger^{G^e(0)} \to \Spec U(\g,e)^\ab$ is surjective. We will show that the map is finite and dominant. Since finite morphisms are closed this will complete the proof. Since $\gr U(\g,e)^\ab$ is reduced (Corollary~\ref{P:relatingabelianisations} and Theorem~\ref{T:generalreduced}), and since the property of being a finite module or an injective homomorphism can both be lifted through the associated graded construction, it is enough to show that the map $\gr \mu_\dagger : \C[e+\g^f]^\ab \to \gr \A_\dagger^{G^e(0)}$ is injective and gives a finite extension of algebras. The $G^e(0)$-action on $\A_\dagger$ factors through a finite group (Proposition~\ref{P:daggerprop}) and so it will suffice to show that $\mu_\dagger : S(\g,e)^\ab \to \gr \A_\dagger$ is injective and finite. 

By Lemma~\ref{L:Namikawainvariants} we know that $\gr \A$ is isomorphic to $\C[\check X]^W$ where $\check X = \prod_{i=1}^s G\times^{P_i} (\z(\l_i) \times X_i \times \r_i)$ and $W = \prod_{i=1}^s W_i$. It will be convenient to work with a slightly larger algebra: let $\check \A$ denote the algebra $\prod_{i=1}^s \Gamma(\D^i)$ where $\D^i$ is the sheaf associated to the conic symplectic singularity $\Spec\C[G/H_i]$ in Proposition~\ref{P:universalQuant}. Then we have $\gr \check \A =\C[\check X]$ and it will suffice to show that $S(\g,e)^\ab\to \gr (\check \A_\dagger)$ is injective and finite, since $\check \A^W = \A$ and $W$ is a finite group. 

The map $S(\g,e) \to \gr (\check\A_\dagger)$ coincides with the pullback to $e+ \g^f$ of $\C[\g] \to \C[\check X]$ (see \cite[Lemma~5.1(4)]{Lo22} or \cite[Lemma~3.3.2(3)]{Lo11}). This means that if we take the preimage $Y$ of $e+\g^f$ inside $\check X$ then $\gr (\check \A_\dagger) = \C[Y]$. The algebra $S(\g,e)^\ab$ is reduced (Theorem~\ref{T:generalreduced}) and it follows that the map is injective if and only if the corresponding morphism of spectra is dominant. Now it suffices to show that $Y \to \Spec S(\g,e)^\ab$ is finite and dominant.

The morphism $\check X \to \g$ is proper, since it is the collapsing map of a bundle over a projective variety $\prod_i G/P_i$. Properness is stable under base change and it follows that $Y \to e+\g^f$ is proper as well.

We now show that $Y \to e + \g^f$ is quasi-finite. The map $G \times^{P_i} (\z(\l_i) \times X_i \times \r_i) \to \g$ factors through $\xi : G \times^{P_i} (\z(\l_i) + \overline{\O_i} + \r_i) \to \g$ and, since $G \times^{P_i} (\z(\l_i) \times X_i \times \r_i)$ is the normalisation of $G \times^{P_i} (\z(\l_i) + \overline{\O_i} + \r_i)$, it suffices to show that $\xi$ map is quasi-finite, when restricted to the preimage of $e+\g^f$. The image of $\xi$ is the closure of the decomposition class $\overline{\D}_i$ associated to $(\l_i, \O_i)$ (see \cite[Lemma~2.5]{Bo81}). 
The variety $Y$ inherits a $\C^\times$-action from the contracting Kazhdan action on $e+\g^f$ described in Section~\ref{ss:classicalandslices}. The map $Y \to e+\g^f$ is $\C^\times$-equivariant, and so to show that it is quasi-finite it suffices to show that the fibre over $e$ is finite. By Lemma~\ref{L:Namikawainvariants}(3) the fibre over $e$ is contained in $G \times^{P_i} (\overline{\O_i} + \r_i)$. The image of the map $G \times^{P_i} (\overline{\O_i} + \r_i) \to \g$ is contained in $\overline{\Ind_{\l_i}^\g(\O_i)}$, and so by dimension comparison we see that $G \times^{P_i} (\overline{\O_i} + \r_i) \to \overline{\Ind_{\l_i}^\g(\O_i)}$ is generically finite-to-one. There is an open subset of $\overline{\Ind_{\l_i}^\g(\O_i)}$ over which the fibres are finite and, since the collapsing map of the fibre bundle is $G$-equivariant, we see that the fibre is finite over $e \in \Ind_{\l_i}^\g(\O_i)$. Thus $Y\to e+\g^f$ is quasi-finite.

We have shown that $Y \to e+\g^f$ is quasi-finite and proper, hence finite.

It remains to show that $Y \to \Spec S(\g,e)^\ab$ is dominant. Let $\Sc_1,...,\Sc_k$ be the sheets of $\g$ containing $\O$. By Proposition~\ref{P:Katsyloproperties}(1) we know that $\Spec S(\g,e)^\ab = (e+\g^f) \cap \bigcup_i \Sc_i$. Since the sheets are classified by the rigid induction data (Theorem~\ref{T:Borhostheorem}) and the $\{(\l_i, \O_i) \mid i=1,...,s\}$ denote all induction data for $G\cdot e$, we might as well assume that $(\l_i, \O_i)$ is the induction datum corresponding to $\Sc_i$ for $i=1,...,k$. By the remarks of the previous paragraph we know that the image of $G \times^{P_i} (\z(\l_i) \times X_i \times \r_i) \to \g$ is dense in $\Sc_i$ and so $Y \to \Spec S(\g,e)^\ab$ is dominant. This completes the proof.
\hfill\qed
\end{proofoforbitmethod}

{\tt Email:} \\ { lt803@bath.ac.uk}\\
\noindent {\tt Address:} \\ Department of Mathematical Sciences, University of Bath, North Road, Claverton Down, Bath BA2 7AY, United Kingdom.




\end{document}